\documentclass[hidelinks,11pt]{article}

\newcommand{\bol}{\boldsymbol}

\newcommand{\ney}{\boldsymbol{y}}                          
\newcommand{\nex}{\boldsymbol{x}}           
\newcommand{\bnex}{\bold{x}}

\newcommand{\ner}{\boldsymbol{r}}

\newcommand{\de}{\,\mathrm{d}}                               
\newcommand{\e}{\operatorname{e}}

\newcommand{\andtext}{\quad\mbox{and}\quad}

\newcommand{\p}{\partial}

\newcommand{\real}{\mathrm{Re}\,}    
                                   
\newcommand{\imag}{\mathrm{Im}\,}

\newcommand{\lf}{\left}
\newcommand{\rg}{\right}

\newcommand{\R}{\mathbb{R}}       
\newcommand{\C}{\mathbb{C}}

\newcommand{\bxi}{\boldsymbol\xi}  
\newcommand{\bnor}{\bold n}          
\newcommand{\nor}{\boldsymbol n}

\usepackage{multirow}
\usepackage{amsmath}
\usepackage{amsfonts}
\usepackage{amsmath}
\usepackage{amssymb}  
\usepackage{graphicx}
\usepackage{caption}
\usepackage{enumitem}
\usepackage{mathrsfs}
\usepackage{upgreek}
\usepackage{amsthm}
\usepackage{subfig}
\usepackage{booktabs}
\usepackage{authblk}
\usepackage[noabbrev]{cleveref}
\crefformat{equation}{(#2#1#3)}
\usepackage[]{algorithm2e}

\newtheorem{theorem}{Theorem}[section]
\newtheorem{lemma}[theorem]{Lemma}

\newtheorem{remark}[theorem]{Remark}

\usepackage{xcolor}
\usepackage[normalem]{ulem}

\title{Harmonic density interpolation methods for high-order evaluation of \\ Laplace layer potentials in 2D and 3D}
\author[1,2]{Carlos P\'erez-Arancibia\thanks{cperezar@mit.edu}}
\author[1]{Luiz M. Faria\thanks{lfaria@mit.edu}}
\affil[1]{\small{Department of Mathematics, Massachusetts Institute of Technology}}
\affil[2]{\small{Institute for Mathematical and Computational Engineering, School of Engineering and Faculty of Mathematics, Pontificia Universidad Cat\'olica de Chile}}
\author[3]{Catalin Turc\thanks{catalin.c.turc@njit.edu}}
\affil[3]{\small{Department of Mathematical Sciences, New Jersey Institute of Technology}}
\date{\today}

\begin{document}

\maketitle

\begin{abstract}
We present  an effective harmonic density interpolation method for the numerical evaluation of singular and nearly singular Laplace boundary integral operators and layer potentials in two and three spatial dimensions. The method relies on the use of Green's third identity and local Taylor-like interpolations of density functions in terms of harmonic polynomials. The proposed technique effectively regularizes the singularities present in boundary integral operators and layer potentials, and recasts the latter in terms of integrands that are  bounded or even more regular, depending on the order of   the density interpolation. The resulting boundary integrals can then be easily, accurately, and inexpensively evaluated by means of  standard quadrature rules.  A variety of numerical examples  demonstrate the effectiveness of the technique when used in conjunction with the classical trapezoidal rule (to integrate over smooth curves) in two-dimensions, and with a Chebyshev-type quadrature rule (to integrate over surfaces  given as unions of non-overlapping  quadrilateral patches) in three-dimensions. 
\end{abstract}
   

\section{Introduction}
The numerical solution of linear constant coefficient partial differential equations (PDEs) by boundary integral equation (BIE) methods provide several advantages
over methods based on volumetric discretizations such as finite difference and finite element methods. 
Indeed, BIE methods can easily handle unbounded domains as they are based on the discretization of the relevant physical
boundaries, giving rise to well-conditioned linear systems of reduced
dimensionality. Although dense, these linear systems can be efficiently solved by means of 
iterative  solvers in conjunction with fast methods such as the Fast Multipole Method~\cite{greengard1987fast}, $\mathcal{H}$-matrix compression~\cite{hackbusch1989fast}, or FFTs based on equivalent-sources methods~\cite{Bruno:2001ima,phillips1997precorrected}. 

One of the main attractive features of BIE discretizations is their ability to deliver high-order convergence. The latter is typically achieved by specialized quadratures carefully designed for the resolution of the singular behavior of the kernels of the operators that enter the BIE. As is known BIE rely on layer potentials defined in terms of the Green function of the underlying PDE and its derivatives. The boundary values of these layer potentials, in turn, give rise to certain boundary integral operators (BIOs) which enter the BIE. For instance, in the case of the Laplace equation under consideration, the  boundary values of the single- and double-layer potentials and their normal derivatives give rise to four BIOs; single-layer, double-layer, adjoint double-layer, and hyper-singular operators. These four BIOs constitute the building blocks of all possible BIE formulations of Laplace equation.  The kernels of the four Laplace BIOs feature the Green function of the Laplace equation and its various normal derivatives, and thus they exhibit different singular behaviors. Indeed, in the case of regular boundaries, (1) the single-layer operators feature weakly singular (integrable) kernels in both two and three dimensions, (2) the double-layer operators feature regular kernels in two dimensions, and weakly singular kernels in three dimensions, and (3) the hyper-singular operators in both two and three dimensions feature boundary integrals that have to be understood in the sense of Hadamard finite-parts integrals. In addition, the evaluation of layer potentials near (but not on) boundaries is faced with the significant challenge of resolving nearly singular kernels. 

\subsection{Previous work}
{Significant efforts have been directed toward the development of high-order quadrature rules for the numerical evaluation of singular and nearly singular boundary integral operators over the last decades. Most of these quadratures are geared toward two dimensional applications and rely on analytical resolution of singularities. As such, the quadratures become more involved as the kernels become more singular.  A comprehensive methodology for high-order evaluations of BIOs on smooth (two-dimensional) curves, which relies on global trigonometric interpolation, logarithmic kernel singularity splitting and resolution of logarithmic singularities, was put forth by Martensen~\cite{MARTENSEN:1963} and Kussmaul~\cite{KUSSMAUL:1969}\footnote{A concise presentation of the Martensen-Kussmaul quadrature rule can be found in~\cite[Chapter 2]{COLTON:2012}}. This method  was subsequently extended by Kress to deliver high-order Nystr\"om discretizations of boundary integral equations posed on curves with corners~\cite{Kress:1990vm} and to boundary integral equations involving  singular kernels that arise from a regularization of the  hypersingular operator~\cite{Kress:1995}. Other high-order, yet not spectrally accurate, methods---some of which are applicable to a more general class  curves in the plane---have been presented in numerous contributions~\cite{alpert1999hybrid,Kapur:1997fi,Helsing:2008gm,Kolm:2001bt,Mayo:1985he}. However, the treatment of nearly singular operators in the context of the methods mentioned above requires post-processing steps such as oversampling of densities and interpolation/extrapolation procedures.

 Motivated in part by fluids applications involving simulation of Stokes flow in multiply connected domains, recent efforts targeted the development of high-order quadratures that can evaluate seamlessly both singular and nearly singular Laplace kernel interactions. Notably, Quadrature by Expansion (QBX) methods have successfully bypassed the need to directly deal with singular integrands~\cite{klockner2013quadrature} by relying on kernel-specific expansions (based on addition theorems for Green's functions) of the layer potentials around points (centers) near the boundary. It is also worth mentioning the second-order accurate scheme of Beale et al.~\cite{beale2001method} for evaluation of  nearly singular integrals in 2D, which is based on a certain mollification and asymptotic correction of the integral kernel. Substantially more accurate and sophisticated numerical procedures to deal with nearly singular integrals were developed by Helsing et al.~\cite{Helsing:2008tn} (which was later improved in~\cite{Barnett:2015kg}) and by Barnett~\cite{barnett2015spectrally}, the latter of which provided the basis for QBX methods~\cite{klockner2013quadrature}. However, besides the fact that some of the more recent  proposed approaches for singular and nearly singular integrals do not easily generalize to three-dimensions (3D), we believe it is fair to claim that, given their level of sophistication, these methods are  arguably difficult to implement. Indeed, the high accuracy achieved by these methods is a result of a judicious combination of involved techniques encompassing  grid oversampling of the density,  kernel approximations and/or expansions, and high-order polynomial interpolation. In addition, and perhaps more importantly, all the aforementioned methods require a careful selection of  several  parameters in order to garner their optimal performance. The selection of the parameters needed in 2D QBX (e.g. location of off-surface centers, order of expansions in the kernel representation, oversampling factors), for instance, can be streamlined by resorting to local kernel expansions in a version of QBX referred to Adaptive Quadrature by Expansion (AQBX)~\cite{klinteberg2017adaptive}.}
	
There are significantly fewer high-order methods for the discretization 3D Laplace BIEs with respect to their 2D counterparts. This fact can be explained by the additional challenges 	that pertain to 3D, amongst which we mention (a) the more severe nature of the singularities of the Green functions and their derivatives, (b) the technical difficulties associated with the reliable generation of high-order surface representations, and (c) the need to employ patches in order to describe surfaces as well as local approximations of densities. First, we remark that the singularities of BIOs---in particular that of the hyper-singular one---are more challenging for Nystr\"om discretizations than for Galerkin (BEM) discretizations, as weak formulations of Laplace BIEs require (double) integration of weak singularities only. We also note that with regards to the latter aspect (c), the use of non-overlapping patches (which is the most natural in our opinion) inherently gives rise to nearly singular integrands corresponding to the scenario when the target and integration points are nearby yet they belong to adjacent but different patches. The resolution of near singularities, encountered mostly in Galerkin discretizations, can be effected by double exponential (e.g., $\operatorname{sinh}$, $\operatorname{tanh}$) changes of variables~\cite{takahasi1974double} or using the singular quadratures of Sauter and Schwab~\cite{sauter2010boundary}. Singularity subtraction techniques are certainly the most common
approach for evaluating singular integrals arising from BIEs and,
therefore, numerous low-order variants of these techniques are available in the
literature not only for the Laplace equation but also for the
Helmholtz, Stokes flow and elastostatic
equations. A notable exception is the high-order version of the singularity subtraction approach proposed in~\cite{2013arXiv1301.7276H} for Nystr\"om discretizations of BIOs on tori. Regardless of discretization type, most 3D high-order quadratures of BIOs are case specific, that is, they are designed to specifically treat a certain type of kernel singularity or near-singularity, and become increasingly involved as the kernel singularity becomes more severe.  Owing to these difficulties, the implementation of high-order numerical methods for evaluation 3D Laplace BIEs is complex, arguably leading to a dearth of available open-source software (BEMpp~\cite{smigaj2015solving}, \texttt{https://bempp.com}, being a notable exception) and ultimately limiting wider use in the engineering community. Without being exhaustive, in what follows we review several high-order Nystr\"om numerical schemes for solution of 3D Laplace BIEs.

A spectrally accurate Nystr\"om method for 3D BIEs in the spirit of Martensen and Kussmaul was developed by Ganesh et al.~\cite{Ganesh:2004cw} for surfaces diffeomorphic with spheres. This approach makes use of the diffeomorphism between the surface of integration and the unit sphere, global interpolation of densities in terms of spherical harmonics, as well as addition theorems that lead to analytic resolution of weak kernel singularities. Stronger kernel singularities can be handled via integration by parts techniques. We mention that variants of this approach have been used in Stokes flow simulations~\cite{Veerapaneni:2011kj}. A more general high-order Nystr\"om method for the solution of BIEs on smooth surfaces is the (provably convergent) method of Bruno at al.~\cite{bruno2013convergence,Bruno:2001ima,bruno2001surface}. The essential ingredients of this method consist of (1)  use of an explicit atlas of overlapping patches and associated partition of unity functions and (2) use of polar change of variables to analytically resolve weak kernel singularities. The method is quite general, as it was demonstrated in~\cite{bruno2007accurate,Ying:2004dw} that surfaces of engineering relevance can be accurately approximated by means of smooth overlapping-patch manifold representations. The change to polar coordinates approach was extended and refined in~\cite{Bremer2012ANM} via additional affine mappings and precomputed quadrature tables, to treat cases when the surface parametrizations are highly non-conformal. This approach was extended to deal with singular integral operators in subsequent contributions~\cite{Bruno:2012dx,Bruno:2013if,Ying:2006gy}. In turn, several post processing techniques  have been proposed in the literature to deal with nearly singular-integrals in 3D. We mention first the ``extraction" technique of Schwab et al.~\cite{Schwab:bj} which, using Taylor expansions along the unit normal to the surface, allows for approximations of layer potentials near and on the surface. The higher-order terms in the Taylor expansions, however,  involve high-order normal derivatives which incur the significant overhead of evaluating several strongly singular boundary integral operators.  Yet another procedure to deal with nearly singular integrands in 3D was presented in~\cite{Ying:2006gy}. Unlike~\cite{Schwab:bj}, this procedure makes used of both on- and off-surface values (at points sufficiently far away from the surface with respect to the mesh size) to interpolate layer potentials near the boundary. A  simple method based on a kernel regularization that can achieve high-order evaluations of both singular and nearly singular Laplace boundary integrals  was put forth by Beale in~\cite{beale2004grid} and then  further developed in~\cite{2016CCoPh..20..733B} and even extended to Stokes flow equations in~\cite{tlupova2013nearly}. This kernel regularization  delivers third-order methods for on-surface as well as near-surface evaluations of the single and double-layer Laplace BIOs. One of the main advantages of  Beale's kernel regularization method is that it does not require a surface parametrization nor a surface triangulations. Finally, we mention the very recently introduced extensions of QBX methods to evaluations of 3D boundary integral operators/ layer potentials~\cite{afKlinteberg:2016ed,siegel2017local}. Just like the kernel regularization methods~\cite{2016CCoPh..20..733B}, QBX methods can seamlessly treat both weakly singular and nearly singular boundary integral operators/layer potentials in 3D; however, evaluations of adjoint double-layer and hyper-singular boundary integral operators or their Stokes counterparts with commensurate orders of singularity are not presented in~\cite{afKlinteberg:2016ed,siegel2017local}.
 
 \subsection{Scope of this contribution}
 This paper presents a novel harmonic density interpolation method that regularizes the kernel singularities of the four BIOs associated with the Laplace equation in smooth two- and three-dimensional  domains. Unlike singularity subtraction techniques, our method achieves regularization of the integral kernels by targeting the densities of the BIOs instead of the kernels themselves. Low-order precursors of regularization techniques that target the integral densities are available in the
literature for BIOs associated with various linear constant coefficient PDEs~\cite{Klaseboer:2012ks,liu1991some,liu1999new,perez2018plane,sun2014robust}. In certain aspects the proposed method is a generalization of the third-order regularization technique  introduced in~\cite{perez2018plane} for the 2D Helmholtz combined field integral operators. The main contribution of this paper consists of developing a 
high-order local Taylor-like interpolation methodology based on harmonic polynomials that takes full advantage of the boundary regularity of the density functions that enter Laplace BIOs. Our approach is universal, in the sense that is applicable to evaluations of all four Laplace BIOs irrespective of the singularity of their associated kernels.  In detail, in order to evaluate the single-layer operator in $\R^d$, $d=2,3$,
\begin{equation*}
S[\varphi](\nex)=\int_{\Gamma}G(\nex,\ney)\varphi(\ney)\de s(\ney),\quad \nex\in\Gamma\subset\R^d,
\end{equation*} 
for example, where $G$ is the free-space Laplace Green function and $\varphi$ is a regular density function defined on a closed boundary $\Gamma$, we make use of a family of smooth harmonic functions $U:\R^d\times\Gamma\to \R$ and Green's third identities to recast $S$ in the form
\begin{equation*}
S[\varphi](\nex)=\frac{1}{2}U(\nex,\nex)+\int_{\Gamma}\frac{\partial G(\nex,\ney)}{\partial n(\ney)}U(\ney,\nex)\de s(\ney)+\int_\Gamma G(\nex,\ney)\lf[\varphi(\ney)-\partial_n^{\ney} U(\ney,\nex)\rg]\de s(\ney).
\end{equation*} 
We  require then $U$ to be such that both expressions $U(\ney,\nex)$ and $\varphi(\ney)-\partial_n^{\ney} U(\ney,\nex)$, where $\partial_n^{\ney}$ denotes the normal derivative with respect to $\ney$, vanish to prescribed orders as $\ney\to\nex$. The key insight of our method is to seek families of functions $U:\R^d\times\Gamma\to \R$ of the form
\begin{equation*}  
  U(\ner,\nex) = \sum_{j=0}^J {c}_j(\nex){H}_j(\ner - \nex),\qquad\ner\in\R^d,\nex\in\Gamma,
\end{equation*}
where $H_j$ are homogeneous harmonic polynomials of order $j$. Thus, the requirement that the expressions $U(\ney,\nex)$ and $\varphi(\ney)-\partial_n^{\ney} U(\ney,\nex)$ vanish simultaneously to high-order as $\ney\to\nex, \nex\in\Gamma, \ney\in\Gamma,$ is equivalent to solving for each point $\nex\in\Gamma$ a local Taylor-like interpolation problem of prescribed order involving the density function $\varphi$ and Dirichlet and Neumann boundary values of harmonic polynomials in the above definition of the functions $U$.  In 2D, such families of functions $U$ can be easily obtained to arbitrarily high-order by means of complex variable techniques. In 3D, in turn, their construction is significantly more involved, yet it is relatively easy to produce  third-order interpolating functions $U$.  The proposed technique extends naturally to deal with nearly singular integrands, allowing layer
potentials and their gradients to be evaluated at target points
arbitrarily close to the boundaries without compromising numerical accuracy. The high-order harmonic  density interpolation method renders the Laplace BIOs directly amenable to evaluations  by standard, readily implementable quadrature rules (e.g. trapezoidal rule in 2D, and Fej\'er, Clenshaw-Curtis quadrature in 3D) as the BIOs feature only regular (at least continuous for weakly singular BIOs and bounded for the hypersingular BIO in 3D) integrands. The integration of our method within the FMM framework is fairly straightforward and will be presented in a future contribution.

We illustrate the effectiveness and simplicity of the proposed high-order harmonic density interpolation technique through a variety of 2D and 3D numerical results concerning evaluation of singular and nearly singular Laplace BIOs and layer potentials defined on closed, smooth curves and surfaces. We show, in
particular, that for sufficiently high interpolation orders~$M$
(which achieve integrands that vanish as $|\nex-\ney|^{M+1}$ at the kernel
singularity $\nex=\ney$)  our technique used in conjunction with the simple trapezoidal quadrature leads to discretization errors in the evaluation of the two-dimensional single-layer BIO of the order $O(h^{2M+3})$ for $M$ even, and respectively $O(h^{2M+1})$ for 
$M$ odd, where $h>0$ denotes the grid spacing. In addition, the same methodology leads to evaluations of the 2D hyper-singular BIO that converge exponentially fast. For 3D problems, in turn, we provide the explicit construction of the harmonic polynomial interpolants  amenable to express the single- and double-layer operators in terms of  continuous integrands, and the adjoint double-layer and hypersingular operators in terms of bounded integrands. Relying on a non-overlapping quadrilateral patch manifold representation of 3D surfaces, underlying local Chebyshev grids, and a (spectrally accurate) Fej\'er quadrature rule~\cite{bruno2018chebyshev,turc2011efficient}, we demonstrate through numerical examples that  the proposed harmonic density interpolation technique yields a third-order accurate Nystr\"om method for BIEs involving the Laplace single- and double-layer operators, and a second-order accurate Nystr\"om method BIEs involving the Laplace adjoint-double layer and hypersingular operators. The levels of accuracy achieved by our method for three dimensional problems (e.g., $10^{-5}$) are competitive for engineering applications, especially in the light of the straightforward implementation of the method. A preliminary Matlab implementation of the proposed methodology for 3D problems is available at: 
\begin{center}
\texttt{https://github.com/caperezar/HDIMethod}.
\end{center}

We also mention that the proposed harmonic density interpolation  can easily accommodate evaluations of nearly singular boundary integral operators. Extensions of our method to treatment of singular and nearly singular BIOs arising in acoustics, electromagnetics, and elastodynamics are currently being pursued; for those problems the local Taylor-like interpolation problems will rely on plane-waves rather than harmonic polynomials.

The structure of this paper is as follows. \Cref{sec:prelim,eq:hif,sec:eval-layer-potent} provide a comprehensive description of the proposed technique for 2D BIEs. The details on the construction of the harmonic expansion functions using complex variable techniques are given in \Cref{eq:hif}. \Cref{sec:eval-layer-potent}  presents the extension of the technique for the evaluation of layer potentials at target points near the boundary. The proposed technique for corresponding 3D problems is described in \Cref{sec:3D}. \Cref{sec:numerics}, finally, presents a variety of numerical examples in two and three spatial dimensions, including comparisons with Beale~et~al.~\cite{beale2001method,ying2013fast} and QBX~\cite{klockner2013quadrature} methods in 2D.

\section{High-order kernel singularity regularization for 2D problems}\label{sec:prelim}

The single-layer ($S$), double-layer ($K$), adjoint double-layer ($K'$) and hypersingular ($N$) operators of Calder\'on calculus associated with the Laplace equation in $\R^2$ are given by 
\begin{subequations}\begin{align}
S[\varphi](\nex) :=&\ \int_{\Gamma} G(\nex,\ney)\varphi(\ney)\de s(\ney), \\ 
K'[\varphi](\nex) :=&\ \int_{\Gamma} \frac{\p G(\nex,\ney)}{\p\nor(\nex)}\varphi(\ney)\de s(\ney),\label{eq:singe}\\
K[\varphi](\nex) :=&\ \int_{\Gamma} \frac{\p G(\nex,\ney)}{\p\nor(\ney)}\varphi(\ney)\de s(\ney),\\ 
N[\varphi](\nex) :=&\ \mathrm{f.p.}\int_{\Gamma} \frac{\p^2 G(\nex,\ney)}{\p\nor(\nex)\p\nor(\ney)}\varphi(\ney)\de s(\ney),\label{eq:hypersingular}
\end{align}\label{eq:int_op}\end{subequations}
for  $\nex\in \Gamma$, where
\begin{equation}
    G(\nex,\ney):=-\frac{1}{2\pi}\log|\nex-\ney| \label{eq:GF}
\end{equation}
is the free-space Green function for the Laplace equation in $\R^2$,
and f.p. in \cref{eq:hypersingular} stands
for the Hadamard finite-part integral.  The curve $\Gamma$ in 2D is assumed
to be a simple closed curve that admits a regular periodic counterclockwise  parametrization
\begin{equation}\Gamma=\{\bold x(t)=\lf({\rm x}_1(t),{\rm x}_2(t)\rg): t\in[0,2\pi)\}\label{eq:param}.\end{equation}
In the parameter space the four integral operators defined in~\cref{eq:int_op} will be denoted by 
$\tilde S[\phi]$, $\tilde K[\phi]$, $\tilde K'[\phi]$ and $\tilde N[\phi]$, where  $\phi(t)=\varphi(\bnex(t)):[0,2\pi)\to\R$. For presentation simplicity throughout this paper both the density function~$\phi$ and the curve parametrization $\bnex$ are assumed to be analytic functions of $t\in[0,2\pi]$.  Note that under the assumptions introduced above, and in the particular case of the Laplace equation in two spatial dimensions, the integral kernels in both double-layer ($K$) and adjoint double-layer ($K'$) operators are smooth functions (see~\Cref{lem:cond_1,lem:cond_2}), so their high-order numerical evaluation can be achieved by application of standard quadrature rules (e.g. trapezoidal rule).  A more general analysis can be carried out to significantly relax the strong analytic regularity assumptions made on $\Gamma$ and $\phi$ (see Remark~\ref{rem:analy} below).

\subsection{Hypersingular operator}\label{sec:singulartiy_sub}
Consider first the hypersingular  operator $ N$, defined in \cref{eq:hypersingular}, applied to a smooth density function $\varphi:\Gamma\to\R$, and evaluated at a point $\nex=\bnex(t) \in \Gamma$. The   proposed technique relies on introducing a known  smooth function $U: \R^2\times\Gamma \to \R$, with Dirichlet and Neumann traces denoted by $P(\tau,t) = {U}(\mathbf{x}(\tau),\bnex(t))$ and $Q(\tau,t) = \nabla  U(\mathbf{x}(\tau),\bnex(t))\cdot\bnor(\tau)$ for $\tau,t\in[0,2\pi]$, respectively, such that
\begin{align}
  \label{eq:2}
  \tilde{N}[\phi](t) = \tilde{N}[\phi(\cdot) - P(\cdot,t)](t) + \tilde{N}[P(\cdot,t)](t),
\end{align}
can be evaluated by integrating regular (non-singular) functions only. Given the
strong $O((s-\tau)^{-2})$ singularity of the hypersingular operator
kernel, it is thus necessary that $\phi(\tau) -P(\tau,t)$ vanishes to high
order as $\tau \to t$. Additionally, we need
$\tilde{N}[P(\cdot,t)](t)$ to admit a representation in terms of
smooth integrands as well. The key idea of our method is to choose $ U$ to be harmonic in the first variable, so that by Green's third
identity~\cite{martin2006multiple} we have
\begin{align}
\tilde N [P(\cdot,t)](t) &=-\frac{Q(t,t)}{2} + \tilde K'\left[Q(\cdot,t)\right](t).
\end{align}
Therefore, we  can then express the hypersingular operator as
\begin{align}
  \label{eq:3}
\tilde N[\phi](t) &=-\frac{Q(t,t)}{2} + \tilde N [\phi-P(\cdot,t)](t) + \tilde K'\lf[Q(\cdot,t)\rg](t),
\end{align}
where the integrals on the right-hand-side involve functions with smoothness that can be controlled by an appropriate choice of $U$, i.e., by requiring  $\phi(\tau)-P(\tau,t)$ to vanish to high enough order as $\tau\to t$. 
In detail, for a prescribed {\emph{density interpolation order} $M\geq 0$, we require $P$ to satisfy the following  condition
\begin{equation}
P(\tau,t) = \phi(\tau) + O\lf(|\tau-t|^{M+1}\rg)\quad\mbox{as}\quad \tau\to t. \label{eq:cond_1}
\end{equation}
We also require that the constants in the ``big-$O$" notation in equation~\eqref{eq:cond_1} be bounded uniformly in $t$. Relying on the smoothness of both $\phi$ and $P(\cdot,t)$ it easy to show (using Taylor's theorem) that a sufficient condition for~\cref{eq:cond_1}  to hold is that $P$ satisfies
\begin{align}
  \label{eq:curve-condition-1}
  \lim_{\tau \to t}  \frac{\p^m}{\p\tau^m} \left\{P(\tau,t) - \phi(\tau) \right\} = 0 \quad \mbox{for} \quad m=0,\ldots, M,
\end{align}
or equivalently, that the $m$-th Taylor expansion coefficient of $P(\cdot,t)$ equals the $m$-th order derivatives of $\phi$ at $\tau=t$. Note that on account of the smoothness of the kernel of the operator $K'$, we do not impose any vanishing conditions on the expressions $Q(\tau,t)$ as $\tau\to t$. Remarkably, and as we will explain in the next sections, our construction of the harmonic functions $U$ that meet the requirement~\eqref{eq:curve-condition-1} achieves as a byproduct a vanishing condition of order $M$ on $Q(\tau,t)$.

The following lemma shows that choosing $M\geq 1$ in the conditions~\cref{eq:curve-condition-1} for~$P$, the hypersingular operator can be  expressed in terms of integrals of (smooth) $2\pi$-periodic  analytic functions.
\begin{lemma}\label{lem:cond_1} Let $\Gamma\subset\R^2$ be a closed simple analytic curve and  $\phi$ be a real $2\pi$-periodic analytic function. Suppose  there exists $U_N:\R^2\times\Gamma\to \R$, $U_N(\cdot,\bnex(t))$ harmonic in $\R^2$ for all $t\in[0,2\pi]$, such that its Dirichlet trace $P_N(\cdot,t) = U_N(\bnex(\cdot),\bnex(t)):[0,2\pi]\to\R$ satisfies \cref{eq:curve-condition-1} for $M\geq 1$. Then, letting  $Q_N(\cdot,t) =\nabla U_N(\cdot,\bnex(t))\cdot \bnor(\cdot)$ denote the Neumann trace of $U_N$,  the hypersingular operator can be expressed 
as
\begin{equation}
\tilde N[\phi](t) =-\frac{Q_N(t,t)}{2}+ \int_{0}^{2\pi}\lf\{R^{(P)}_N(\tau,t)+R^{(Q)}_N(\tau,t)\rg\}|\bnex'(\tau)|\de\tau,\label{eq:hyper_rep}
\end{equation}
where 
\begin{equation}\label{eq:NP}
R^{(P)}_N(\tau,t) :=\lf\{\begin{array}{crr}\displaystyle\lf(\frac{\bnor(\tau)\cdot\bnor(t)}{2\pi|\bnex(t)-\bnex(\tau)|^2}-\frac{\{\bnex(t)-\bnex(\tau)\}\cdot \bnor(t)\ \{\bnex(t)-\bnex(\tau)\}\cdot \bnor(\tau)}{\pi |\bnex(t)-\bnex(\tau)|^4}\rg)&\medskip\\
\hfill\times\lf\{\phi(\tau)-P_N(\tau,t)\rg\}\quad\mbox{if}\ \tau\neq t,\bigskip\\
\displaystyle\frac{1}{2\pi|\bnex'(t)|^2}\frac{\p^2}{\p \tau^2}\lf\{\phi(\tau)-P_N(\tau,t)\rg\}\Big|_{\tau=t}\hfill\mbox{if}\ \tau=t,
\end{array}\rg.
\end{equation}
and 
\begin{equation}\label{eq:NQ}
 R^{(Q)}_N(\tau,t) :=\lf\{\begin{array}{rrr}\displaystyle-\frac{1}{2\pi}\frac{\{\bnex(t)-\bnex(\tau)\}\cdot \bnor(t)}{ |\bnex(t)-\bnex(\tau)|^2} Q_N(\tau,t)&\mbox{if}&\tau\neq t,\medskip\\
 \displaystyle\frac{\bnex''(t)\cdot \bnor(t)}{4\pi |\bnex'(t)|^2}Q_N(t,t) &\mbox{if}&\tau=t,
 \end{array}\rg.
\end{equation}
are real analytic $2\pi$-periodic functions of $\tau\in[0,2\pi]$ for all $t\in[0,2\pi]$.
\end{lemma}
\begin{proof} The explicit expressions for $R^{(P)}_N$ in~\cref{eq:NP} and
  $R^{(Q)}_N$ in~\cref{eq:NQ} when $\tau\neq t$ are obtained  from
  identity~\cref{eq:3}. Their  values at $\tau=t$, in turn, follow
  from a direct application of L'Hospital rule.

  In order to prove the analyticity of $R^{(P)}_N(\cdot,t)$ and
  $R^{(Q)}_N(\cdot,t)$ for $t\in[0,2\pi]$, we first note that since
  the curve parametrization $\bnex:[0,2\pi]\to \Gamma$ is analytic and
  $2\pi$-periodic, the following expressions hold:
  $|\bnex(t)-\bnex(\tau)|^2 =\sin^2(\frac{t-\tau}{2})a(\tau,t)$,
  $\{\bnex(t)-\bnex(\tau)\}\cdot \bnor(t) = \sin^2(\frac{t-\tau}{2})b(\tau,t)$
  and
  $\{\bnex(t)-\bnex(\tau)\}\cdot \bnor(\tau) =
  \sin^2(\frac{t-\tau}{2})c(\tau,t)$, where $a$, $b$ and $c$ are analytic
  $2\pi$-periodic functions. Furthermore, since the curve
  parametrization is regular we have $a(t,t) =|\bnex'(t)|^2\neq 0$. By
  the interpolation conditions~\cref{eq:curve-condition-1} and the
  fact that $\phi$ and $P_N(\cdot,t)$ are analytic $2\pi$-periodic
  functions (due to the fact that $P_N(\tau,t)$ is the trace on
  $\Gamma$ of the harmonic function $U_N(\cdot,\bnex(t))$, which is
  analytic in all of $\R^2$), on the other hand, it follows that
  $\phi(\tau)-P_N(\tau,t)=\sin^{M+1}(\frac{t-\tau}{2})d(\tau,t)$ with
  $d(\cdot,t)$ being an analytic $2\pi$-periodic function. Therefore,
  since $M\geq 1$ by hypothesis of the lemma, we conclude
  that~\cref{eq:NP} and~\cref{eq:NQ} are analytic $2\pi$-periodic
  functions of $\tau$ with a removable singularity at $\tau=t$. The
  proof is now complete.
\end{proof}

\subsection{Single-layer operator}
A calculation similar to the one shown above for the hypersingular
operator yields that $\tilde{S}$ can be expressed as 
\begin{equation}\label{eq:SL_rep}
\tilde S[\phi](t) =\frac{P(t,t)}{2} + \tilde K [P(\cdot,t)](t) + \tilde S\lf[\phi-Q(\cdot,t)\rg](t),
\end{equation}
for $t\in[0,2\pi]$. Hence, in order for the right-hand-side
of~\cref{eq:SL_rep} to be given in terms of integrals of smooth
functions, we need to find a harmonic function $U$ with Neumann trace $Q$ satisfying 
\begin{align}
  \label{eq:curve-condition-2}
  \lim_{\tau \to t}  \frac{\p^m}{\p\tau^m} \left\{Q(\tau,t) - \phi(\tau) \right\}  = 0 \quad \mbox{for} \quad m=0,\ldots, M.
\end{align}
Again, we do not formally impose any  vanishing conditions on the expressions $P(\tau,t)$ as $\tau\to t$, yet our construction of the harmonic functions $U$ will automatically realize such conditions of order $M+2$. The following lemma shows  that the single-layer operator can in fact be expressed in terms of  smooth $C^{M}$ integrands provided  the interpolation conditions~\cref{eq:curve-condition-2} are satisfied.

\begin{lemma}\label{lem:cond_2}Let $\Gamma\subset\R^2$ be a closed simple analytic curve and $\phi$ be a real $2\pi$-periodic analytic function. Suppose there exists $U_S:\R^2\times\Gamma\to \R$, $U_S(\cdot,\bnex(t))$ harmonic in $\R^2$ for all $t\in[0,2\pi]$, such that its Neumann trace $Q_S(\cdot,t) =\nabla U_S(\cdot,\bnex(t))\cdot \bnor(\cdot):[0,2\pi]\to\R$ satisfies \cref{eq:curve-condition-2} for $M\geq 0$. Then, letting  $P_S(\cdot,t) = U_S(\bnex(\cdot),\bnex(t))$ denote the Dirichlet trace of $U_S$,  the single-layer operator can be expressed 
as
\begin{align}
  \label{eq:single-layer-removed-singularity}
\tilde S[\phi](t) =\frac{P_S(t,t)}{2}+ \int_{0}^{2\pi}\lf\{R^{(P)}_S(\tau,t)+R^{(Q)}_S(\tau,t)\rg\}|\bnex'(\tau)|\de\tau,  
\end{align}
where 
\begin{equation}\label{eq:SP}
 R^{(P)}_S(\tau,t) :=\lf\{\begin{array}{ccc}\displaystyle\frac{1}{2\pi}\frac{\{\bnex(t)-\bnex(\tau)\}\cdot \bnor(\tau)}{ |\bnex(t)-\bnex(\tau)|^2} P_S(\tau,t)&\mbox{if}&\tau\neq t,\medskip\\
 \displaystyle -\frac{\bnex''(t)\cdot \bnor(t)}{4\pi |\bnex'(t)|^2}P_S(t,t) &\mbox{if}&\tau=t,
 \end{array}\rg.
\end{equation}
is a  analytic $2\pi$-periodic  function of $\tau\in[0,2\pi]$, and 
\begin{equation}\label{eq:SQ}
 R^{(Q)}_S(\tau,t) :=	\lf\{\begin{array}{ccc}\displaystyle-\frac{1}{2\pi}\log\lf(|\bnex(t)-\bnex(\tau)|\rg) \lf\{\phi(\tau)-Q_S(\tau,t)\rg\}&\mbox{if}&\tau\neq t,\medskip\\
0 &\mbox{if}&\tau=t,
 \end{array}\rg.    
\end{equation}
is a $M$-times continuously differentiable $2\pi$-periodic function of $\tau\in[0,2\pi]$,  for all $t\in[0,2\pi]$. Furthermore, the $(M+1)$-th derivative of $R_S^{(Q)}(\cdot,t)$ is an integrable function for all $t\in[0,2\pi]$
\end{lemma}
\begin{proof}
  The explicit expressions for the integrands in~\cref{eq:SP}
  and~\cref{eq:SQ} are obtained directly from
  identity~\cref{eq:SL_rep}. The proof that $R^{(P)}_S(\cdot,t)$ is an
  analytic $2\pi$-periodic function, on the other hand, is essentially
  the same as in the case of $R^{(Q)}_N(\cdot,t)$ in~\Cref{lem:cond_1}, so it is omitted here.
  
  Since both $\phi$ and $Q_S(\cdot,t)$ are  analytic $2\pi$-periodic
  functions we have that $ \phi-Q_S(\cdot,t)$ can be expressed as $
  \phi(\tau)-Q_S(\tau,t)=\sin^{M+1}(\frac{t-\tau}{2})a(\tau,t)$ with
  $a(\cdot,t)$   being  analytic and $2\pi$-periodic. Therefore,
  adding and subtracting the function
  $\frac{1}{4\pi}\sin^{M+1}(\frac{t-\tau}{2})a(\tau,t)\log(4\sin^2(\frac{t-\tau}{2}))$
  to $R^{(Q)}_S(\tau,t)$ in~\cref{eq:SQ} we obtain
  $R^{(Q)}_S(\tau,t) =-\frac{1}{4\pi}
  \sin^{M+1}(\frac{t-\tau}{2})a(\tau,t)\log(4\sin^2(\frac{t-\tau}{2})) +
  b(\tau,t)$ with~~$b(\tau,t)=-\frac{1}{4\pi}
  \sin^{M+1}(\frac{t-\tau}{2})a(\tau,t)\log(|\bnex(\tau)-\bnex(t)|^2/4\sin^2(\frac{t-\tau}{2}))
  $ being  analytic and $2\pi$-periodic. Therefore, by the product
  rule we readily obtain that  the  $M$-th order derivative of
  $R^{(Q)}_S(\cdot,t)$ is a $2\pi$-periodic  function that vanishes as
  $O((t-\tau)\log|t-\tau|)$ as $\tau\to t$, so it is indeed
  continuous at $\tau=t$. The $(M+1)$-th order derivative of
  $R^{(Q)}_S(\cdot,t)$, in turn, features a $O(\log|t-\tau|)$
  integrable singularity as $\tau\to t$. The proof is now complete.
\end{proof}

\begin{remark}\label{rem:analy}
The analytic regularity of both the curve $\Gamma$ and the density function  $\phi$ assumed in this section can be significantly relaxed. Indeed, straightforward modifications of Lemmas~\ref{lem:cond_1} and~\ref{lem:cond_2} can be carried out to show that for  $\Gamma$ of class $\mathcal C^{\tilde \kappa}$ and $\phi$ of class $\mathcal C^{\kappa}$, with $\kappa\leq \tilde\kappa$ and $2\leq M\leq  \kappa$, the proposed density interpolation technique yields $2\pi$-periodic $(\kappa-2)$-times continuously differentiable integrands in the cases of the hypersingular, double-layer, and adjoint double-layer operators, and it yields  $2\pi$-periodic $M$-times continuously differentiable integrands in the case of the single-layer operator. 

 It is also worth mentioning that the proposed technique can in principle be extended to piecewise smooth curves in 2D when the (global) parametrization of the curve and the layer potential densities are not differentiable at corners. The lack of smoothness at corners can be circumvented in practice by discretizing the smooth panels using grids that (i) are refined toward corner points and (ii) do not include corner points (e.g. Chebyshev nodes). As demonstrated in the numerical experiments in Section~\ref{3d_num} below, this approach works well in 3D for low-order interpolations order (see also~\cite{perez2018plane} for a related approach for the 2D Helmholtz equation using graded meshes constructed via sigmoid transforms).
\end{remark}
In the next section we provide a procedure to construct harmonic interpolating functions
$U_N$ and $U_S$ satisfying the conditions
in~\Cref{lem:cond_1,lem:cond_2}, respectively, for
any given  density interpolation order $M\geq 0$. 

\section{Harmonic interpolating functions in 2D}\label{eq:hif}
In order to construct the desired harmonic functions $U_N,U_S:\R^2\times\Gamma\to \R$ in~\Cref{lem:cond_1,lem:cond_2}, we consider linear combinations of harmonic polynomials of the form
\begin{equation}  \label{eq:1}
  U(\nex,\ney) = \sum_{j=0}^J \mathbf{C}_j(\ney) \cdot \mathbf{H}_j(\nex - \ney),\qquad\nex=(x_1,x_2)\in\R^2,\ney=(y_1,y_2)\in\Gamma,
\end{equation}
where   $\bold H_j(\nex-\ney)=[\real (z-w)^j,\imag (z-w)^j]^T$, with $z=x_1+ix_2$, $w=y_1+iy_2,$ and  
  $\bold C_j:\Gamma\to \R^2$, for $j=0,\ldots,J$. Instead of working with the real expression~\eqref{eq:1} it turns out
to be convenient to complexify it and consider instead 
$U(\nex,\ney) = \real\{F(z,w)\}$  where
\begin{equation}  \label{eq:powe_series}
F(z,w) = \sum_{j=0}^J \frac{\tilde c_j(w)}{j!}  (z - w)^j,\qquad z\in\C,w\in\gamma,
\end{equation}
with $\gamma$ denoting the curve in the complex plane parametrized by the (analytic) complex valued function
$\zeta(t) = {\rm x_1}(t)+i{\rm x_2(t)}$, $t\in[0,2\pi].$ We here recall that the curve  $\Gamma\subset\R^2$ is parametrized by $\bold x(t)=({\rm x}_1(t),{\rm x}_2(t))$, $t\in[0,2\pi]$. Clearly, (by the Cauchy-Riemann equations) $U(\nex,\ney) = \real\{F(z,w)\}$ is harmonic in $\nex\in\R^2$.
  
  Letting then $w=\zeta(t)$, $z=\zeta(\tau)$, and  calling $f(\tau,t)=F(\zeta(\tau),\zeta(t))$ for $t,\tau\in [0,2\pi]$ in~\eqref{eq:powe_series},  we obtain the expressions
\begin{subequations} \begin{align}
 P(\tau,t) &= \real\{f(\tau,t)\}\label{eq:dir_trace_complex}, \\
 Q(\tau,t) &=\frac{1}{|\zeta'(\tau)|}\imag\lf\{\frac{\p }{\p\tau}f(\tau,t)\rg\},\label{eq:neu_trace_complex}
\end{align}\end{subequations}
for the Dirichlet and Neumann traces of $U$  (in the parameter space), respectively, where we made use of the identity  $\bold n(t) =(\imag\zeta'(t),-\real\zeta'(t))/|\zeta'(t)|$. Furthermore, it is easy to show that 
\begin{subequations}\begin{align}
\frac{\p^m}{\p \tau^m}P(\tau,t) &=   \real\lf\{\frac{\p^m}{\de
  \tau^m}f(\tau,t)\rg\},\\
\frac{\p^m}{\p \tau^m}Q(\tau,t) &=  \imag\lf\{\frac{\p^{m}}{\p \tau^{m}}\lf(\frac{1}{|\zeta'(\tau)|}\frac{\p f}{\p\tau}(\tau,t)\rg)\rg\},
\end{align}\label{eq:n_ders}\end{subequations}
for $m\geq1$, where by the Fa\`a di Bruno formula we have
\begin{equation}\begin{split}
  \frac{\p^m}{\p \tau^m}f(\tau,t)\Big|_{\tau=t} 
  =  \sum_{j=1}^m c_{j}(t) 
\mathbb  B_{m,j}\left(\zeta'(t), \zeta''(t), \dots, \zeta^{(m-j+1)} (t)\right),\label{eq:ders}
\end{split}\end{equation}
with  $c_j(t) = \tilde c_j(\zeta(t))$ and $\mathbb B_{m,j}$, $1\leq m,j\leq J$, denoting the so-called partial or incomplete Bell polynomials;~cf.~\cite[Chapter 13]{aldrovandi2001special}.

In what follows of this section we use the identities~\eqref{eq:n_ders} and~\eqref{eq:ders} to show that given a density interpolation order $M\geq 0$, there exist  harmonic functions $U_N(\nex,\ney)=\real \{F_N(z,w)\}$ and $U_S(\nex,\ney)=\real \{F_S(z,w)\}$ satisfying the interpolating conditions in~\cref{eq:curve-condition-1} and~\cref{eq:curve-condition-2}, where $F_N$ and $F_S$ are functions of the form~\cref{eq:powe_series}. 

\subsection{Interpolating function for the hypersingular operator}
We start by seeking coefficients $c_j^{(N)}$ in the expansion
\begin{equation}\label{eq:f_N}
f_N(\tau,t) =F_N(\zeta(\tau),\zeta(t)) =\sum_{j=0}^J\frac{c^{(N)}_j(t)}{j!}(\zeta(\tau)-\zeta(t))^j,\quad \tau,t\in[0,2\pi],
\end{equation}
so that the interpolation conditions~\cref{eq:curve-condition-1} are satisfied.  
In view of equations~\cref{eq:n_ders}, we then note that  conditions~\cref{eq:curve-condition-1} can be enforced on $P_N(\tau,t)=\real\{f_N(\tau,t)\}$ by requiring $f_N$ to satisfy 
\begin{equation}\label{eq:complex_point_cond_1}
\lim_{\tau \to t}  \frac{\p^m}{\p \tau^m}\left[f_N(\tau,t) - \phi(\tau) \right]= 0 \quad\mbox{for}\quad m=0,\ldots,M.
\end{equation}
Clearly, the $m=0$ condition in~\cref{eq:complex_point_cond_1} implies that $c^{(N)}_0(t)=\phi(t)$. Taking then $J=M$ in~\cref{eq:f_N}  we  get from~\cref{eq:ders} that the remaining $M$ conditions (for $m=1,\ldots,M$) can be expressed as the  linear system 
\begin{equation}\label{eq:lin_sym_2d}
 B(t) \mathbf{c}^{(N)}(t)=\bol{\phi}^{(N)}(t),
\end{equation}
whose unknowns are the coefficients $\bold c^{(N)}(t) = [c^{(N)}_1(t),\dots,c^{(N)}_J(t)]^T$, where  the entries of the $J\times J$ matrix function $B$  are given by
\begin{align}
  \label{eq:15}
  B_{m,j}(t) &:= 
  \begin{cases}
    0, \quad m<j, \medskip\\
\displaystyle    \mathbb B_{m,j}\left(\zeta'(t), \dots, \zeta^{(m-j+1)}(t) \right), \quad m \geq j,
\end{cases}
\end{align}
and where the vector on the right-hand-side of equation~\eqref{eq:lin_sym_2d} is given explicitly by $\bol\phi^{(N)}(t) =\lf[\frac{\de}{\de t} \phi(t),\cdots,\frac{\de^{J}}{\de t^J} \phi(t)\rg]^T$. The matrix $B$ defined in~\cref{eq:15} is sometimes referred to as Bell matrix;~cf.~\cite[Chapter 13]{aldrovandi2001special}.

Since $B$ is a lower triangular matrix and its diagonal terms are  $B_{j,j} =
(\zeta')^j$~\cite[Chapter 13]{aldrovandi2001special}, where $\zeta'(t)\neq 0$ for all $t\in[0,2\pi]$
(i.e., the curve is regular), we
readily conclude that the linear system~\cref{eq:lin_sym_2d} is
invertible for all $t\in[0,2\pi]$. Therefore, having retrieved the
coefficients $c^{(N)}_j(t)$ from solving the linear system~\cref{eq:lin_sym_2d}, we immediately obtain $P_N(\tau,t) = \real\left\{f_N(\tau,t)\right\}$ and $Q_N(\tau,t)=|\zeta(\tau)|^{-1}\imag
\left\{\frac{\p}{\p \tau}f_N(\tau,t) \right\}$.

Given that the matrix $B$ is lower triangular,  the solution $\bold  c^{(N)}(t)$ of the linear system~\cref{eq:lin_sym_2d} can be obtained using forward substitution. For the  density interpolation order $M=5$, for example, the matrix $B$ is explicitly given by 
\begin{align}
B=\lf[\begin{array}{ccccc}
\zeta'&0&0&0&0\medskip\\
\zeta''&(\zeta')^2&0&0&0\medskip\\
\zeta'''&3\zeta'\zeta''&(\zeta')^3&0&0\\
\zeta^{(4)}&3(\zeta'')^2+4\zeta'\zeta'''&6(\zeta')^2\zeta''&(\zeta')^4&0\\
\zeta^{(5)}&10\zeta''\zeta'''+5\zeta'\zeta^{(4)}&15\zeta'(\zeta'')^2+10(\zeta')^2\zeta'''&10(\zeta')^3\zeta''&(\zeta')^5
\end{array}\rg].  \label{eq:B_explit}
\end{align}
Matrices $B$ corresponding to orders $M<5$ are simply submatrices of the matrix displayed in~\cref{eq:B_explit}.

\begin{remark} \label{rem:N}It follows from~\cref{eq:complex_point_cond_1}, and the fact that $\phi$ is real-valued,  that $$\lim_{\tau\to t}\imag \lf\{\frac{\p^{m+1}}{\p \tau^{m+1}}f_N(\tau,t)\rg\} = 0\quad\mbox{for}\quad m=0,\ldots,M-1.$$ Therefore, $Q_N$, defined as in~\cref{eq:neu_trace_complex} in terms of $f_N$, satisfies
$$
\lim_{\tau\to t}\frac{\p^m}{\p \tau^m}Q_N(\tau,t) =\lim_{\tau\to t} \imag\lf\{\frac{\p^{m}}{\p \tau^{m}}\lf(\frac{1}{|\zeta'(\tau)|}\frac{\p f_N}{\p\tau}(\tau,t)\rg)\rg\}=0\ \ \mbox{for}\ \ m=0,\ldots,M-1,
$$
 and, consequently, $Q_N(\tau,t) = O((\tau-t)^{M})$ as $\tau\to t$. This fact will play an important role in~\cref{sec:eval-layer-potent}, where an  extension of the proposed   density interpolation technique is presented for the regularization of nearly singular integrals. 
\end{remark}

\subsection{Interpolating function for the single-layer operator}
A procedure similar to the one described above for the construction of $f_N$ allows us to find the coefficients $c^{(S)}_j$ in the expansion
\begin{equation}\label{eq:f_S}
f_S(\tau,t) =F_S(\zeta(\tau),\zeta(t)) =\sum_{j=0}^J\frac{c^{(S)}_j(t)}{j!}(\zeta(\tau)-\zeta(t))^j,\quad \tau,t\in[0,2\pi].
\end{equation}
In fact, in view of  identities~\cref{eq:n_ders},   conditions~\cref{eq:curve-condition-2} can be enforced on $Q_S(\tau,t)=|\zeta(\tau)|^{-1}\imag \frac{\p}{\p \tau} f_S(\tau,t)$ by requiring $f_S$ to satisfy 
\begin{equation}\label{eq:complex_point_cond_2}
\lim_{\tau \to t}  \frac{\p^m}{\p \tau^m}\left[\frac{1}{|\zeta'(\tau)|}\frac{\p}{\p\tau}f_S(\tau,t) -i \phi(\tau) \right]= 0 \quad\mbox{for}\quad m=0,\ldots,M.
\end{equation}
Conditions~\cref{eq:complex_point_cond_2} do not pose any constrain on
$c_0^{(S)}$, so we may set  $c_0^{(S)}(t)=0$.
Letting then $J=M+1$ in~\cref{eq:f_S}  we obtain that conditions~\cref{eq:complex_point_cond_2} 
are fulfilled if and only if  $\bold c^{(S)}(t) = [c^{(S)}_1(t),\dots,c^{(S)}_{J}(t)]^T$
satisfies
\begin{equation}\label{eq:lin_sym_S}
A(t)B(t) \mathbf{c}^{(S)}(t)=\bol\phi^{(S)}(t),
\end{equation}
where $A$ is the $J\times J$  lower-triangular matrix
\begin{align}
  \label{eq:bin}
  A_{m,j}(t) &:= 
  \begin{cases}
    0, \quad m<j, \medskip\\
\displaystyle   \binom{m-1}{j-1}\frac{\de^{m-j}}{\de t^{m-j}}\lf(\frac{1}{|\zeta'(t)|}\rg), \quad m \geq j,
\end{cases}
\end{align}
  $B$ is the Bell matrix~\cref{eq:15}, and  $\bol\phi^{(S)}(t) =\lf[\phi(t),\frac{\de}{\de t} \phi(t),\cdots,\frac{\de^{J-1}}{\de t^{J-1}} \phi(t)\rg]^T$.  Since both $A$ and $B$ are invertible matrices (none of the diagonal entries of  $A$ is zero), the coefficients $c^{(S)}_j(t)$ in the expansion~\cref{eq:f_S} are uniquely determined  by the linear system~\cref{eq:lin_sym_S}. Having found the coefficients $c^{(S)}_j(t)$, we obtain $P_S(\tau,t) = \real\{f_S(\tau,t)\}$ and $Q_S(\tau,t)=|\zeta'(\tau)|^{-1}\imag\{\frac{\p}{\p \tau}f_S(\tau,t)\}$.
\begin{remark}\label{rem:S} Note that~\cref{eq:complex_point_cond_2}, the choice $c_0^{(S)}=0$, and the fact that $\phi$ is  real-valued, imply that that $P_S$ satisfies
$$
\lim_{\tau\to t}\frac{\p^m}{\p \tau^m}P_S(\tau,t) = \lim_{\tau\to t}\real \lf\{\frac{\p^{m}}{\p \tau^{m}}f_S(\tau,t)\rg\} = 0,
$$
for $m=0,\ldots,M+1$. Therefore $P_S(\tau,t) =
O\lf((\tau-t)^{M+2}\rg)$ as $\tau\to t$. This identity, together with \Cref{rem:N}, will be used in \cref{sec:eval-layer-potent} to extend the proposed  density interpolation technique to the regularization of nearly singular integrals.
\end{remark}
\begin{remark}\label{rem:arclength-parametrization} If an arc-length
  parametrization of the curve $\Gamma$ is considered, i.e.,
  $\Gamma = \{ \mathbf{x}(s) = ({\rm x}_1(s),{\rm x}_2(s)): s \in [0,L) \}$ where
  $L$ is the length of the curve and
  $|\mathbf{x'}(s)| = 1$ for all  $s \in [0,L)$, then the matrix  $A$ in~\cref{eq:lin_sym_S} becomes the
  identity matrix and the linear systems \cref{eq:lin_sym} and
  \cref{eq:lin_sym_S} differ only by the right hand side. Although in
  theory it is always possible to re-parametrize a smooth regular curve by
  its arc-length, there are practical advantages of considering the
  more general parametrization~\cref{eq:param}  assumed here.
\end{remark}

As announced, in the next section we present and extension of the harmonic interpolation technique to the regularization of nearly singular integrals. 
\section{Nearly singular integrals in 2D} \label{sec:eval-layer-potent}
The single- and double-layer
potentials are respectively defined as
\begin{align}
\mathcal S[\varphi](\nex) := \int_{\Gamma}G(\nex,\ney)\varphi(\ney)\de s(\ney)\andtext
\mathcal D[\varphi](\nex) := \int_{\Gamma}\frac{\p G(\nex,\ney)}{\p n(\ney)}\varphi(\ney)\de s(\ney),
\label{eq:potentials_gen}\end{align}
for $\nex\in\R^2\setminus\Gamma$, which, using the curve parametrization in 2D, become
\begin{subequations}\begin{align}
\mathcal S[\varphi](\nex) :=& -\frac{1}{2\pi} \int_{0}^{2\pi} \log\lf(|\nex-\bnex(\tau)|\rg)\phi(\tau)|\bnex'(\tau)|\de \tau,\\
\mathcal D[\varphi](\nex) :=&~  \frac{1}{2\pi} \int_{0}^{2\pi} \frac{(\nex-\bnex(\tau))\cdot \bnor(\tau)}{|\nex-\bnex(\tau)|^2}\phi(\tau)|\bnex'(\tau)|\de \tau,
\end{align}\label{eq:potentials}\end{subequations} where
$\phi(\tau)=\varphi(\mathbf{x}(\tau))$. As it is well known, the kernels in \cref{eq:potentials} become
nearly singular (i.e., develop sharp peaks) as the target point $\nex$ approaches the boundary
$\Gamma$, making the numerical evaluation of
\cref{eq:potentials} challenging. In this section we show that the
harmonic interpolating functions $U_N$ and $U_S$ constructed
in~\cref{eq:hif} can be effectively used to effectively regularize the kernels of
both layer potentials.


We first note that by Green's third identity, any function $U:\R^2\times\Gamma\to\R$  harmonic in the first variable, satisfies 
\begin{equation}-\mu(\nex) U(\nex,\nex_0)=\mathcal D[U(\cdot,\nex_0)](\nex)-\mathcal S[\p_n U(\cdot,\nex_0)](\nex),\quad \nex\notin\Gamma,
\label{eq:GHarmonic}\end{equation} where $\mu(\nex)=1$ if  $\nex$ lies inside the domain enclosed by $\Gamma$, and $\mu(\nex)=0$ otherwise. Using this identity and recalling that $P$ and $Q$ denote the Dirichlet and Neumann boundary values of $U$ on $\Gamma$, we have that  the single-layer potential  can be expressed as 
\begin{equation}\label{eq:single_layer_pot}\begin{split}
&\mathcal S[\varphi](\nex) =\mu(\nex) U(\nex,\nex_0) \\
& +\int_{0}^{2\pi}\!\! \lf\{\frac{(\nex-\bnex(\tau))\cdot \bnor(\tau)}{2\pi|\nex-\bnex(\tau)|^2}P(\tau,t_0)- \frac{\log\lf(|\nex-\bnex(\tau)|\rg)}{2\pi}(\phi(\tau)-Q(\tau,t_0))\rg\}|\bnex'(\tau)|\de \tau,
\end{split}\end{equation}
while the double-layer potential can be expressed as
\begin{equation}\label{eq:double_layer_pot}
\begin{split}
&\mathcal D[\varphi](\nex) =\mu(\nex) U(\nex,\nex_0) \\
& +\int_{0}^{2\pi}\!\! \lf\{\frac{(\nex-\bnex(\tau))\cdot \bnor(\tau)}{2\pi|\nex-\bnex(\tau)|^2}(\phi(\tau)-P(\tau,t_0))- \frac{\log\lf(|\nex-\bnex(\tau)|\rg)}{2\pi}Q(\tau,t_0)\rg\}|\bnex'(\tau)|\de \tau,
\end{split}\end{equation}
for  all $\nex\notin\Gamma$ and $\nex_0=\bnex(t_0)\in\Gamma$. 

In order to smooth out the nearly singular integrands
in~\cref{eq:single_layer_pot} and \cref{eq:double_layer_pot} that
arise for points $\nex$ close to $\Gamma$, we let $U=U_S$
in~\cref{eq:single_layer_pot} and $U= U_N$
in~\cref{eq:double_layer_pot}, where $U_S$ and $U_N$ are the harmonic
interpolating functions constructed in~\cref{eq:hif}. The expansion
point $\nex_0$ is then selected as $\nex_0=\bnex(t_0)$ with $t_0=
{\rm arg}\min_{t\in[0,2\pi)}|\nex-\bnex(t)|$, where we assume  that $\nex$ is
sufficiently close to $\Gamma$ so that this minimizer is unique. Doing so, and in
view of the fact that  $\phi(\tau)-Q_S(\tau,t) = O((\tau-t)^{M+1})$ and
$P_S(\tau,t) = O\lf((\tau-t)^{M+2}\rg)$ as $\tau\to t$ (see~\Cref{rem:S}), we  obtain  that the integrands
in~\cref{eq:single_layer_pot} satisfy
\begin{subequations}\label{eq:asymp_1}\begin{align}
\frac{(\nex-\bnex(\tau))\cdot \bnor(\tau)}{|\nex-\bnex(\tau)|^2}P_S(\tau,t) &= O\lf(\frac{(\tau-t)^{M+2}}{|\nex-\bnex(\tau)|}\rg),\\
\log\lf(|\nex-\bnex(\tau)|\rg)(\phi(\tau)-Q_S(\tau,t_0)) &= O\lf(\log\lf(|\nex-\bnex(\tau)|\rg)(\tau-t)^{M+1}\rg),
\end{align}\end{subequations}
as $\tau\to t$.  Similarly, recalling that 
$\phi(\tau) - P_N(\tau,t) = O\lf((\tau-t)^{M+1}\rg)$ and $Q_N(\tau,t) = O((\tau-t)^{M})$ as $\tau\to t$
(see~\Cref{rem:N}), it can be shown the integrands in~\cref{eq:double_layer_pot} satisfy
\begin{subequations}\label{eq:asymp_2}\begin{align}
\frac{(\nex-\bnex(\tau))\cdot \bnor(\tau)}{|\nex-\bnex(\tau)|^2}(\phi(\tau)-P_N(\tau,t_0)) &= O\lf(\frac{(\tau-t)^{M+1}}{|\nex-\bnex(\tau)|}\rg),\\
\log\lf(|\nex-\bnex(\tau)|\rg)Q_N &= O\lf(\log\lf(|\nex-\bnex(\tau)|\rg)(\tau-t)^{M}\rg),
\end{align}\end{subequations}
as $\tau\to t$, where $M$ is the prescribed harmonic interpolation order. The asymptotic identities~\cref{eq:asymp_1} and~\cref{eq:asymp_2} show that the smoothness of the integrands at and around the nearly singular point $\nex_0\in\Gamma$ can be controlled by means of our technique.

As shown in~\cite[Section 3.4]{perez2018plane}, the proposed
technique can also be utilized to regularize the kernels that arise in
problems involving multiply connected domains with boundary components
that are close to each other. In this case, nearly singular integrals
associated with any of the integral operators~\cref{eq:int_op} can be
directly computed by evaluating the regularized expressions for the
layer potentials~\cref{eq:single_layer_pot}
and~\cref{eq:double_layer_pot}, or their respective normal derivatives
on a curve $\Gamma'$ that does not intersect $\Gamma$. The
numerical results presented in this paper demonstrate the
effectiveness of the proposed technique for the regularization of
nearly singular integrals.

\section{High-order kernel singularity regularization  for 3D problems}\label{sec:3D}
As it turns out, most of the ideas presented above in this paper for the Laplace boundary integral operators in 2D can be directly generalized to 3D. Indeed, letting $U_S,U_N:\R^3\times\Gamma\to\R$ denote  harmonic functions (in the first variable) in all of $\R^3$  it can be shown---using Green's third identity---that for a smooth closed oriented surface $\Gamma$ the four integral operators of Calder\'on calculus can be expressed~as
\begin{subequations}\begin{align}
S[\varphi](\nex) &=\frac{U_S(\nex,\nex)}{2} +  K [U_S(\cdot,\nex)])(\nex) +  S\lf[\varphi-\p_n U_S(\cdot,\nex)\rg](\nex),\label{eq:3d_SL}\\
K[\varphi](\nex) &=-\frac{U_N(\nex,\nex)}{2} +  K [\varphi-U_N(\cdot,\nex)](\nex) +  S\lf[\p_n U_N(\cdot,\nex)\rg](\nex),\label{eq:3d_DL}\\
K'[\varphi](\nex) &=\frac{\p_n U_S(\nex,\nex)}{2} +  N [U_S(\cdot,\nex)](\nex) +  K'\lf[\varphi-\p_n U_S(\cdot,\nex)\rg](\nex),\label{eq:3d_DS}\\
N[\varphi](\nex) &=-\frac{\p_n U_N(\nex,\nex)}{2} +  N [\varphi-U_N(\cdot,\nex)])(\nex) +  K'\lf[\p_n U_N(\cdot,\nex)\rg](\nex),\label{eq:3d_HS}
\end{align}\label{eq:IE_3d}\end{subequations}
for $\nex\in\Gamma\subset\R^3$, where the definition of the boundary integral operators $S$, $K$, $K'$ and $N$ is the same as the one given in equation~\eqref{eq:int_op}, except for  the fundamental solution of the Laplace equation whose expression in 3D is
\begin{equation}
G(\nex,\ney):=\frac{1}{4\pi|\nex-\ney|}. \label{eq:GF_3D}\end{equation}
Note that unlike the 2D case, the kernels of the operators $K$ and $K'$ are (weakly) singular, and thus all the four Laplace boundary integral operators in 3D require application of the  density interpolation technique to achieve integral expressions in terms of more regular integrands.

In order to establish a set of point conditions on the Dirichlet and Neumann boundary values of $U_S$ and $U_N$ on $\Gamma$ that lead to more regular integrands in the expressions on the right-hand-side of~\eqref{eq:IE_3d}, we resort to the local parametrization of the surface $\Gamma$.  It follows from the regularity of $\Gamma$ that at every point $\nex\in\Gamma$ there exists a smooth local diffeomorphism $\bnex:B_{\bxi}\to\Gamma\cap B_{\nex}$, where $B_{\bxi}\subset\R^2$ is an open neighborhood of $\bxi$, $B_{\nex}$ in an open neighborhood of $\nex\in\Gamma$, and the mapping $\bnex$ is such that $\bnex(\bxi)=\nex$.  Resorting to the parametrization $\bnex$ and the index notation for partial derivatives, i.e., letting $\alpha=(\alpha_1,\alpha_2)$, $|\alpha|=\alpha_1+\alpha_2$, and $\p^\alpha = \p_1^{\alpha_1}\p_2^{\alpha_2}$ where $\p_j$  denotes the derivative with respect to $\xi_j$ ($\bxi=(\xi_1,\xi_2)$), it can be shown that the point conditions  
\begin{subequations} \label{eq:curve-condition-2_3D}\begin{align} 
  \lim_{\bxi' \to \bxi}  \p^{\alpha} P_S\lf(\bxi',\bxi\rg)   =&~0 \quad \mbox{for} \quad |\alpha|=0,\ldots, M_P,\\
  \lim_{\bxi' \to \bxi}  \p^{\alpha} \left\{\phi\lf(\bxi'\rg)-Q_S\lf(\bxi',\bxi\rg)  \right\}  =&~0 \quad \mbox{for} \quad |\alpha|=0,\ldots, M_Q,
\end{align}\end{subequations}
and 
\begin{subequations}\label{eq:curve-condition-1_3D}\begin{align}  
    \lim_{\bxi' \to \bxi}  \p^{\alpha}   \left\{\phi\lf(\bxi'\rg)-P_N\lf(\bxi',\bxi\rg)  \right\}  = 0&\qquad \mbox{for} \quad |\alpha|=0,\ldots, M_P,\\
  \lim_{\bxi' \to \bxi}  \p^{\alpha} Q_N\lf(\bxi',\bxi\rg)   = 0&\qquad \mbox{for} \quad |\alpha|=0,\ldots, M_Q,
\end{align}\end{subequations}
where $\phi(\bxi) = \varphi(\bnex(\bxi)),$ $P_{S,N}(\bxi',\bxi) = U_{S,N}(\bnex(\bxi'),\bnex(\bxi))$ and $Q_{S,N}(\bxi',\bxi) = \p_n U_{S,N}(\bnex(\bxi'),\bnex(\bxi))$, suffice to guarantee that  the estimates
$$U_S(\ney,\nex)=O(|\nex-\ney|^{M_P+1}), \quad\varphi(\ney)-\p_n U_S(\ney,\nex)=O(|\nex-\ney|^{M_Q+1}),$$  $$\p_nU_N(\ney,\nex)=O(|\nex-\ney|^{M_Q+1}) \andtext  \varphi(\ney)-U_N(\ney,\nex)=O(|\nex-\ney|^{M_P+1})$$ hold true as $\ney\to\nex$, where all the constants in the ``big-$O$" notations above can be bounded uniformly in $\nex$. Therefore, the smoothness of integrands in~\eqref{eq:IE_3d} is controlled by the regularization orders $M_P$ and $M_Q$, provided the requirements~\eqref{eq:curve-condition-2_3D} and~\eqref{eq:curve-condition-1_3D} are satisfied. We emphasize that unlike in the 2D case, there are more local vanishing conditions to be imposed in 3D per each point $\nex$ in order to achieve the same integrand regularization orders.

In the numerical examples considered in this paper we select $M_P = 2$ and $M_Q=1$ (i.e. the lowest values of these parameters that remove the singularity of the hyper-singular BIO and at the same time lead to $C^1$ integrands for the evaluation of both single and double-layer BIOs) and harmonic function functions $U_S$ and $U_N$ given in terms of linear combinations of homogeneous harmonic polynomials of order at most two. In detail we let
\begin{align}
  \label{eq:4}
  U_{S}(\ner,\nex) = \sum_{j=0}^8c^{(S)}_{j}(\nex) H_j(\ner-\nex)\ \ \mbox{and}\ \ \  U_{N}(\ner,\nex) = \sum_{j=0}^8c^{(N)}_{j}(\nex)H_j(\ner-\nex)\ \ (\ner\in\R^3,\ \nex\in\Gamma),\end{align}
where the homogeneous harmonic polynomials utilized in the expansions above are 
\begin{equation}\label{eq:Har_Pol}\begin{split}
H_0(\ner) = 1,\quad  H_1(\ner) = x,\quad H_2(\ner)=y,\quad  H_3(\ner)=z,\quad H_4(\ner)=xy,\quad H_5(\ner)= xz,\\
 H_6(\ner)=yz,\quad H_7(\ner)=x^2-y^2,\andtext H_8(\ner)=x^2-z^2,\qquad \lf(\ner = (x,y,z)\rg),
\end{split}\end{equation}
and the expansion coefficients $c_j^{(S)}$ and $c_j^{(N)}$ must be obtained by enforcing the point conditions~\eqref{eq:curve-condition-2_3D} and~\eqref{eq:curve-condition-1_3D} on $U_S$ and $U_N$. In order to enforce such conditions,  we first let  $h_j(\bxi',\bxi)=H_j(\bnex(\bxi')-\bnex(\bxi))$ and $h_{n,j}(\bxi',\bxi) = \nabla H_j(\bnex(\bxi')-\bnex(\bxi))\cdot \bnor(\bxi')$ denote the Dirichlet and Neumann traces of the harmonic polynomials~$H_j$, respectively,  where the unit normal at $\nex=\bnex(\bxi)\in\Gamma$ is given by $
\bnor(\bxi) = (\p_1\bnex(\bxi)\wedge  \p_2\bnex(\bxi))/|\p_1\bnex(\bxi)\wedge  \p_2\bnex(\bxi)|.$
Therefore, the enforcement of the points conditions~\eqref{eq:curve-condition-2_3D} and~\eqref{eq:curve-condition-1_3D}, respectively, leads  to the following  linear systems
\begin{equation}
A(\bxi)\bold c^{(S)}(\bnex(\bxi)) = \boldsymbol \phi^{(S)}(\bxi)\andtext A(\bxi)\bold c^{(N)}(\bnex(\bxi)) = \boldsymbol\phi^{(N)}(\bxi),\label{eq:lin_sym}
\end{equation}
for the vectors of  coefficients $\bold c^{(S)} = [c_0^{(S)},\ldots,c_8^{(S)}]^T$ and $\bold c^{(N)} = [c_0^{(N)},\ldots,c_8^{(N)}]^T$, where the entries of the matrix $A=(a_{i,j})$ are given by 
\begin{subequations}\begin{align}
&a_{1,j}(\bxi) = h_{j-1}(\bxi,\bxi),&& a_{2,j}(\bxi) = \p_{1}h_{j-1}(\bxi,\bxi),&& a_{3,j}(\bxi) = \p_{2}h_{j-1}(\bxi,\bxi),\\
&a_{4,j} (\bxi)= h_{n,j-1}(\bxi,\bxi),&& a_{5,j}(\bxi) = \p_1h_{n,j-1}(\bxi,\bxi),&&  a_{6,j}(\bxi) = \p_2h_{n,j-1}(\bxi,\bxi),\\
&a_{7,j}(\bxi) = \p^2_{1}h_{j-1}(\bxi,\bxi),&& a_{8,j}(\bxi) = \p_{1}\p_2h_{j-1}(\bxi,\bxi),&& a_{9,j}(\bxi) = \p^2_2h_{j-1}(\bxi,\bxi),
\end{align}\label{eq:mat_entries}\end{subequations}
for $j=1,\ldots,9$, and the right-hand-side vectors are 
\begin{equation}\begin{split} \boldsymbol\phi^{(S)}(\bxi) =&~\lf[0,0,0,\phi(\bxi),\p_1\phi(\bxi), \p_2\phi(\bxi),0,0,0\rg]^T,\\
 \boldsymbol\phi^{(N)}(\bxi) =&~\lf[\phi(\bxi),\p_1\phi(\bxi),\p_2\phi(\bxi), 0,0,0,\p^2_1\phi(\bxi),\p_1\p_2\phi(\bxi),\p^2_2\phi(\bxi)\rg]^T.
\end{split}\label{eq:rhs_vecs}\end{equation}
As discussed in Appendix~\ref{app:inv_mat}, the matrix $A(\bxi)\in\R^{9\times 9}$ is invertible for all the points on the surface and in fact, $\operatorname{det}(A(\bxi)) = -4|\p_1\bnex(\bxi)\wedge\p_2\bnex(\bxi)|^5\neq 0$ (we note here that $|\p_1\bnex(\bxi)\wedge\p_2\bnex(\bxi)|$ represents the surface element, and thus the determinant of the 3D harmonic Taylor-like interpolation problem bears similarities to its 2D counterpart). Clearly, with the aforementioned selections the integrands on the right-hand-side of~\eqref{eq:3d_SL} and~\eqref{eq:3d_DL} become $C(\Gamma)$-functions, while the integrands in~\eqref{eq:3d_DS} and~\eqref{eq:3d_HS} become  bounded functions. Higher-order versions of the harmonic interpolation technique (corresponding to larger values of the parameters $M_P$ and $M_Q$ in equations~\eqref{eq:curve-condition-2_3D} and~\eqref{eq:curve-condition-1_3D}) can be pursued at the cost of incorporating higher-order harmonic polynomials in the Taylor-like interpolation scheme as well as higher-order derivatives of the surface parametrization and of the density $\phi$. However, the invertibility of matrices corresponding to the ensuing interpolation problems~\eqref{eq:lin_sym} for higher values of the parameters $M_P$ and $M_Q$ remains an open question (note that for the next  density interpolation order, that is $M_P=3$ and $M_Q=2$, one would have to deal with $16\times 16$ matrices $A(\bxi)$). Nevertheless, the numerical results presented in Section~\ref{3d_num} illustrate that the choice $M_P=2$ and $M_Q=1$ already leads to very accurate results produced by a simple implementation.

As in 2D, nearly singular integrands arising due to observation points $\nex\not\in\Gamma$ near the surface, can be regularized utilizing the harmonic functions $U_S$ and $U_N$ to interpolate the density function $\varphi$ at the surface point $\nex_0={\rm arg}\min_{\ney\in\Gamma}|\nex-\ney|\in\Gamma$. In particular, it follows from~\eqref{eq:potentials_gen} and~\eqref{eq:GHarmonic} that following expressions for the single- and double-layer potentials
\begin{subequations}\begin{align}\begin{split}
\mathcal S[\varphi](\nex) =&~\mu(\nex) U_S(\nex,\nex_0) \\
&+\int_{\Gamma} \lf\{\frac{\p G(\nex,\ney)}{\p n(\ney)}U_S(\ney,\nex_0)+ G(\nex,\ney)(\varphi(\ney)-\p_nU_S(\ney,\nex_0))\rg\}\de s(\ney), \end{split}\label{eq:single_layer_pot_3D}\\
\begin{split}\mathcal D[\varphi](\nex) =&~-\mu(\nex) U_N(\nex,\nex_0) \\
&+ \int_{\Gamma}\lf\{\frac{\p G(\nex,\ney)}{\p n(\ney)}(\varphi(\ney)-U_N(\ney,\nex_0))+G(\nex,\ney)\p_n U_N(\ney,\nex_0)\rg\}\de s(\ney),\end{split}\label{eq:double_layer_pot_3D}
\end{align}\label{eq:lay_pots}\end{subequations}
hold for all $\nex\notin\Gamma\subset\R^3$, where $\mu(\nex)=1$ if  $\nex$ lies inside the domain enclosed by $\Gamma$, and $\mu(\nex)=0$ otherwise.

We present next a variety of numerical results that showcase the effectiveness of the regularization technique via harmonic density interpolation in both two and three dimensions.

 \section{Numerical examples and applications}\label{sec:numerics}
 \subsection{Singular integrals in 2D}
We first present a simple Nystr\"om method based on the high-order harmonic density interpolation technique (referred to in what follows by the acronym HDI) for kernel regularization combined with the classical
trapezoidal rule for the direct numerical evaluation of  the
single-layer and hypersingular operators. Given that we assumed throughout that the smooth closed curves $\Gamma$ are given in terms of smooth, $2\pi$ periodic parametrizations, we consider a uniform discretization of the interval $[0,2\pi]$ with grid points $t_j = h j$, $h=\pi/N$ for $j=0,\ldots,2N-1$, where $N>0$. Using global trigonometric polynomial interpolation of the densities $\phi$, application of the trapezoidal rule leads to the following semi-discrete approximations of the parametrized single-layer and hyper-singular operators
\begin{align}
\tilde N[\phi](t) &\approx-\frac{Q_N(t,t)}{2}+h\sum_{j=0}^{2N-1} \lf\{R^{(P)}_N(t_j,t)+R^{(Q)}_N(t_j,t)\rg\}|\bnex'(t_j)|,\label{eq:Ntrap}\\
  \tilde S[\phi](t) &\approx\frac{P_S(t,t)}{2}+h\sum_{j=0}^{2N-1}\lf\{R^{(P)}_S(t_j,t)+R^{(Q)}_S(t_j,t)\rg\}|\bnex'(t_j)|,\label{eq:Strap}
\end{align}
for $t\in[0,2\pi]$, where~$R^{(P)}_N$, $R^{(Q)}_N$, $R^{(P)}_S$ and $R^{(Q)}_S$ are defined in~\cref{eq:NP}, \cref{eq:NQ}, \cref{eq:SP} and \cref{eq:SQ}, respectively. The construction of all these functions is described in~\cref{eq:hif} and requires computation of high-order derivatives of  the parametrization $\bnex(\tau)$ as well as the trigonometric polynomial interpolants of $\phi(\tau)$ at the grid points. Since the functions $\phi(\tau)$ are  $2\pi$-periodic and analytic, the derivatives of their trigonometric polynomial interpolants can be computed via FFT-based numerical differentiation with errors that decay exponentially fast as  the size $N$ of the equi-spaced grid increases; cf.~\cite{tadmor1986exponential}.   In light of~\Cref{lem:cond_1}, which established the smoothness and periodicity of $R^{(P)}_N(\cdot,t)$ and $R^{(Q)}_N(\cdot,t)$,  we conclude that  for any regularization order $M\geq 1$ the trapezoidal rule approximation~\cref{eq:Ntrap} of the hypersingular operator converges exponentially fast as $N$ increases~\cite{davis2007methods,trefethen2014exponentially}. Finally, fully discrete approximations of the operators $\tilde{N}$ and $\tilde{S}$ are obtained by simply evaluating at the grid points their semi-discrete versions in equations~\eqref{eq:Ntrap} and respectively~\eqref{eq:Strap}.

In order to demonstrate the fast convergence of the HDI approximation~\cref{eq:Ntrap} of the hyper-singular operator, we present in~\Cref{fig_N_orders} maximum absolute errors on the circular boundary  between our discretizations corresponding to harmonic interpolation orders~$M= 0,1$ and $2$ and a reference solution produced by the spectrally accurate evaluation of the hypersingular operator~\cite{kress2014collocation}. The reference solution was obtained using a refined uniform discretization of the interval $[0,2\pi)$ with grid size $h=\pi/640$. In the case of the application of the harmonic interpolation technique of order $M=0$---the blue curve in Figure~\ref{fig_N_orders}, the (undefined) values of the integrand at $\tau=t$  were replaced by zero.  As expected, for orders $M\geq 1$ the  kernel in~\cref{eq:hyper_rep}
becomes analytic, and therefore  exponential convergence of
\cref{eq:Ntrap} is observed as $N=2\pi/h$
increases. \Cref{fig_N_integrand} displays the smooth integrand~$R^{(P)}_N$ resulting from application of the harmonic interpolation technique for $M=1$.

\begin{figure}[h!]
\centering	
 \subfloat[][Linear and exponential convergence of the trapezoidal rule for orders $M=0$ and $M=1,2$, respectively.]{\includegraphics[scale=0.6]{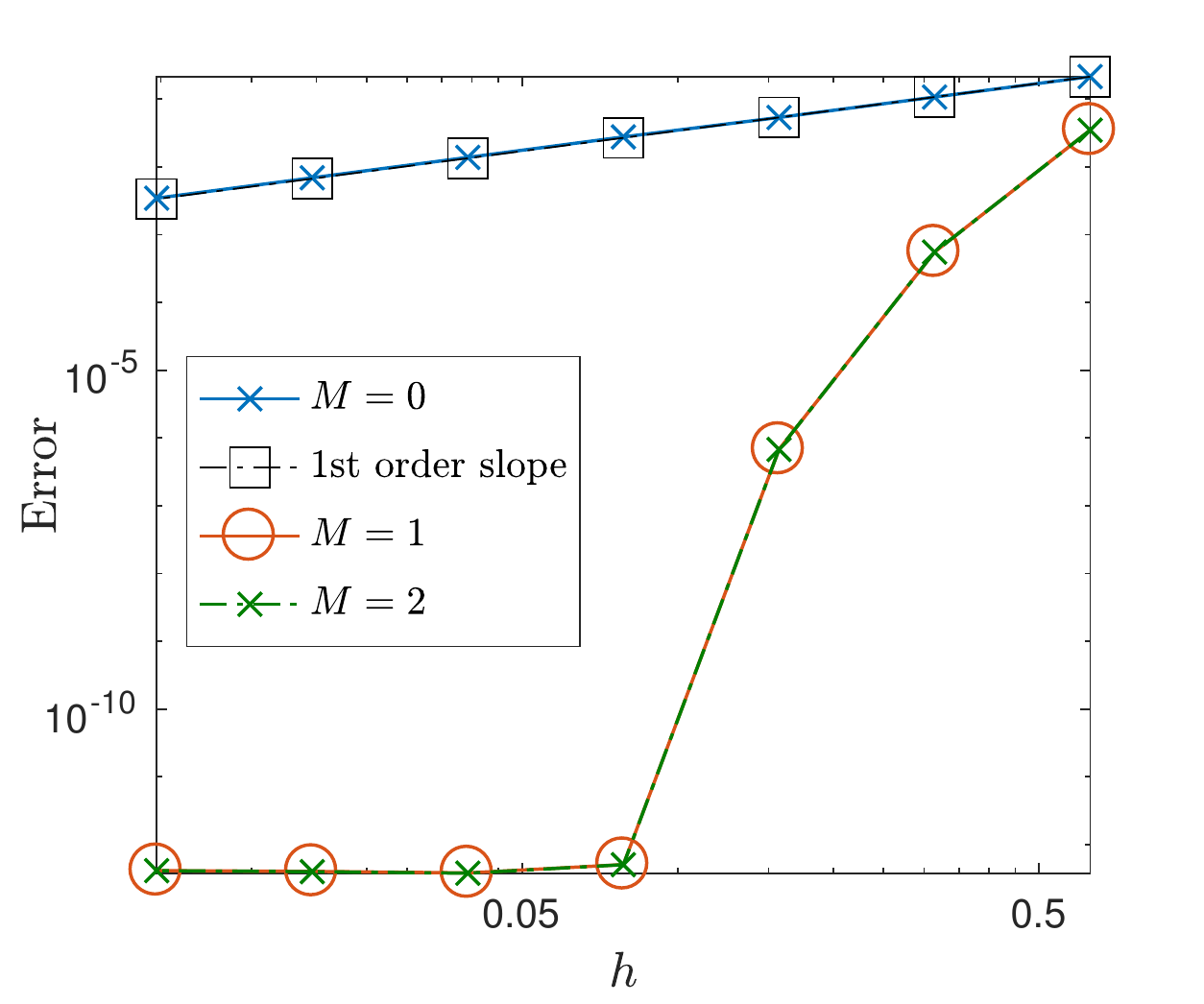}\label{fig_N_orders}}\quad
 \subfloat[][Regularized singular integrand in the hypersingular operator for $M=1$. The diagonal terms of the smoothed integrand are marked in red.]{\includegraphics[scale=0.47]{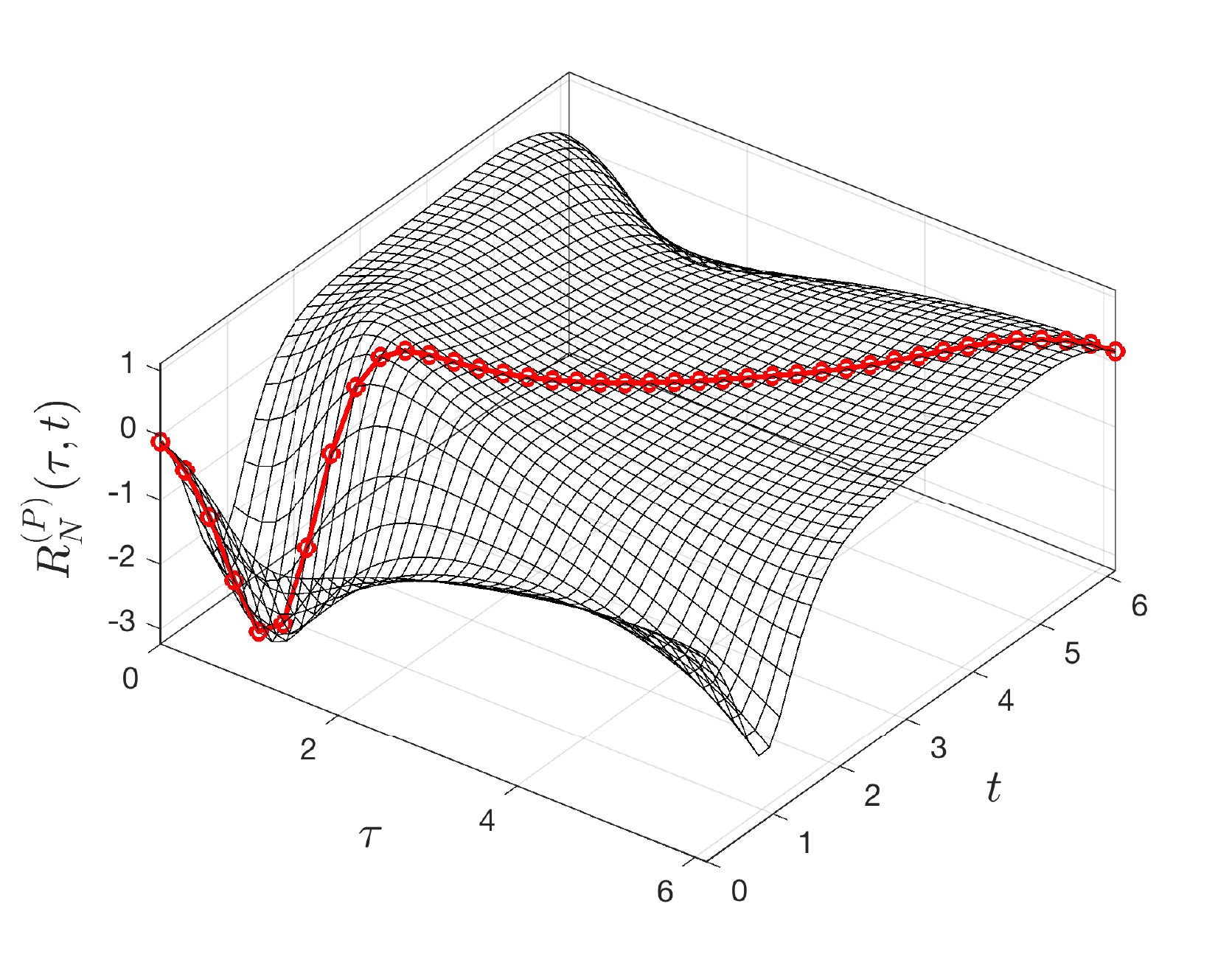}\label{fig_N_integrand}}\\
\caption{Convergence of the HDI discretization~\cref{eq:Ntrap} for the hypersingular operator $N[\varphi]$ and smoothness of the $R_N^{(P)}$. The density function utilized in these examples is $\varphi = u|_{\Gamma}$ with $u(\nex) = \e^{\sin(x_1\cos x_2)}/\sqrt{(x_1-\frac{1}{3})^2+(x_2-\frac{1}{3})^2}$ and  $\Gamma = \lf\{\sqrt{x_1^2+x_2^2}=1\rg\}$.}\label{fig:N_orders}
\end{figure}  

For the single-layer operator, in turn, the
order of convergence of the HDI discretization~\cref{eq:Strap} is limited by the polylogarithmic singularity of~$R^{(Q)}_S$ (see~\Cref{lem:cond_2}).  As $M$ increases, $R^{(Q)}_S$ becomes smoother, but the logarithmic terms cannot be completely removed by the proposed technique. The effect of the smoothness of $R^{(Q)}_S$  on the convergence of the HDI discretization is demonstrated in
\Cref{fig:SL_orders}, which displays the maximum absolute errors in the evaluation of the single-layer operators for even (\Cref{fig_S_even}) and odd (\Cref{fig_S_odd}) regularization orders $M$. The errors are measured with respect to a highly accurate evaluation of single-layer operator obtained by means of the Martensen-Kussmaul quadrature rule~\cite{KUSSMAUL:1969,MARTENSEN:1963} using a refined uniform grid with $h=\pi/640$. Interestingly, as noted in~
\Cref{rem:extra-order-single-layer} below, an extra order of convergence
is gained due to the symmetry of the integrand $R^{(Q)}_S$ for $M$ even, which explains the fact that the same
convergence rates are observed in~\Cref{fig_S_even,fig_S_odd} for even and odd orders $M$, respectively.
\begin{remark}
  \label{rem:extra-order-single-layer}
It follows from~\cite{celorrio1999euler}---where the Euler-Maclaurin
formula in the presence of a logarithmic singularity is derived---that for an odd  density interpolation order $M=2m+1$, $m\geq 0$, the approximation of the single-layer operator~\cref{eq:Strap} yields errors of order $h^{2m+3}$.

For an even  density interpolation order $M=2m$, $m\geq 0$, in turn,  we can write  $[\phi(t)-Q_S(\tau,t)]|\bnex'(\tau)|=\sin^{2m+1}(\frac{\tau-t}{2})g(\tau,t)$ where $g(\cdot,t)$ is an analytic $2\pi$-periodic function. Using this fact we have that  $R_S^{Q}(\tau,t)|\bnex'(\tau)|$ in~\cref{eq:single-layer-removed-singularity} can be expressed as
 \begin{equation}\begin{split}
R_S^{Q}(\tau,t)|\bnex'(\tau)|= \rho(\tau,t)-\frac{\sin^{2m+1}(\frac{\tau-t}{2})g(\tau,t)}{4\pi}\log\lf(\frac{|\bnex(\tau)-\bnex(t)|^2}{4\sin^2\lf(\frac{\tau-t}{2}\rg)}\rg)|\bnex'(\tau)|\end{split}\end{equation}  
in terms of 
$$
\rho(\tau,t)=-\frac{\sin^{2m+1}(\frac{\tau-t}{2})g(\tau,t)}{4\pi}\log\lf(4\sin^2\lf(\frac{\tau-t}{2}\rg)\rg),
$$
 and a $2\pi$-periodic analytic  function that is integrated with exponentially small errors by the trapezoidal rule.
Furthermore, we note that $\rho(\tau,t)$ can be split as
\begin{equation}\begin{split}
\rho(\tau,t)=&-\frac{\sin^{2m+1}(\frac{\tau-t}{2})g(t,t)}{4\pi}\log\lf(4\sin^2\lf(\frac{\tau-t}{2}\rg)\rg)\\
&-\frac{\sin^{2m+1}(\frac{\tau-t}{2})[g(\tau,t)-g(t,t)]}{4\pi}\log\lf(4\sin^2\lf(\frac{\tau-t}{2}\rg)\rg).
\end{split}\label{eq:imp_term}\end{equation}
Since the first term in~\cref{eq:imp_term} is a $2\pi$-periodic function that is odd with respect to any of the quadrature points $\tau=t_j=j\pi/N$, $j=0,\ldots,2N-1$, it readily follows that the trapezoidal rule integrates it exactly (and yields 0). On the other hand, since the second term in~\cref{eq:imp_term}  is a $2\pi$-periodic function of class $C^{2m+1}$ and the first $2m+1$ derivates of  $\sin^{2m+1}(\frac{\tau-t}{2})[g(\tau,t)-g(t,t)]$ vanish at $\tau=t$, we have from~\cite{celorrio1999euler} again,  that the trapezoidal rule  applied to this term yields errors of order $h^{2m+3}$. From these observations we thus conclude that  the trapezoidal rule approximation~\cref{eq:Strap} yields errors  of  order $h^{2m+3}$ for even $M=2m$.\end{remark}
\begin{figure}[h!]
\centering	
 \subfloat[][Trapezoidal rule convergence for 
 $M=0,2$ and 4.]{\includegraphics[scale=0.6]{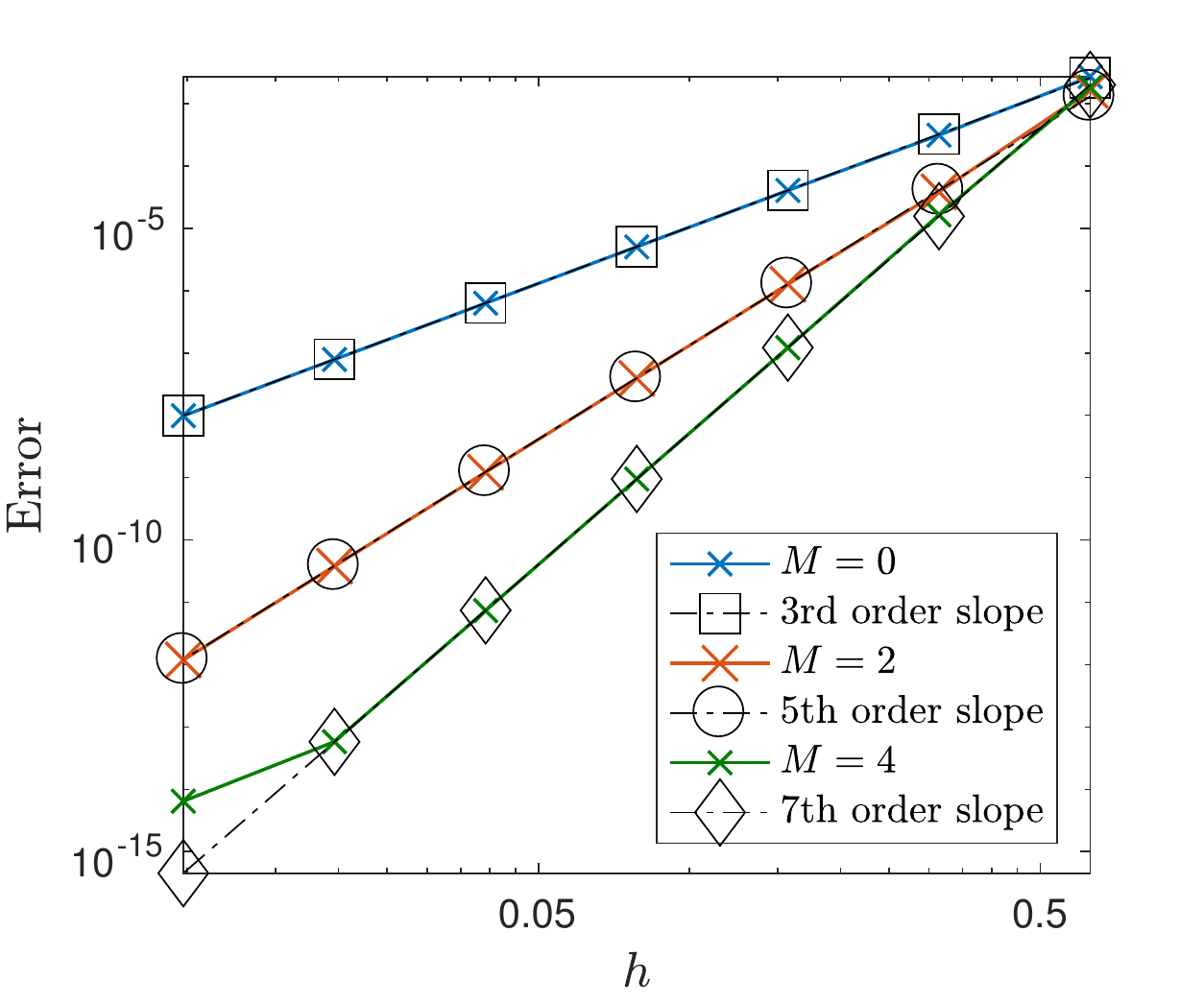}\label{fig_S_even}}\quad
 \subfloat[][Trapezoidal rule convergence for 
 $M=1,3$ and 5.]{\includegraphics[scale=0.6]{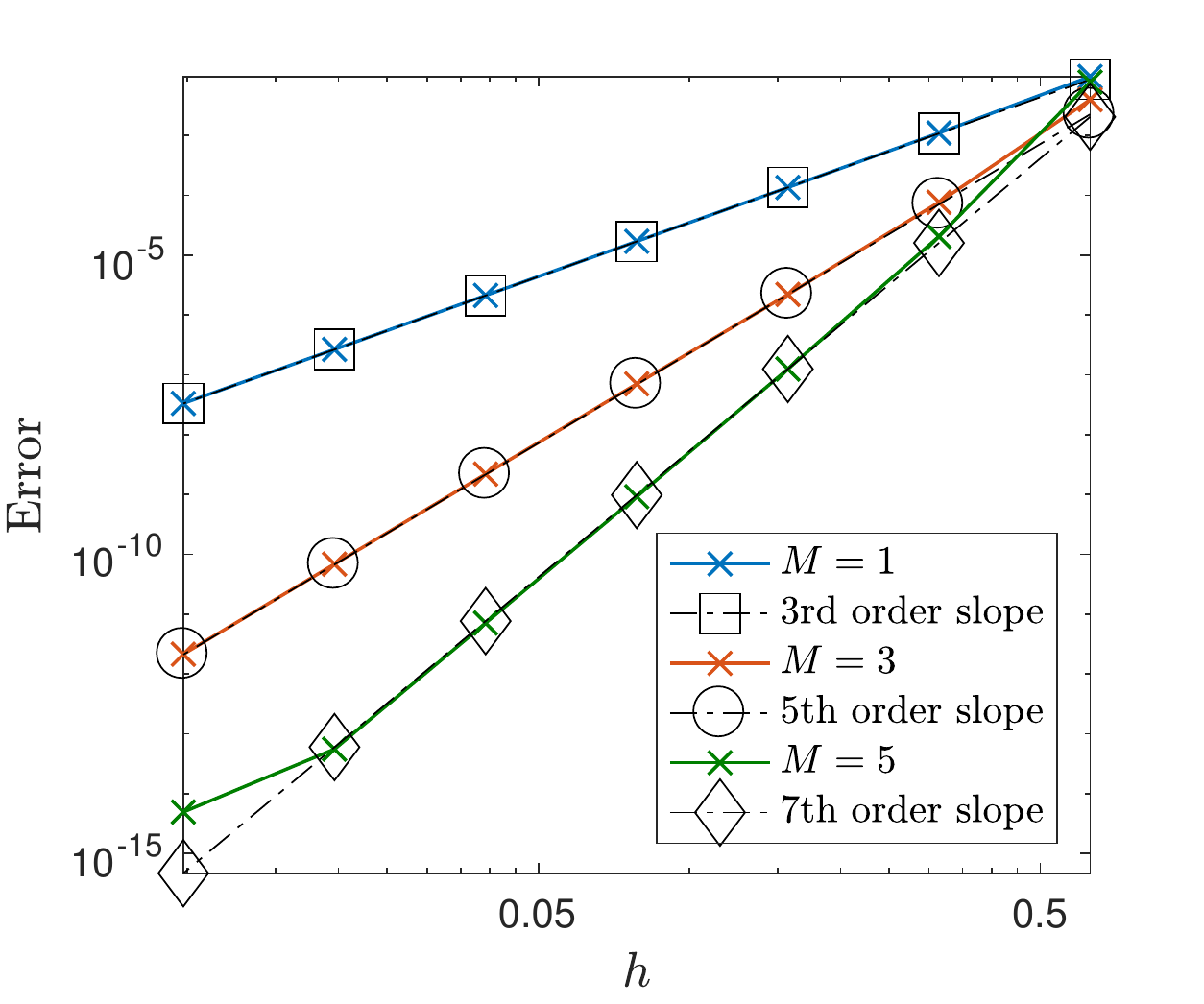}\label{fig_S_odd}}\\
\caption{Convergence of the HDI discretization~\cref{eq:Strap} of the single-layer operator $S[\varphi]$ for  density interpolation orders  (a) $M=0,2,4$  and (b) $M=1,3,5$. The density function utilized in these examples is $\varphi = u|_{\Gamma}$ with $u(\nex) = \e^{\sin(x_1\cos x_2)}/\sqrt{(x_1-\frac{1}{3})^2+(x_2-\frac{1}{3})^2}$ and $\Gamma = \lf\{\sqrt{x_1^2+x_2^2}=1\rg\}$.}\label{fig:SL_orders}
\end{figure}

{Finally, in order to illustrate the competitiveness of the proposed HDI method we present comparisons with the recently introduced  Quadrature-by-Expansion (QBX) method  of Klockner et al. \cite{klockner2013quadrature}. The reason why we consider the QBX method is that, just like the HDI kernel regularization method, the QBX method can deliver high-order discretizations of singular and nearly singular boundary integral operators in both two and three dimensions. For presentation simplicity we focus here on the evaluation of the single-layer operator. Following \cite{Barnett:2014tq}, we compute the single-layer operator by the surrogate expansion
\begin{align}
  \label{eq:5}
  \tilde S[\phi](t) = \real\left\{ \frac{-1}{2\pi}\int_\Gamma \log\lf(\zeta(\tau)-\zeta(t)\rg) \phi(\tau)\zeta'(\tau)\de\tau\right\}\approx \real \left\{ \sum_{l=0}^{p-1} c_l \lf(\zeta(t)-z_0\rg)^l \right\}, 
\end{align}
where the expansion center $z_0$ lies at a distance $\epsilon$ away from the evaluation point~$\zeta(t)$ along the normal direction to the curve, with $\epsilon>0$ (resp. $\epsilon<0$) corresponding to points $z_0$ lying outside (resp. inside) the curve.  
The coefficients $c_l$ in~\eqref{eq:5}, in turn, are given by
\begin{align}
  \label{eq:6}
  \quad c_0 = -\frac{1}{2\pi}\int_0^{2\pi} \log(\zeta(t)-z_0)\phi(t)\zeta'(t) \de t\andtext  c_l = \frac{1}{2\pi l} \int_0^{2\pi} \frac{\phi(t)}{(\zeta(t)-z_0)^l} \zeta'(t)\de t,
\end{align}
and are computed by the trapezoidal rule on a uniform oversampled grid with $\beta N=\beta (2\pi/h)$ points ($\beta>1$). We use one expansion center per discretization point and the expansion center distance to the curve ($\epsilon$) is selected to be proportional to the local distance between the discretization points ($d$) in physical space. That is, the distance to the curve $\epsilon$ of the expansion center $z_{0,k}$ associated to the discretization point $\nex_k=\bnex(kh)$, $0\leq k\leq N$, is proportional to $d=\min\{|\nex_{k}-\nex_{k-1}|,|\nex_{k}-\nex_{k+1}|\}.$ 

 A straightforward comparison between QBX and the HDI method becomes somewhat difficult due to the selection of the various  parameters involved. As shown by the analysis in \cite{Barnett:2014tq,klockner2013quadrature}, the error in the QBX discretization is governed by a subtle balance between the order of the expansion ($p$), the distance of the expansion centers from the curve ($\epsilon$), and the oversampling ratio ($\beta$). Figure~\ref{fig:SL_orders_QBX} presents a comparison between the two methods where the errors in the evaluation of the single-layer operator are displayed for various choices of the QBX $\epsilon$ parameter, keeping the order $p$ and the oversampling ratio $\beta$ constant. To make the comparison fair, an oversampled grid with $\beta N$ points is also utilized in the HDI single-layer evaluation.  Although not necessarily optimal, the oversampling ratio $\beta = 4$ and the harmonic interpolation order $M=2$ are used in both examples. We point out here, that since our technique requires the computation of high-order derivatives along the curve (that is the interpolation coefficients $c^{(M)}$ in equation~\eqref{eq:f_N} depend on derivatives of both the parametrization and of the density up to order $M$), we have observed that in practice it is advisable to use harmonic interpolation orders in the range $M \lesssim 6$. Clearly, even for $M = 2$ (i.e., a fifth-order HDI method), the observed errors are comparable in practice to those of the QBX method for small  and large values of the parameter~$p$.
\begin{figure}[h!]
\centering	
 \subfloat[][Trapezoidal rule convergence for HDI
with  $M=2$ and QBX with  $p=5$, both of which correspond to fifth order methods.]{\includegraphics[width=0.48\textwidth]{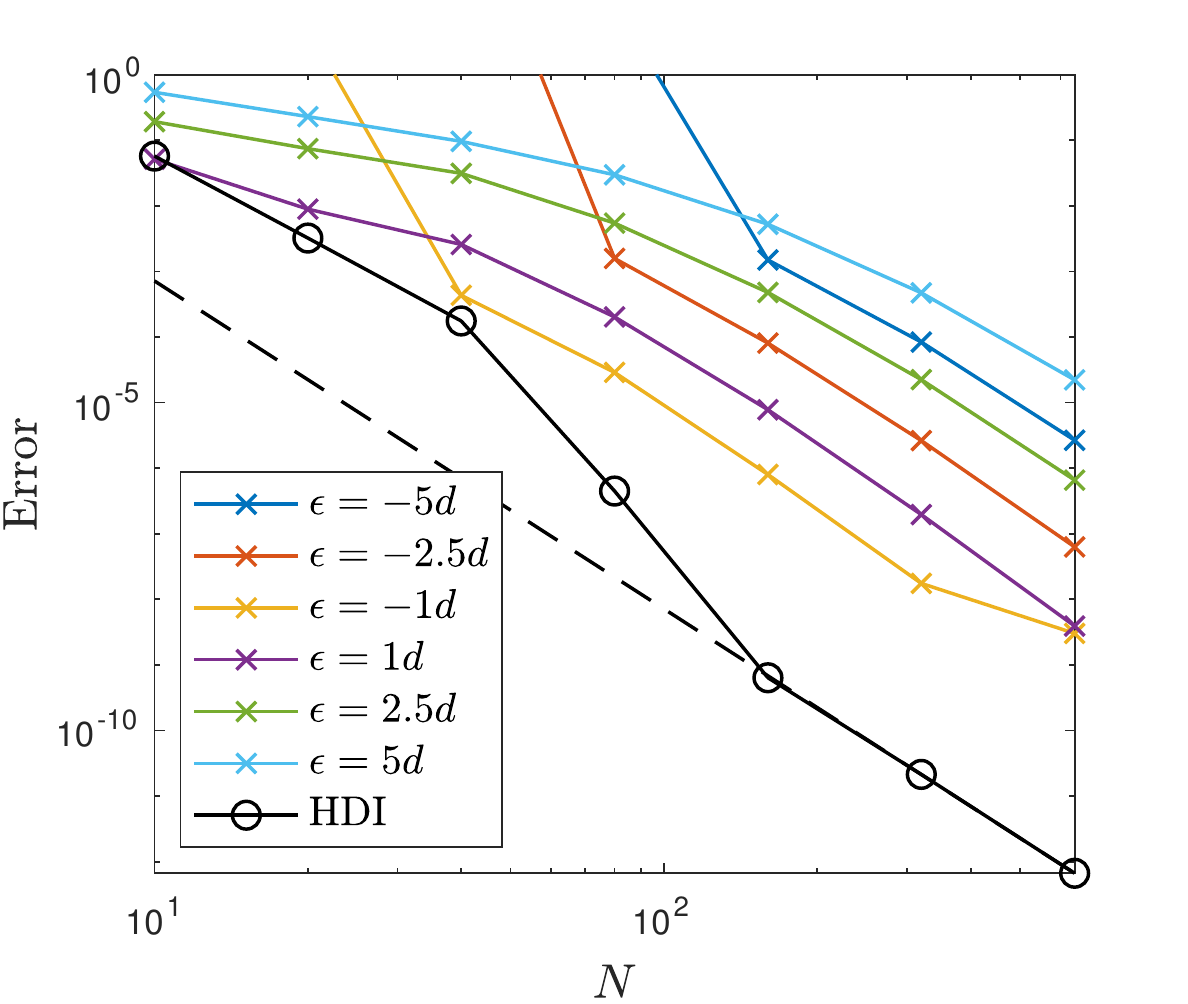}\label{fig_S_even}}\quad
 \subfloat[][Trapezoidal rule convergence for HDI with
 $M=2$ and  QBX with $p=16$, corresponding formally to a 16th order method.]{\includegraphics[width=0.48\textwidth]{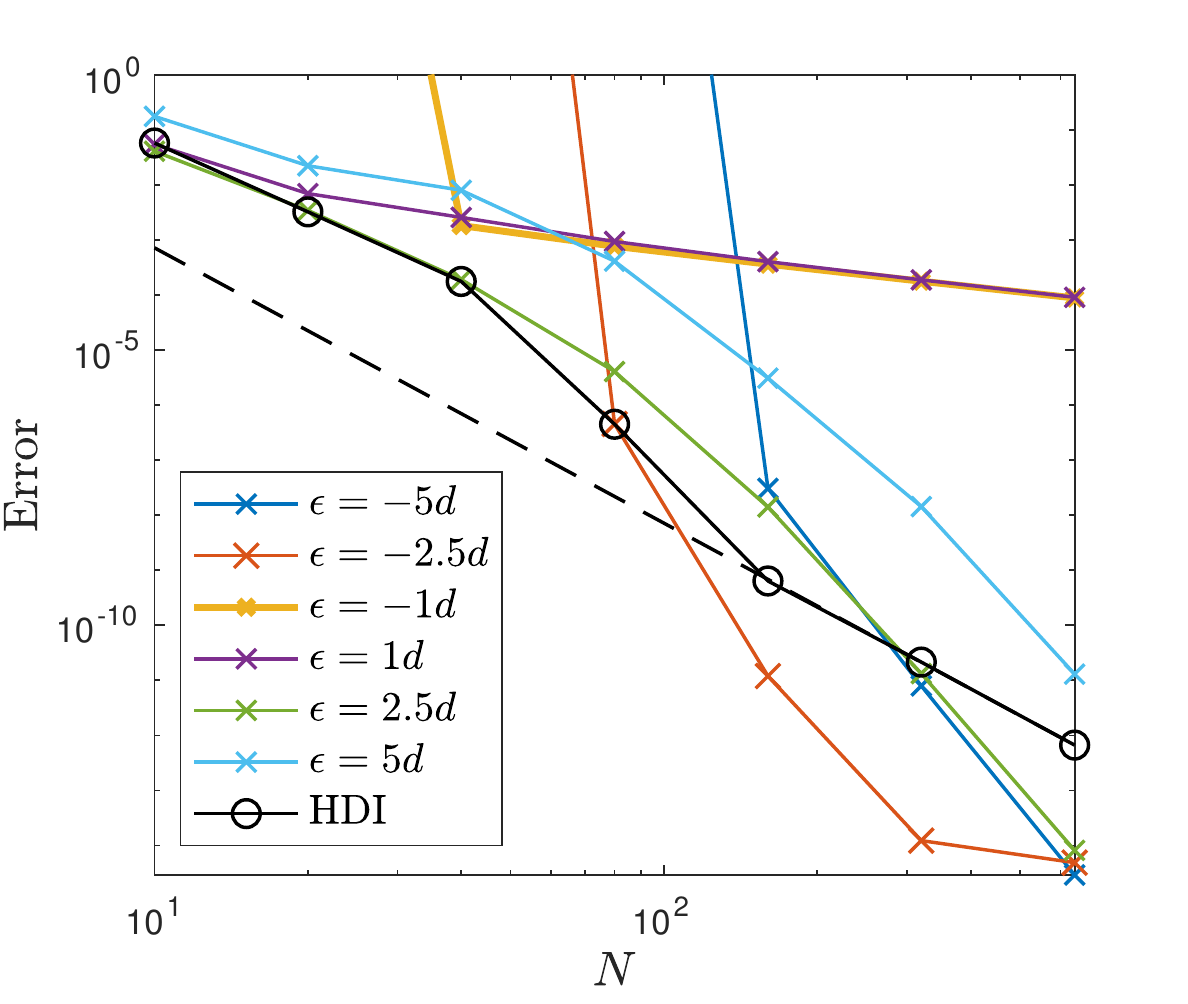}\label{fig_S_odd}}\\
\caption{Convergence of the trapezoidal rule discretization of the single-layer operator $\tilde S[\phi]$ by means of the HDI method and the QBX method for various distances $\varepsilon$ of the expansion centers. The density function utilized in these examples is $\varphi = u|_{\Gamma}$ with $u(\nex) = \e^{\sin(x_1\cos x_2)}/\sqrt{x_1^2+x_2^2}$ and $\Gamma = [\cos(t)+0.65\cos(2t)-0.65, 1.5\sin(t)]$. The dashed lines indicate the fifth-order slope.}\label{fig:SL_orders_QBX}
\end{figure}

A more thorough comparison of the two methods is presented in Table \ref{tab:QBX-vs-HOSS}, where we report the smallest error obtained using the QBX method by scanning the order parameter $p$ between $1$ and $20$, and the distance of the expansion centers to the curve $\epsilon$ between $-5d$ and $5d$ by increments of~$d$,  for four representative $\beta$ values. We then compare these errors to those obtained using the HDI method for $M=2$ (fifth-order). In conclusion, for a large range of discretization sizes ($N$) the observed HDI errors are comparable to those of the QBX method even when using optimized values of $p$ and $\epsilon$.

 \begin{table}
   \begin{center}
     \scalebox{0.95}{\begin{tabular}{c|c|c|c|c|c|c}
\toprule
$N$& \multicolumn{4}{c|} {QBX}  & \multicolumn{2}{c} {HDI} \\
\cline{2-7	}
$=\frac{2\pi}{h_{\phantom{A}}}$&$\beta=1$&$\beta=2$& $\beta=4$ & $\beta=8$ & $M=2$ & $M=2,\beta=4$ \\
\hline
10&1.7992e-01 &6.5711e-02&9.1543e-03 &2.1013e-03& 2.4120e-01&9.5925e-04\\
20&5.8644e-02&3.8506e-03&1.4650e-03&3.4615e-05&
6.3149e-02&1.1084e-05\\
40&  5.3949e-03&5.7387e-04 &2.7260e-05 &6.6149e-09&2.7396e-03&5.8703e-07\\
80&8.2720e-04  &1.9999e-05 &2.7393e-09&4.7841e-11&2.1551e-05&1.7826e-08\\
160& 1.5315e-05 &2.9372e-09 &1.4481e-12 &2.7478e-15&6.6702e-07&6.3816e-10\\
320&1.7909e-07 &1.9874e-12&2.5535e-15 &1.9984e-15&2.2100e-08 &2.1436e-11\\  
 \bottomrule
\end{tabular}}
\caption{\label{tab:QBX-vs-HOSS} Maximum errors in single-layer operator applied to the density $\varphi = u|_{\Gamma}$ with $u(\nex) = e^{x_2\sin(x_1+5)}/\sqrt{x_1^2+x_2^2}$, computed using QBX and HDI. For each column of QBX fix the oversampling ratio $\beta$, and display the minimum error found by varying $1 \leq p\leq 20$ and $-5d \leq \epsilon \leq 5d$ in increments of $d$ where the $d$ denotes the local distance between . }
\end{center}
 \end{table}}

\subsection{Nearly singular integrals in 2D}
In our next example we consider the numerical evaluation of the
single- and double-layer potentials and their gradients inside the
domain $\Omega$ enclosed by the curve
$\Gamma = \{(\cos t, \sin t/(1+\sin^6(t)),t\in[0,2\pi]\}$. The errors
are measured with respect to a manufactured  solution of the Laplace equation
produced by taking point sources at  $\nex_1 = (-0.6, 1)$,
$\nex_2= (-1.2, 0.2)$, $\nex_3 = (0.2, -1.1)$ and
$\nex_4 = (1.5, 0.2)$ that lie outside~$\Omega$. By construction
$u_{\mathrm{exact}}(\nex) =\sum_{j=1}^4 \log(|\nex-\nex_j|)$ is harmonic in~$\Omega$. 

To test the accuracy of the
evaluation of the double- and single-layer potentials and their
gradients, we first find densities $\varphi:\Gamma\to\R$ and $\psi:\Gamma\to\R$ to represent $u_\mathrm{exact}$ by means of the double- and single-layer potentials~\cref{eq:potentials}. 
 Using the double-layer representation $u_{\mathrm{exact}}  = \mathcal D[\varphi]$ in $\Omega$,  we readily obtain the
second-kind integral equation
$(-I/2+K)\varphi = u_{\mathrm{exact}}|_{\Gamma}$ for~$\varphi$. Similarly, using a single-layer representation $u=\mathcal S[\psi]$, we
obtain the first-kind integral equation
$S[\psi] = u_{\mathrm{exact}}|_{\Gamma}$ for~$\psi$. 

Both surface densities $\varphi$ and $\psi$ are then computed by means of
a spectrally accurate Nystr\"om method~\cite{MARTENSEN:1963,KUSSMAUL:1969} using a fixed number of points
$2N=200$ in the discretization of~$\Gamma$. For this discretization,
the resulting approximate densities exhibit maximum absolute errors smaller that~$10^{-11}$. 
The corresponding potentials and their gradients are then evaluated everywhere inside
$\Omega$ by means of direct application of  the trapezoidal rule
to the integral expressions~\cref{eq:single_layer_pot} and \cref{eq:double_layer_pot} for the regularized single- and double-layer potentials, respectively.  The logarithm in base ten of the absolute errors in the evaluation of the single-layer potential and its gradient for $M=0,2$ and $4$ are displayed in~\Cref{fig:SL_near}, while the error plots corresponding to the double-layer potential and its gradient are displayed in~\Cref{fig:DL_near}.
 As demonstrated in these figures, the HDI technique
reduces significantly the numerical errors at observation points that are close the boundary.

\begin{figure}[h!]
\centering	
 \subfloat[][Without kernel regularization, $E=9.01\cdot 10^{-1}$.]{\includegraphics[scale=0.08]{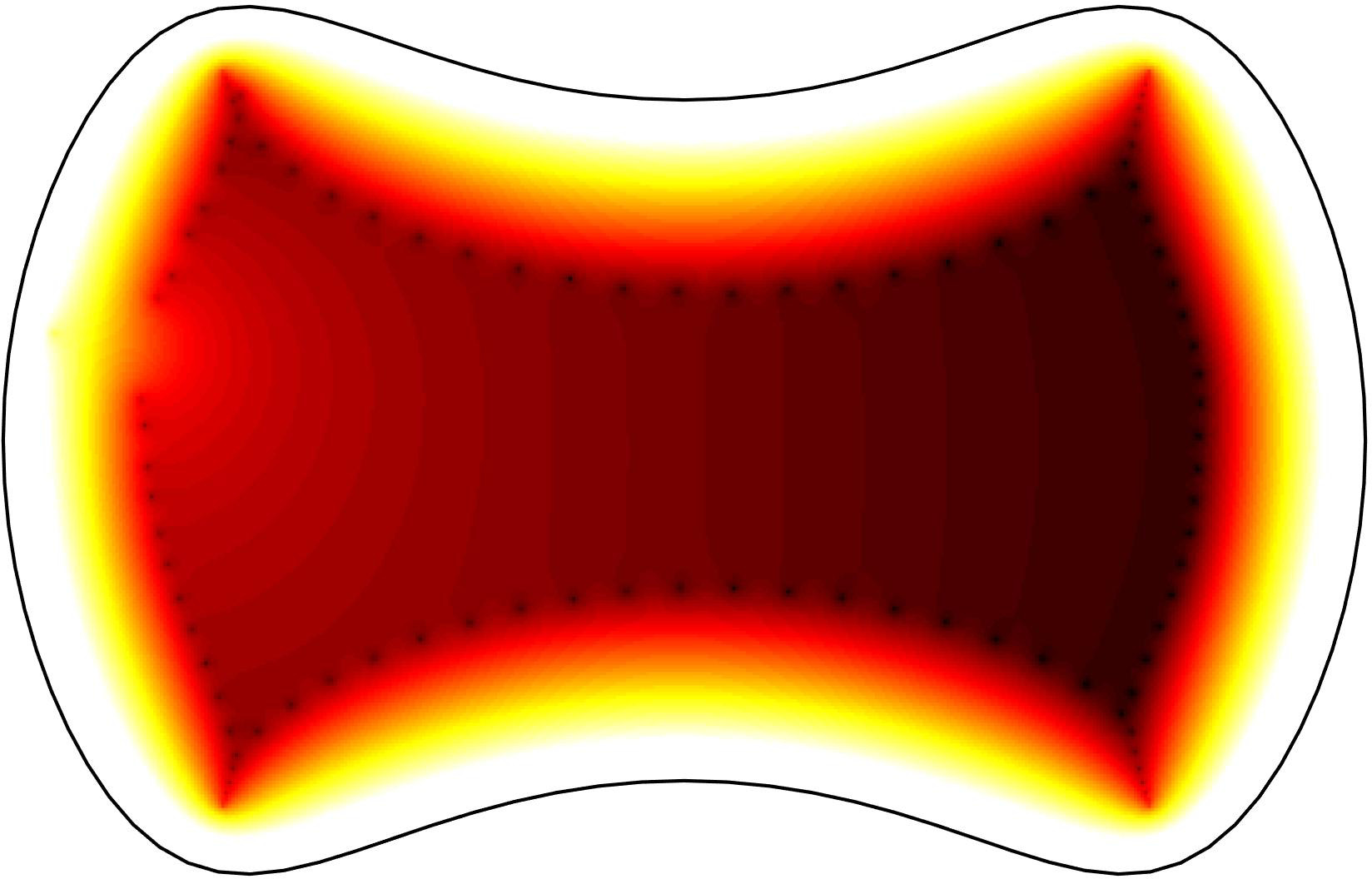}\label{fig_S_none}}\quad
 \subfloat[][$M=0$, $E=6.01\cdot 10^{-3}$.]{\includegraphics[scale=0.08]{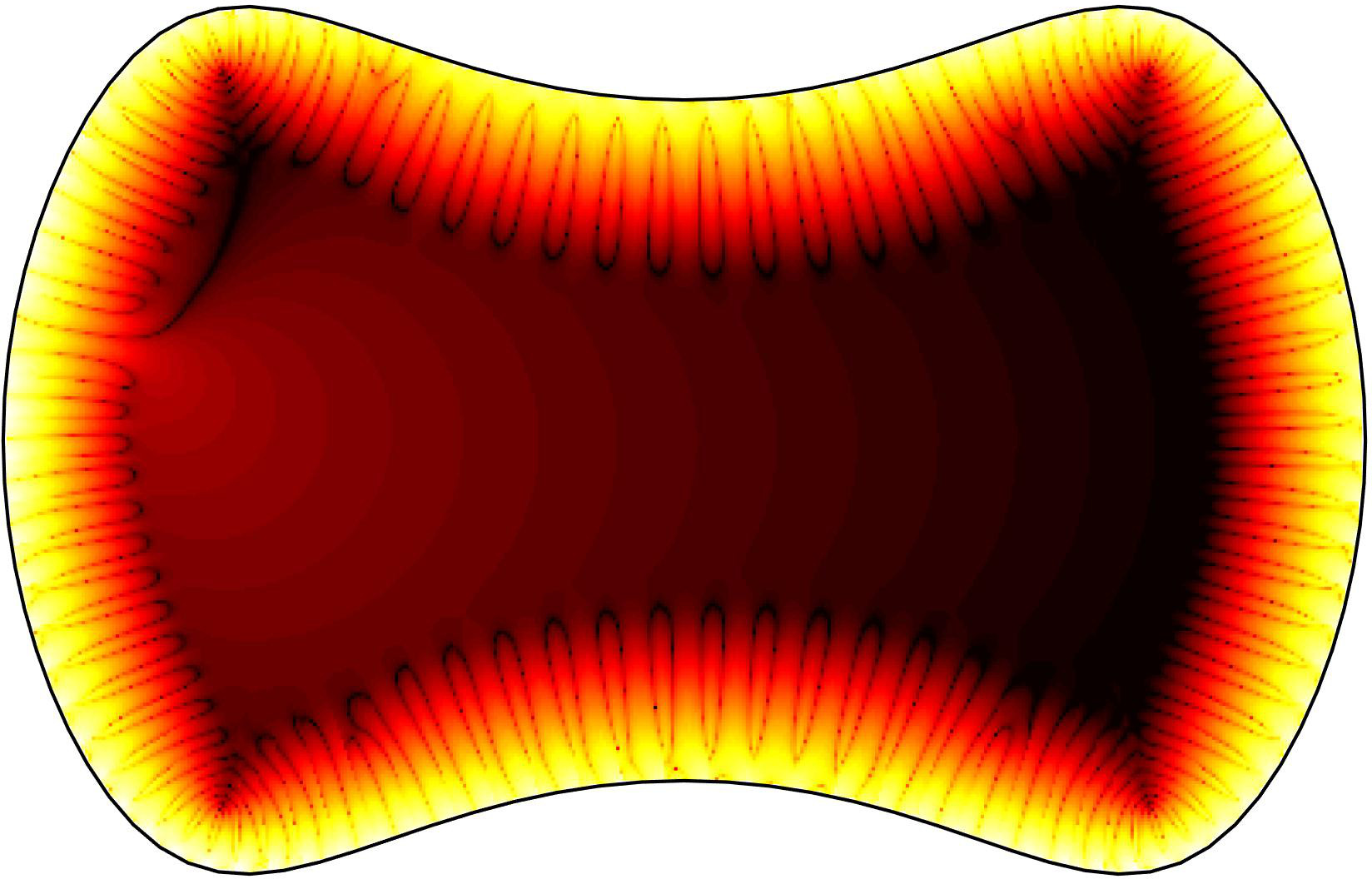}\label{fig_S_1}}\quad
 \subfloat[][$M=4$, $E=7.23\cdot 10^{-6}$]{\includegraphics[scale=0.08]{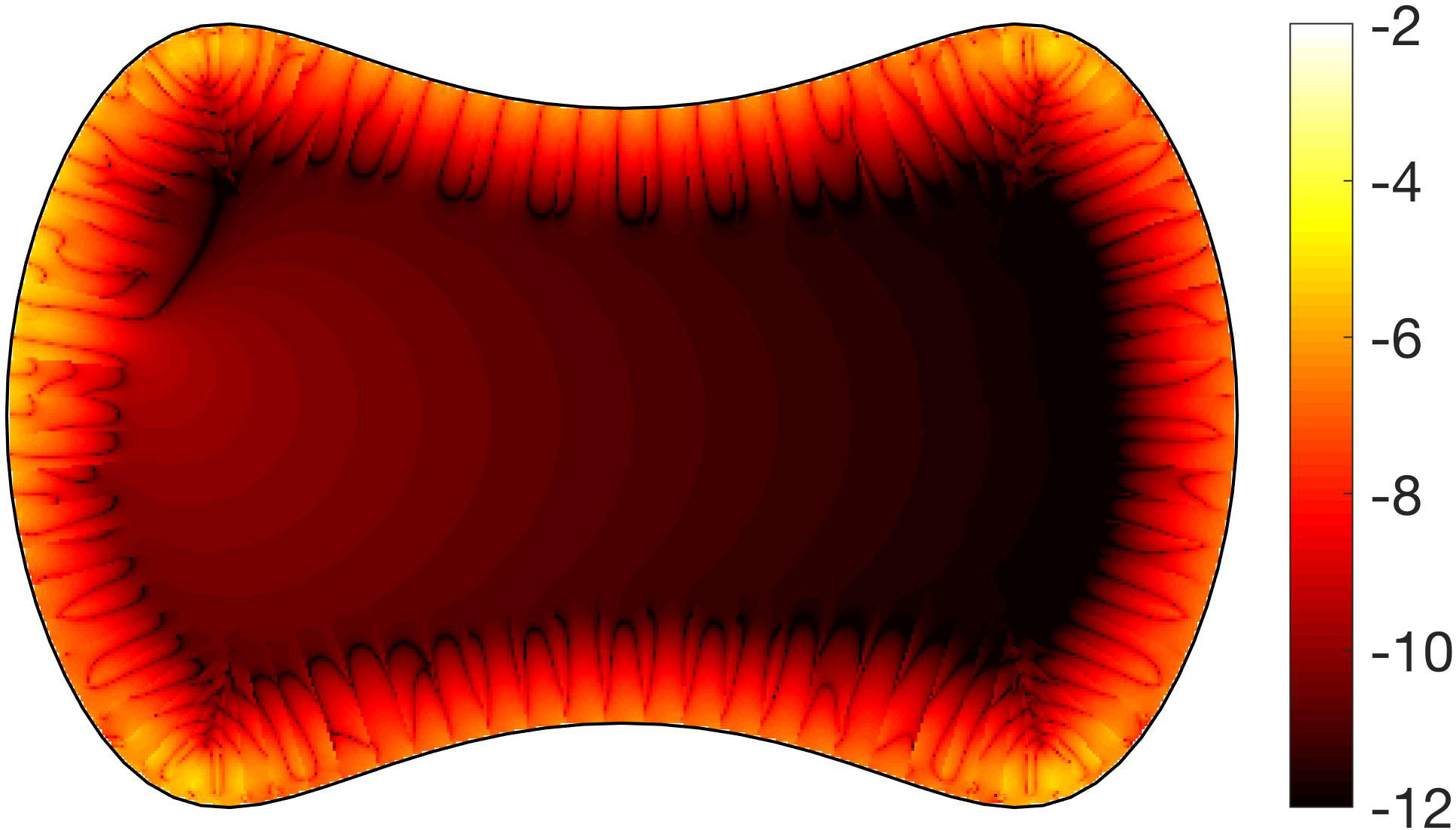}\label{fig_S_5}}\\
 \subfloat[][Without kernel regularization, $E = 2.7\cdot 10^{1}$.]{\includegraphics[scale=0.08]{grad_SL_none_v0}\label{fig_grad_S_none}}\quad
 \subfloat[][$M=0$, $E = 5.54\cdot 10^{-1}$.]{\includegraphics[scale=0.08]{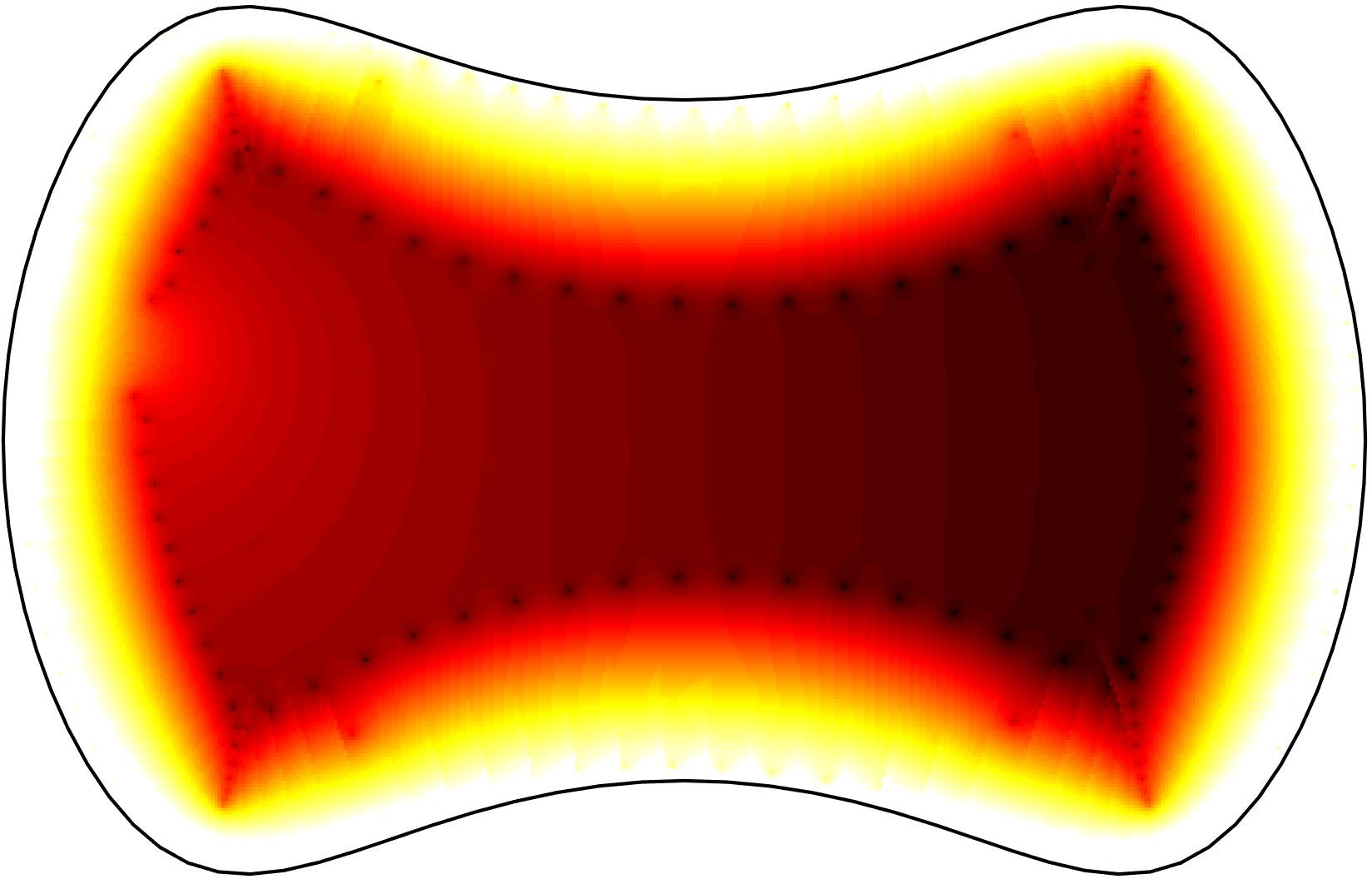}\label{fig_grad_S_0}}\quad
 \subfloat[][$M=4$, $E=6.61\cdot 10^{-4}$.]{\includegraphics[scale=0.08]{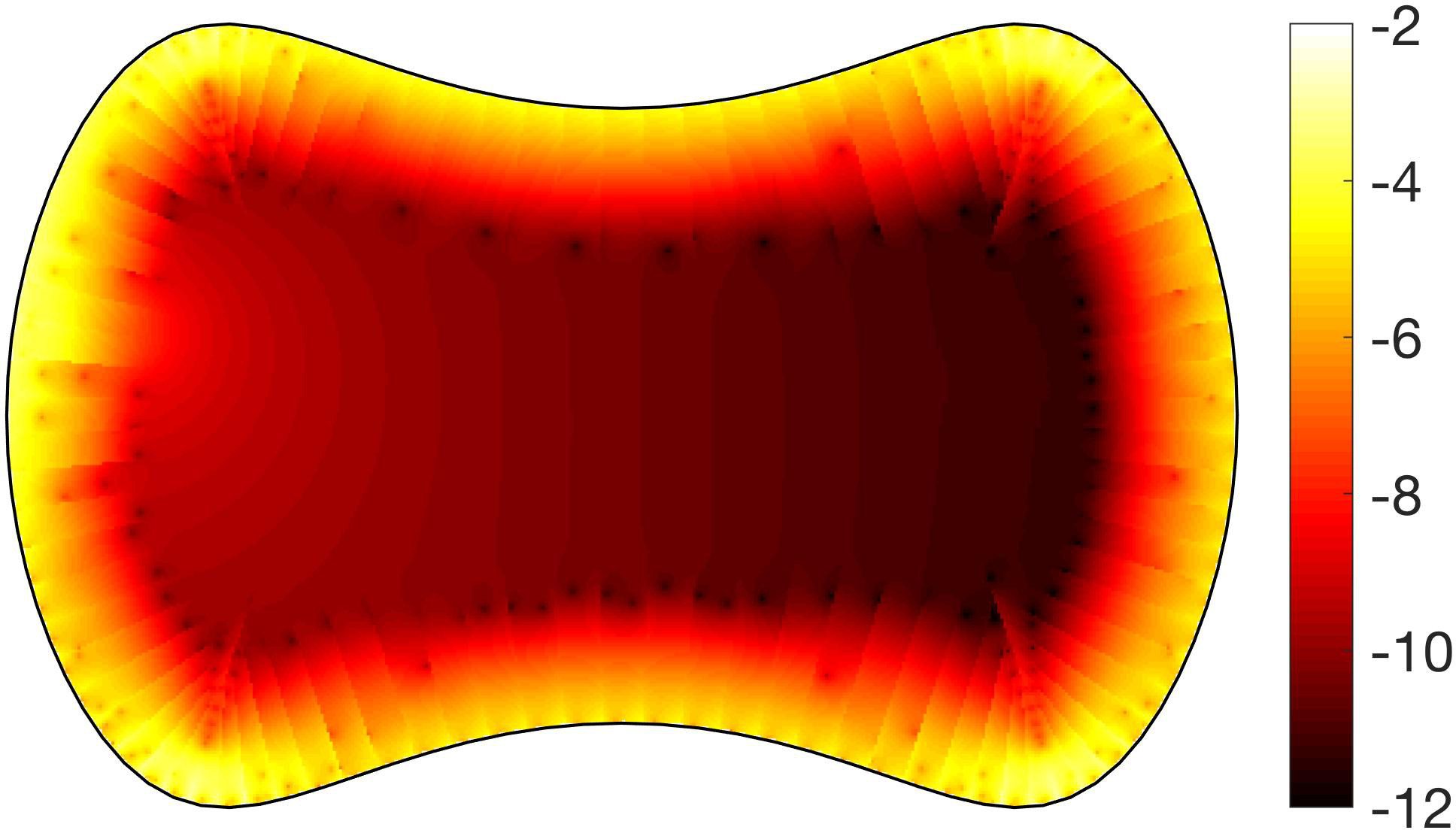}\label{fig_grad_S_4}}
\caption{Logarithm in base ten of the absolute error in the
  evaluation of the single-layer potential (top row) and its gradient (bottom row). The maximum absolute error $E$ is indicated in the caption corresponding to each plot. HDI is used for all observation points at a distance smaller than $10h=\pi/10$ from the boundary.}\label{fig:SL_near}
\end{figure}  

\begin{figure}[h!]
\centering	
\subfloat[Without kernel regularization, $E=1.38\cdot 10^1$.]{\includegraphics[scale=0.08]{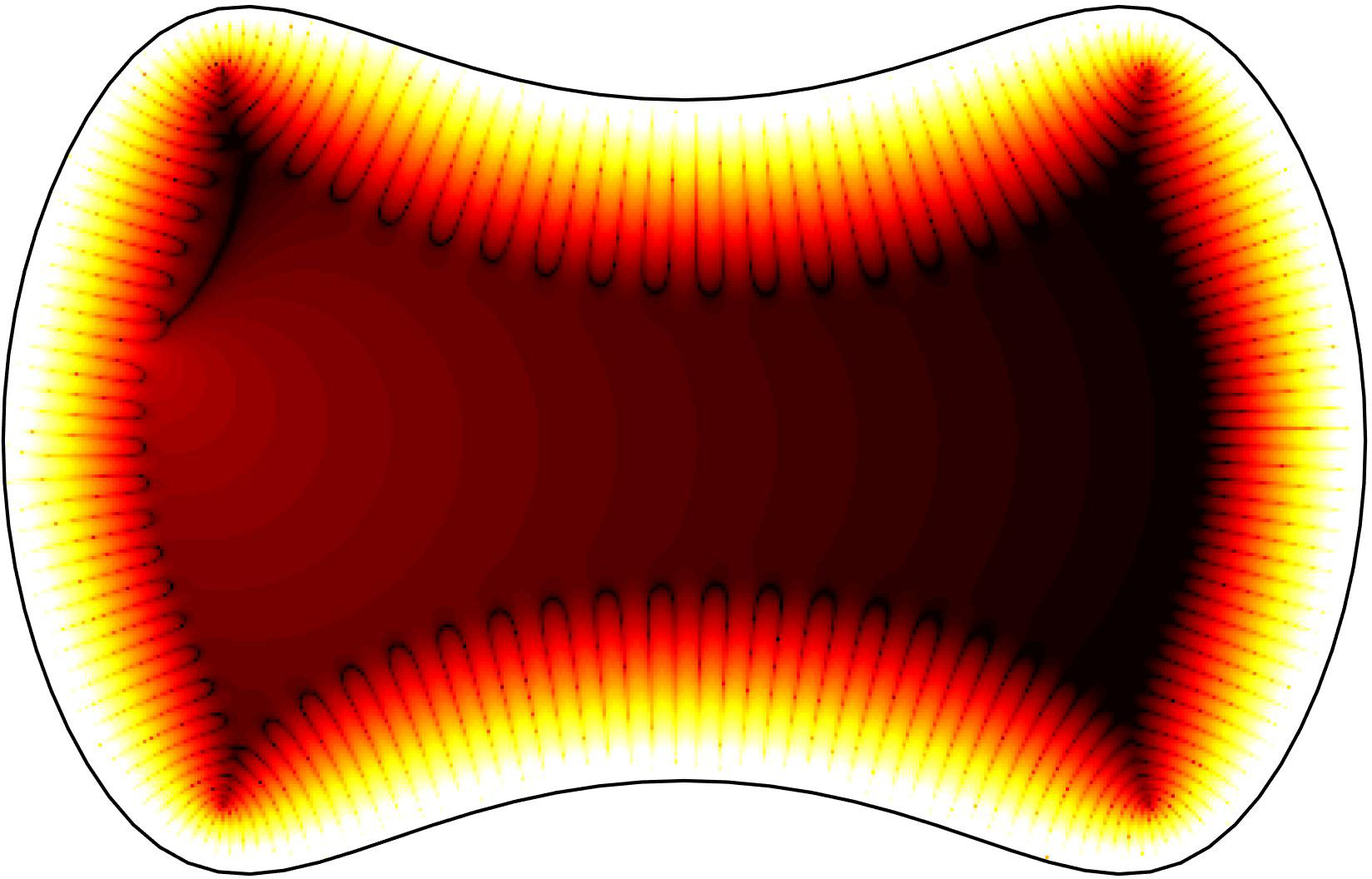}\label{fig_D_none}}\quad
\subfloat[$M=0$, $E=6.35\cdot 10^{-2}$.]{\includegraphics[scale=0.08]{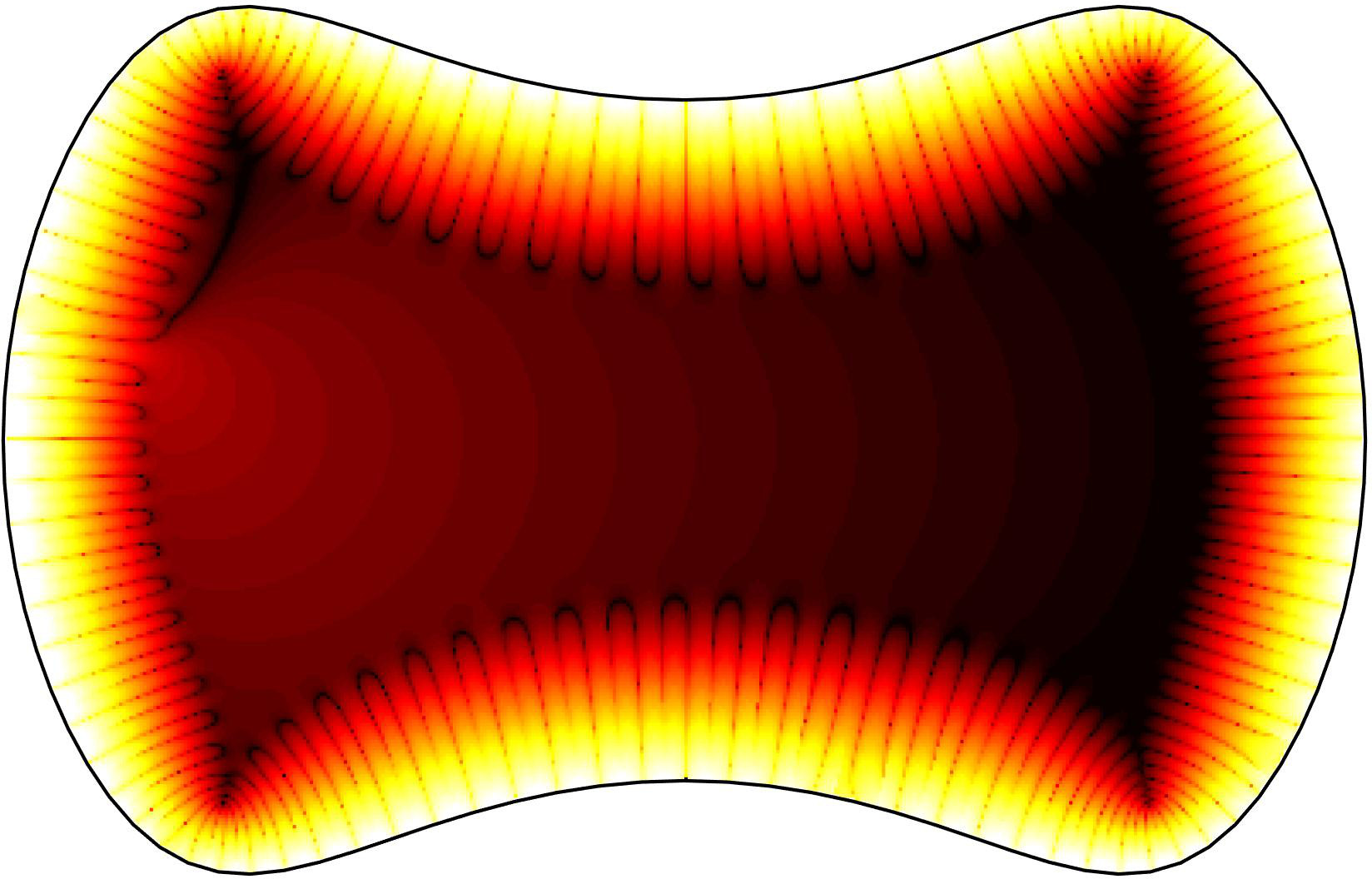}\label{fig_D_0}}\quad
\subfloat[$M=4$, $E=1.54\cdot 10^{-5}$.]{\includegraphics[scale=0.08]{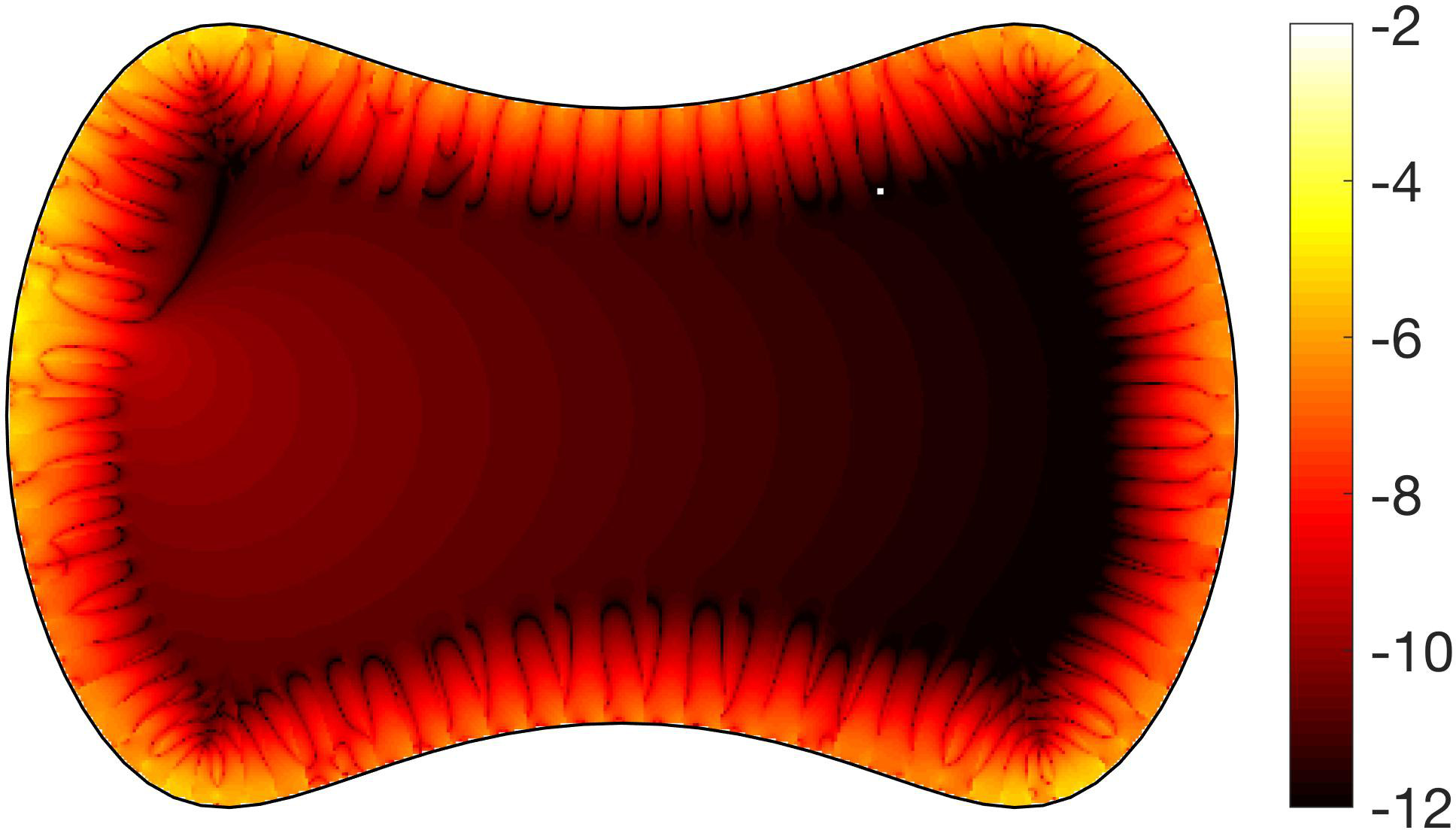}\label{fig_D_4}}\\
 \subfloat[Without kernel regularization, $E=3.05\cdot 10^{3}$.]{\includegraphics[scale=0.08]{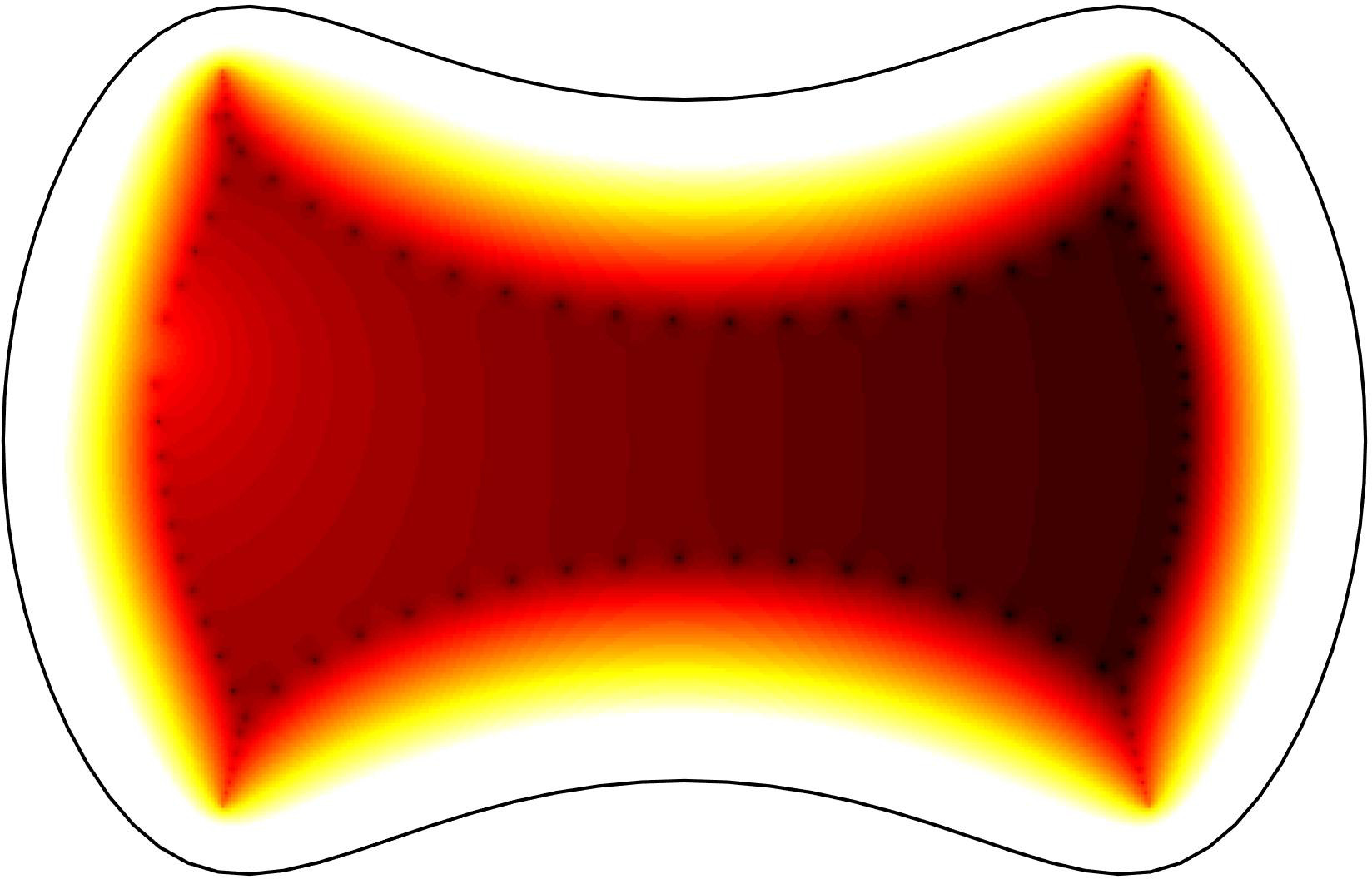}\label{fig_grad_D_none}}\quad
 \subfloat[$M=0$, $E=2.00\cdot 10^2$.]{\includegraphics[scale=0.08]{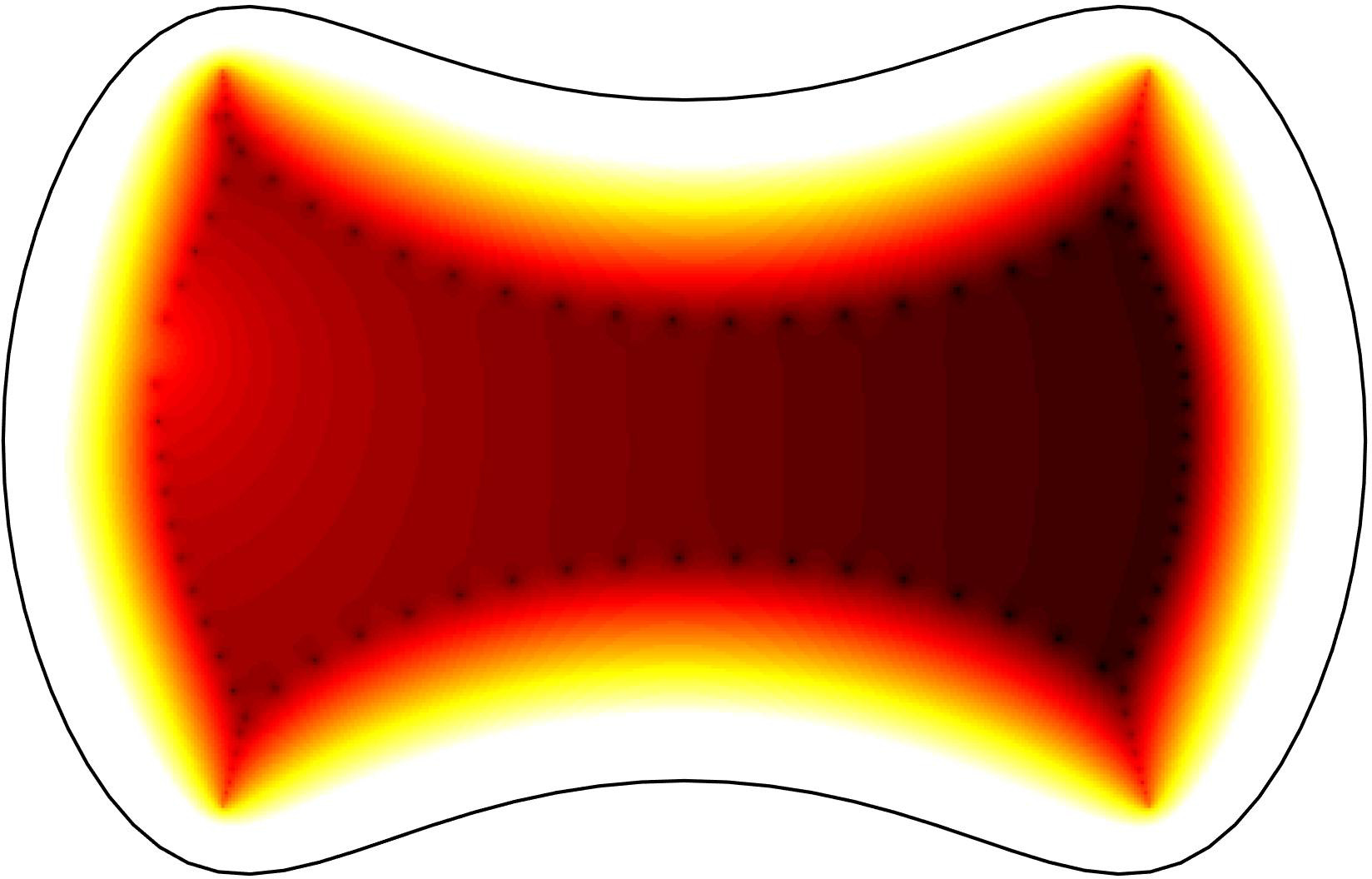}\label{fig_grad_D_0}}\quad
 \subfloat[a][$M=4$, $E=5.38\cdot 10^{-4}$.]{\includegraphics[scale=0.08]{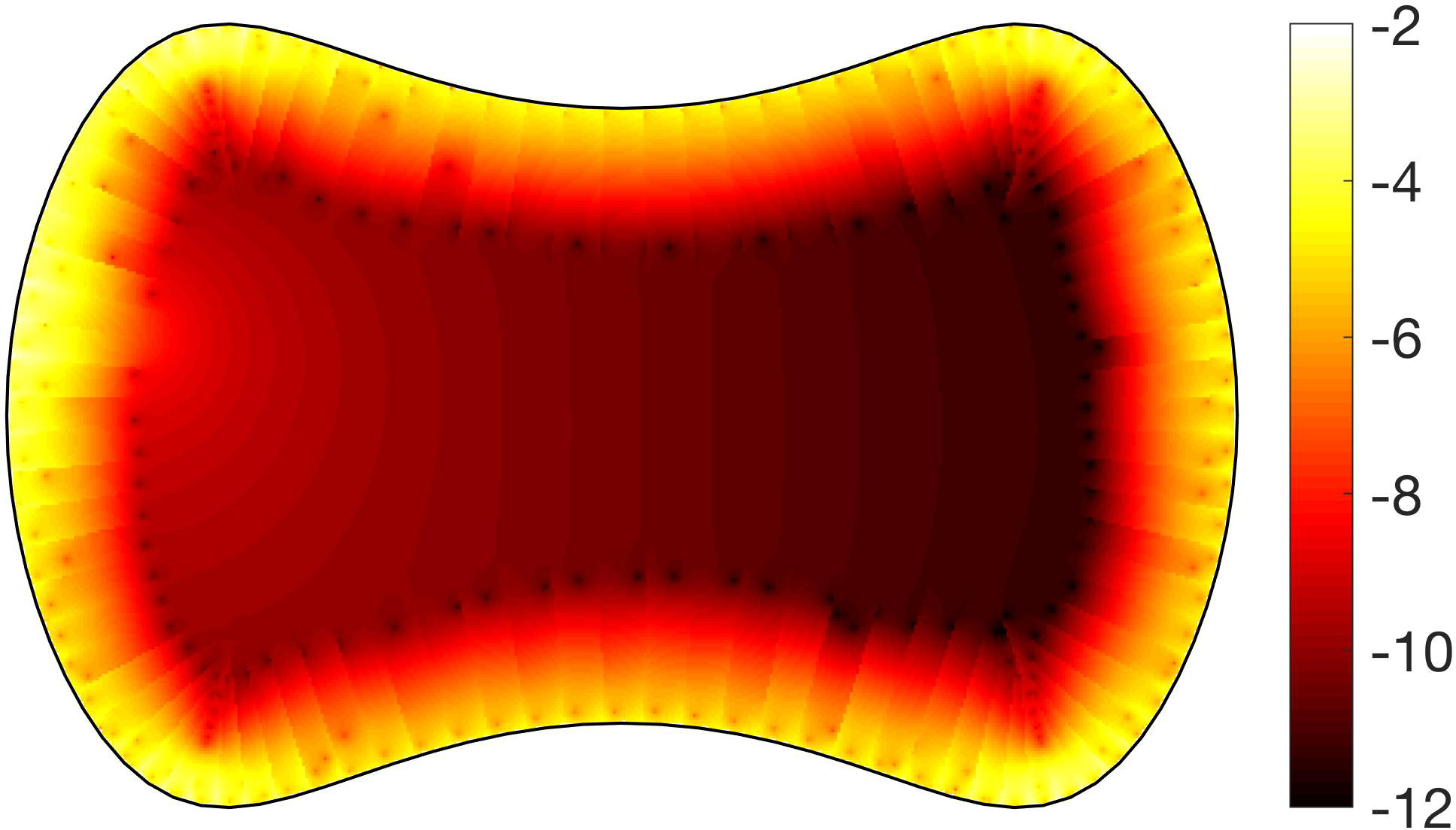}\label{fig_grad_D_4}}
 \caption{Logarithm in base ten of the absolute error in the
   evaluation of the double-layer potential (top row) and its gradient
   (bottom row). The maximum absolute error $E$ is indicated in the
   caption corresponding to each plot. HDI method is used
   for all observation points at a distance smaller than $10h=\pi/10$
   from the boundary.}\label{fig:DL_near}
\end{figure}

\subsection{Electrical response of closely packed biological cells}
Finally, this section considers an application of the HDI method
to the computation of the electrostatic potential in presence of
closely packed cells. A detailed formulation of this problem, which
has applications in gene transfection, electrochemotherapy of tumors,
and cardiac defibrillation, is presented in the recent
contribution~\cite{ying2013fast}. A third-order boundary integral
equation method for the numerical solution of this challenging
problem---based on the previous work~\cite{beale2001method}---is also presented in~\cite{ying2013fast}.  Here we
compare the accuracy of our approach with the method presented
in those references. To this end we consider a
benchmark problem consisting of $20$
elliptical cells whose centers, semi-axes, and orientation angles are
given in~\cite[Table~4]{ying2013fast} (see
also~\Cref{fig:pot_ellls_eval} below).

In detail, in this application we look for an electrostatic potential given by 
\begin{equation}
\Phi(\nex) = -\mathcal D[v](\nex)+\mathcal S[q](\nex)-\bold E\cdot\nex,\quad\nex\in\R^2\setminus\Gamma,\label{eq:elst_pot}
\end{equation}
in terms of the constant electric field $\bold E=(1,0)$ and the single- and double-layer potentials defined in~\cref{eq:potentials}. Here $\Gamma$ denotes the multiply connected curve  $\Gamma=\bigcup_{k=1}^{20}\Gamma_k$ where $\Gamma_k$, $k=1,\ldots,20$, are the boundaries of each individual elliptical cell. Letting $\mu = (\sigma_e-\sigma_i)/ (\sigma_e+\sigma_i)$ with $\sigma_i=1$ and $\sigma_e=2$ denoting  the electric conductivities of the interior and exterior domains, respectively, we have that the unknown charge density $q$ in~\cref{eq:elst_pot} is given by the solution of the following second-kind integral equation
\begin{equation}
\lf(\frac{I}{2}-\mu K'\rg)[q](\nex) = -\mu N[v](\nex)-\mu \bold E\cdot n(\nex) + j(\nex),\quad\nex\in\Gamma,\label{eq:IE_cells}
\end{equation}
where $n$ is the unit normal to $\Gamma$ and where the functions $v$  and $j$ are assumed known (see~\cite{ying2013fast} for details). 

Clearly, evaluation of the adjoint double-layer ($K'$) and
hypersingular ($N$) operators in~\cref{eq:IE_cells} involve
integration over each one of the curves $\Gamma_k$, $k=1,\dots,20$. In
view of~\Cref{lem:cond_1}, when evaluated on a curve $\Gamma_{k}$, the
$K'$ integrand on $\Gamma_k$ is a smooth function, so no  density interpolation is required. The $N$ integrand over $\Gamma_k$, in turn,
requires a  density interpolation order $M\geq 1$ for it to become a
smooth real analytic function. Now, the $K'$ and $N$ integrands over
$\Gamma_{k'}$, $k'\neq k$, are nearly singular if $\Gamma_{k}$ is
``close" to $\Gamma_{k'}$. In this case we evaluate the nearly
singular integrals utilizing the smoothing procedure presented
in~\Cref{sec:eval-layer-potent}.

\Cref{tab:comparison} displays the maximum absolute errors in the
charge density $q$ obtained by means of three different BIE methods,
namely; the HDI kernel regularization method  used in
conjunction with the trapezoidal rule; the third-order
 kernel regularization method of Beale, Lai, and Ying (BLY) introduced in
references~\cite{beale2001method,ying2013fast}; and the spectrally
accurate method of Kress~\cite{kress2014collocation}\footnote{Note that ``method of Kress" here refers to the high-order method for evaluation of the hypersingular operator via trigonometric interpolation~\cite{kress2014collocation}, not the Martensen-Kussmaul method described in the~book~\cite{COLTON:2012}.} (which is only
used for evaluation of the hypersingular operator, while all other
relevant integrals were directly approximated by the
trapezoidal rule). The latter is considered here for reference and in
order to highlight the importance of properly treating nearly singular
integrals. The errors corresponding to the BLY method were taken
directly from~\cite[Table 5]{ying2013fast} for the parameter values
$\gamma=3$ and $C=4$ in that reference. Clearly, for $M=2$ our
approach matches the accuracy of the BLY method and for $M>2$ the
HDI  approach is substantially more accurate than the BLY method
for the problem considered.  The improvements as the harmonic interpolation order $M$ increases are evident. 
 \begin{table}
   \begin{center}
     \scalebox{0.95}{\begin{tabular}{c|c|c|c|c|c|c|c}
\toprule
$2N$& \multicolumn{5}{c|} {HDI}  & \multicolumn{1}{c|} {BLY} & \multicolumn{1}{c} {K} \\
\cline{2-6}
$=\frac{2\pi}{h_{\phantom{A}}}$&$M=1$&$M=2$& $M=3$ & $M=4$ &$M=5$  & \\
\hline
64&$2.85\cdot 10^{-2}$&$5.64\cdot 10^{-3}$&$4.71\cdot 10^{-4}$&$ 1.00\cdot 10^{-4}$&$1.84\cdot 10^{-5}$&$1.25\cdot 10^{-3}$&$\ 2.75\cdot 10^{+1^{\phantom{1}}}$\\
128&$4.74\cdot 10^{-3}$ &$2.52\cdot 10^{-4}$ &$2.61\cdot 10^{-5}$  &$ 2.69\cdot 10^{-6}$ & $1.64\cdot 10^{-7}$  &$1.88\cdot10^{-4}$  &$2.48\cdot 10^{+1}$\\
256&$5.70\cdot10^{-4}$  &$1.73\cdot 10^{-5}$ &$1.12\cdot 10^{-6}$  &$ 8.19\cdot 10^{-8}$ & $8.09\cdot 10^{-9}$  &$2.89\cdot10^{-5}$  &$4.26\cdot 10^{+0}$\\
512&$8.76\cdot 10^{-6}$ &$4.36\cdot 10^{-7}$ &$1.24\cdot 10^{-8}$ &$ 3.46\cdot 10^{-10}$& $4.67\cdot 10^{-11}$&$3.48\cdot 10^{-6}$ &$8.23\cdot 10^{-2}$\\
 \bottomrule
\end{tabular}}
\caption{\label{tab:comparison} Absolute errors, measured in the maximum norm, in the charge density~$q$ which is computed by solving~\cref{eq:IE_cells} by three different methods, namely; the high-order  density interpolation (HDI) technique proposed here; Beale, Lai and Ying (BLY) method~\cite{beale2001method,ying2013fast}, and; Kress' method (K)  \cite{kress2014collocation} without regularization of nearly singular integrals. Note that this table displays the most accurate results reported in~\cite[Table 5]{ying2013fast} for the same benchmark problem and the same number of discretization points $2N$.}
\end{center}
 \end{table}
 
Finally, we show in~\Cref{fig:pot_ellls_eval} the absolute error in
 the electrostatic potential~\cref{eq:elst_pot} for various harmonic interpolation orders $M$
 but for a fixed number ($2N=64$) of discretization points (per ellipse). We can clearly see the
 improvement as $M$ increases, specially when compared to
 \Cref{fig:pot_ellls_eval}a, where no regularization of any kind is used. In fact for $M=5$ (see
 ~\Cref{fig:pot_ellls_eval}f) we obtain global errors in the order of
 $10^{-7}$ even when using only $2N=64$ discretization points per
 ellipse. We note that there are numerical methods~\cite{Barnett:2015kg} capable of producing higher accuracies than those produced by HDI for Stokes flow in the same closed-packed elliptical configurations considered in this section (actually, the singularities of BIE formulations of Stokes problems are no worse than this of the Laplace BIOs). However, the highly-performant methods in~\cite{Barnett:2015kg} rely heavily on complex-analytic methods, and as such are not extendable to three dimensional applications, a fact stated by the authors themselves while concluding their contribution. In contrast, the relatively simple HDI technique, while not the most performant method in 2D, extends easily to three dimensions, and it is capable to consistently produce levels of accuracy (e.g., $10^{-5}$) that are more than appropriate for engineering applications.

\begin{figure}[h!]
\centering	
\includegraphics[scale=0.45]{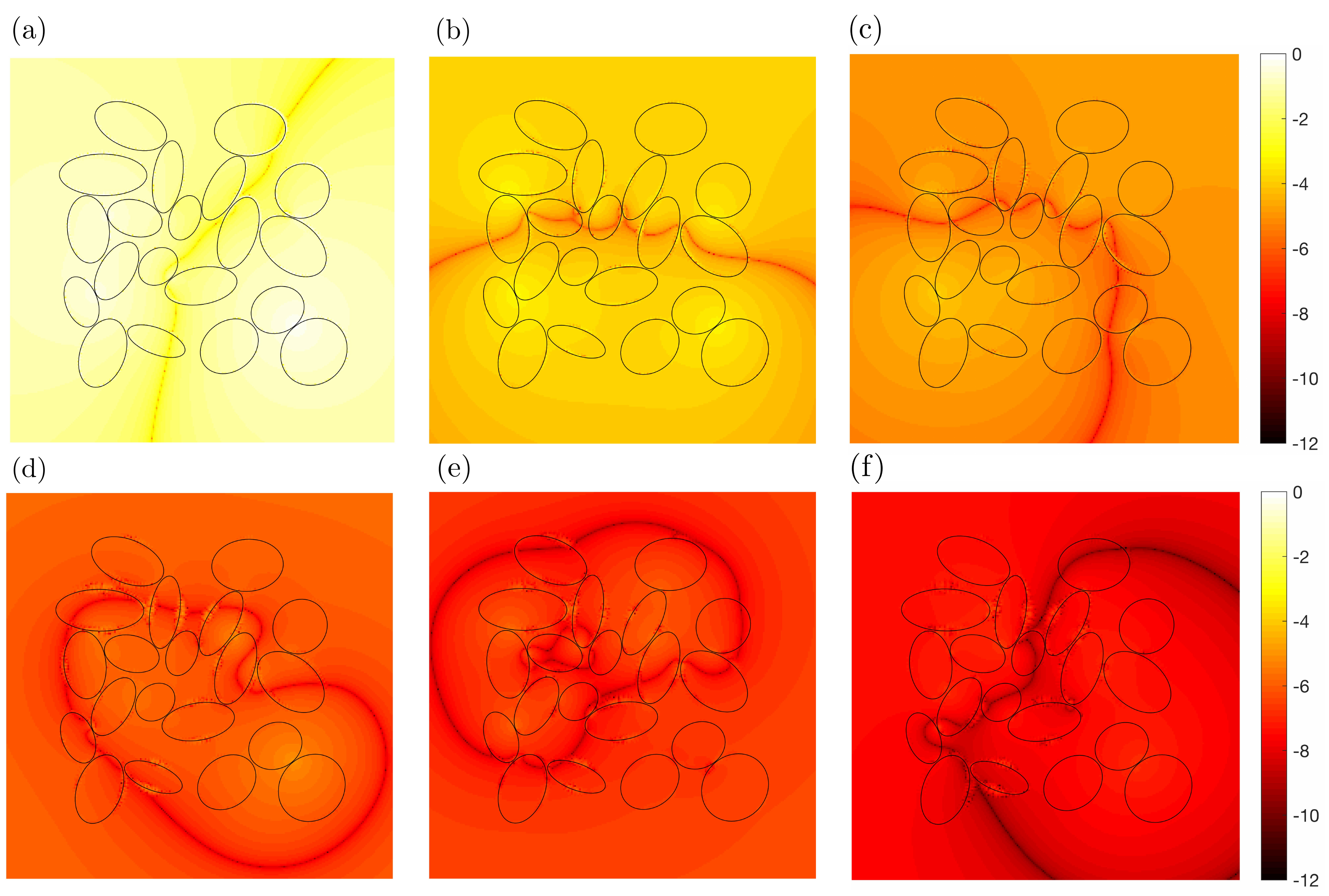}
\caption{Logarithm in base ten of the absolute error in the
  electrostatic potential~\cref{eq:elst_pot} for various harmonic interpolation orders. The same number $2N=64$ of discretization points
  is used on each one of the twenty  ellipses. (a) Errors without the HDI kernel regularization . (b)--(f); errors using HDI in both integral equations and potential evaluations for $M=1,2,3,4$ and $5$, respectively. The maximum errors displayed in figures (b)--(f) are $1.95\cdot 10^{-3}$, $1.98\cdot 10^{-4}$, $1.34\cdot 10^{-5}$, $2.83\cdot 10^{-6}$ and $7.05\cdot 10^{-7}$, respectively.}\label{fig:pot_ellls_eval}
\end{figure}

\subsection{Singular and nearly singular integrals in 3D}\label{3d_num}
In order to produce accurate numerical discretizations of the integral operators~\eqref{eq:IE_3d} we resort to non-overlapping surface representation with quadrilateral patches. To the best of the authors' knowledge, the high-order surface discretization approach described in this section was originally developed by Oscar Bruno's group at Caltech for the high-order evaluation of BIOs by means of polar and, more recently,  rectangular-polar singularity resolution techniques~\cite{bruno2018chebyshev,turc2011efficient}. In detail, the surface $\Gamma$ is represented as the union $\Gamma=\bigcup_{k=1}^{N_p} \overline{\mathcal P^k}$  of non-overlapping  patches  $\mathcal P^k$, $k=1,\dots,N_p$, where $\mathcal P^k\cap\mathcal P^{l}=\emptyset$ if $k\neq l$. Associated to each surface patch $\mathcal P^k$ there is a  bijective $\mathcal C^\infty$ coordinate map $\bnex^k:\mathcal H\mathcal\to \overline{\mathcal P^k}$,
\begin{equation}\label{eq:maps}
\bnex^k(\bxi) := \lf(x^k_1(\xi_1,\xi_2),x^k_2(\xi_1,\xi_2),x^k_3(\xi_1,\xi_2)\rg),\quad k=1,\ldots,N_p,\quad (\bxi=(\xi_1,\xi_2))
\end{equation} where the domain $\mathcal H = [-1,1]\times [-1,1]\subset\R^2$ is henceforth referred to as the \emph{parameter space}. Figure~\ref{fig_beam_param} illustrates a set of six coordinate patches that make up the surface of a bean shaped domain. The coordinate maps~\eqref{eq:maps} are selected in such a way that the unit normal
\begin{equation}
\bnor^k(\bxi) = \frac{\p_1\bnex^k(\bxi)\wedge  \p_2\bnex^k(\bxi)}{|\p_1\bnex^k(\bxi)\wedge  \p_2\bnex^k(\bxi)|}\label{eq:unit_normal_j}
\end{equation}
at the point $\bnex^k(\bxi)\in\mathcal P^k$ points outward to the surface $\Gamma$.  As described in Section~\ref{sec:3D},  the numerical evaluation of the integral operators~\eqref{eq:IE_3d} as well as the layer potentials~\eqref{eq:lay_pots} by means of the proposed  density interpolation technique requires (1) integration of functions that are at least continuous on the parameter space $\mathcal H$ and (2) numerical differentiation of smooth functions defined on the surface $\Gamma$. Using the surface parametrization, the surface integral of a sufficiently regular function $F:\Gamma\to\R$---such as the integrands on the right-hand-side of the identities~\eqref{eq:3d_HS} and~\eqref{eq:lay_pots}---is given by
$$
\int_{\Gamma}F(\nex)\de s = \sum_{k=1}^{N_p} \int_{\mathcal H}F\lf(\bnex^k(\bxi)\rg)|\p_1\bnex^k(\bxi)\wedge\p_2\bnex^k(\bxi)|\de \bxi.
$$
\begin{figure}[h!]\centering
 \subfloat[][Bean-shaped surface~\cite{Bruno:2001ima} and its parametric representation using six non-overlapping coordinate patches.]{\includegraphics[scale=1]{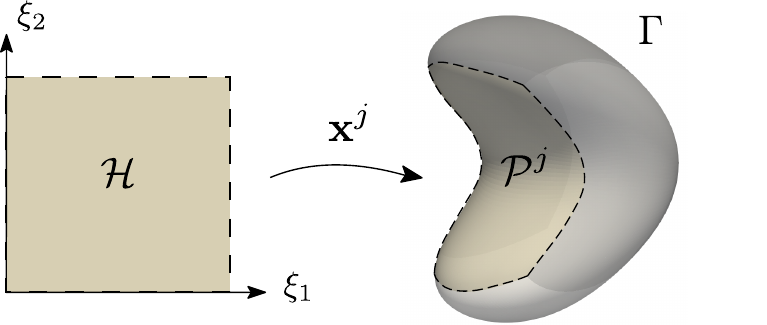}\label{fig_beam_param}}\qquad
 \subfloat[][Discretization of the bean-shaped surface produced by a $20\times 20$ Chebyshev grid~\eqref{eq:2d_grid} in each one of the six patches.]{\includegraphics[scale=0.45]{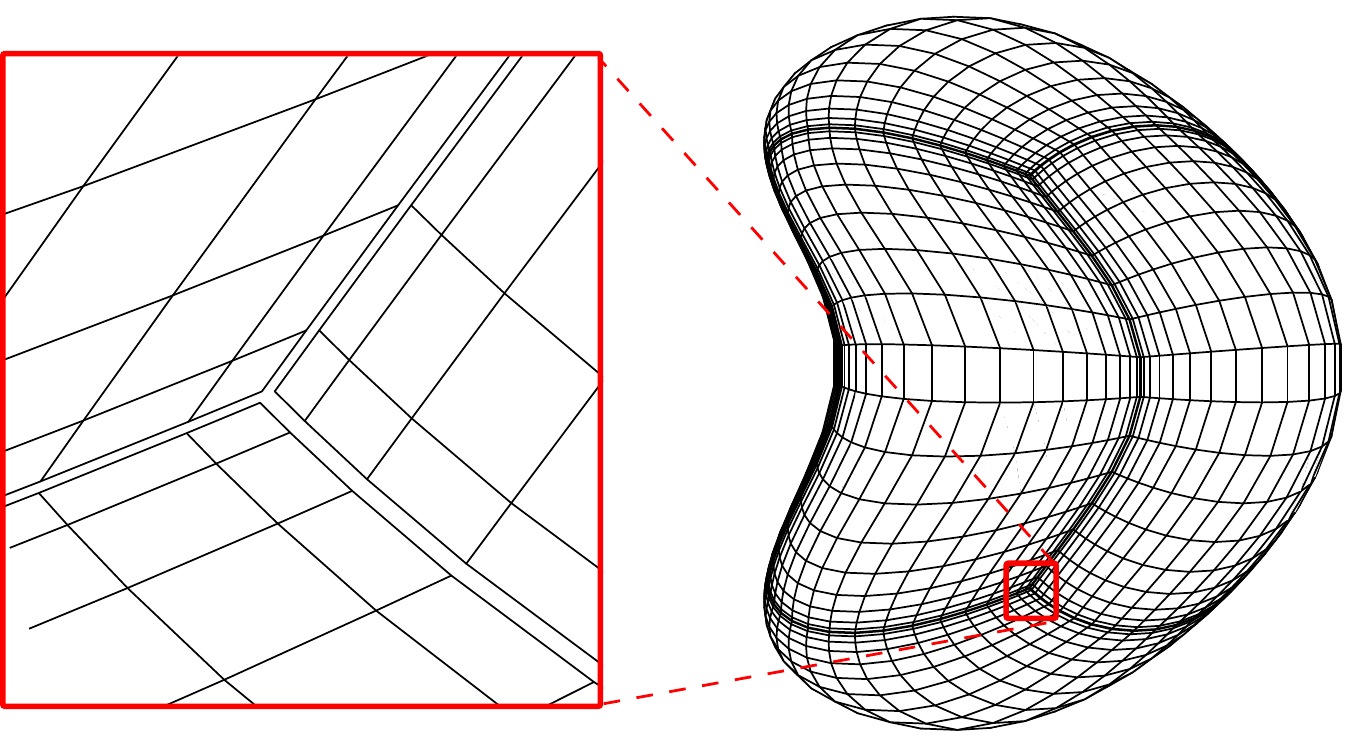}\label{fig_beam_disc}}\\ 
\caption{Example of a non-overlapping surface parametrization and its discretization using Chebyshev grids.}\label{fig:patches}
\end{figure}
 In order to evaluate accurately the integrals above, we employ \emph{open Chebyshev grids} in the parameter space $\mathcal{H}$ to collocate the functions $F\lf(\bnex^k(\bxi)\rg)$. Accordingly, numerical integration over the parameters space $\mathcal H$ is  carried out by means of the so-called  Fej\'er's first quadrature rules~\cite{davis2007methods}.  Specifically, the numerical value of the integral of functions $f:\mathcal H\to\R$ is then approximated by the quadrature rule
\begin{equation}
\int_{\mathcal H}f(\bxi)\de \bxi \approx \sum_{i=1}^N\sum_{j=1}^N f(t_i,t_j)\omega_{i}\omega_j,\label{eq:quad_rule}
\end{equation}
where $\mathcal H$ is discretized by the $N\times N$ tensor-product grid 
\begin{equation}
\lf(t_i, t_j\rg)\in\mathcal H=[-1,1]\times [-1,1],\qquad i,j=1,\ldots, N,\label{eq:2d_grid}
\end{equation}
where the quadrature points $t_j$ are the Chebyshev zero points 
\begin{equation}
t_j := \cos\lf(\vartheta_{j}\rg),\quad\vartheta_j := \frac{(2j-1)\pi}{2N},\quad j=1,\ldots,N,\label{eq:grid_points}
\end{equation}
and where the Fej\'er quadrature weights are given by
\begin{equation}
 \omega_j:= \frac{2}{N}\lf(1-2\sum_{\ell=1}^{[N/2]}\frac{1}{4\ell^2-1}\cos(2\ell\vartheta_{j})\rg),\quad j=1,\ldots,N.\label{eq:weights}
\end{equation}
  The quadrature weights~\eqref{eq:weights} can be efficiently computed  by means of the Fast Fourier Transform (FFT)~\cite{waldvogel2006fast}. The quadrature rule~\eqref{eq:quad_rule} yields spectral (super-algebraic) accuracy for integration of smooth $C^\infty(\mathcal H)$  functions. (For presentation simplicity we selected here  the same numbers $N$ of points to discretize both variables $\xi_1$ and $\xi_2$, but this need not necessarily be the case.)

 Another key feature of the proposed discretization scheme is that derivatives of smooth functions $f:\mathcal H\to\C$ can be computed with spectral accuracy by means of FFT algorithms. 
In detail, partial derivatives $\p^\alpha f$, $\alpha=(\alpha_1,\alpha_2)$, can be approximated on the tensor-product grid~\eqref{eq:2d_grid} from  the grid sample $\{f(t_i,t_j)\}_{i,j=1}^{N}$  as 
$$\p^\alpha f(t_i,t_j)\approx (-1)^{\alpha_1+\alpha_2}\sin(\vartheta_i)^{-\alpha_1}\sin(\vartheta_j)^{-\alpha_2}\lf(D^\alpha_{\rm FFT}F\rg)_{i,j},$$ 
where $D^\alpha_{\rm FFT}F$ corresponds to the numerical derivative of $F(\vartheta,\vartheta')=f(\cos\vartheta,\cos\vartheta')$ of order $\alpha_1$ (resp. $\alpha_2$) in the variable $\vartheta$ (resp. $\vartheta'$) on the grid $(\vartheta_i,\vartheta_j)$, $i,j=1,\ldots,N$. Owing to the fact the latter is an uniform grid, $D^\alpha_{\rm FFT}F$ can be computed from the discrete Fourier transform of $\{f(t_i,t_j)\}_{i,j=1}^N$  which can in turn be obtained by means of the FFT~\cite{johnson2011notes}. We thus conclude that all the partial derivatives of the coordinate maps $\bnex^k$ and unit normals $\bnor^k$ on the grid~\eqref{eq:2d_grid}---which are needed in the construction of the matrix $A$ corresponding to  the harmonic polynomial interpolation procedure~\eqref{eq:mat_entries}---are efficiently and accurately obtained by means of the Chebyshev-FFT differentiation procedure outlined above.  Similarly,  the partial derivatives of the density function $\varphi:\Gamma\to\R$ on a  patch $\mathcal P^k$---which are needed for the construction of the right-hand-side vectors~\eqref{eq:rhs_vecs}---are also evaluated by means of this procedure, which has to be applied in this case to the function $\phi^k(\bxi) = \varphi(\bnex^k(\bxi))$. A summary of the numerical procedure for the evaluation of the single-layer operator is presented below (completely analogous procedures can be followed for  evaluation of the double-layer, adjoint double-layer and hypersingular operators).

\begin{algorithm}[H]
 \KwData{Grids $\{\nex^k_{i,j}\}_{i,j=1}^{i,j=N}\subset\mathcal P^k$, $k=1,\ldots, N_p$, corresponding to the discretization of the surface $\Gamma$ using $N_p$ non-overlapping patches, generated using Chebyshev grids in the parameters space $\mathcal H$;  discrete density function $\varphi(\nex_{i,j}^k)= \phi_{i,j}^k$, $i,j=1,\ldots,N.$ $k=1,\ldots, N_p$.}
 \KwResult{Quantities $I_{i,j}^k$ corresponding to the approximate value of $S[\varphi]$ at the grid points $\nex_{i,j}^k$, $i,j=1,\ldots,N$, $k=1,\ldots,N_p$.}

 \For{$k$ from 1 to $N_p$}{
  compute approximate derivatives of the density functions $\varphi$ on the patch $\mathcal P^k$ using FFT-based spectral differentiation of the 2D array $\{\phi^k_{i,j}\}_{i,j=1}^{i,j=N}$\;
 } 
 set $I_{i,j}^k=0$ for $i,j=1,\ldots,N$ and $k=1,\ldots,N_p$;\\
  \For{each grid point $\nex_{i,j}^k$}{
	use approximate derivatives of $\varphi$  at $\nex_{i,j}^k$ to compute the coefficients $c_\ell^S(\nex_{i,j}^k)$, $\ell=1,\ldots,8$, of the interpolating harmonic polynomial $U_S$~\eqref{eq:4};\\
	\For{$m$ from 1 to $N_p$}{
	evaluate the approximate integral $I=\sum_{\nex_{p,q}^m\in \mathcal P^m} f(\nex_{p,q}^m)w^m_{p,q}\approx \int_{\mathcal P^m}f(\ney)\de s$ with $f(\ney)=G(\nex_{i,j}^k,\ney)\lf\{\varphi(\ney)-\p_nU_S(\ney,\nex_{i,j}^k)\rg\}+\frac{\p G(\nex_{i,j}^k,\ney)}{\p n(\ney)}U_S(\ney,\nex_{i,j}^k)$  using Fej\'er's quadrature rule\;
	$I_{i,j}^k\gets I_{i,j}^k + I$\;
	
	}
 }
 \caption{Numerical evaluation of the single-layer operator.\label{alg:sl_eval}}
\end{algorithm}

In our  first 3D numerical example we consider the bean shaped obstacle depicted in Figure~\ref{fig:patches}. The surface of the obstacle, whose exact definition is given in~\cite[Sec.~6.4]{Bruno:2001ima}, is parametrized using six non-overlapping quadrilateral patches, each discretized using the same number $N\times N$ of Chebyshev quadrature points. To test the accuracy of the proposed technique for the numerical evaluation of the four Laplace boundary integral operators, we consider the function $u(\nex) = 1/|\nex-\nex_0|-1/|\nex+\nex_0|$, $\nex_0=(2,2,2)$, which is harmonic in the interior of the bean obstacle. We then evaluate the error in the Green's formulas:
\begin{subequations}\begin{eqnarray}
-\frac{u(\nex)}{2} &=& K[u](\nex)-S[\p_n u](\nex)\qquad \mbox{(SL-DL)},\label{eq:single-double}\\
  -\frac{\p_n u(\nex)}{2} &=& N[u](\nex)-K'[\p_n u](\nex)\qquad \mbox{(ADL-HS)},\label{eq:adj-hyper}
\end{eqnarray}\end{subequations}
for $\nex\in\Gamma$, where all  four Laplace integral operators are numerically approximated by means of the proposed HDI technique combined with Chebyshev integration/differentiation. The errors (in the maximum norm) are displayed in Figure~\ref{convergence_bean}.  Two convergence regimes  can be distinguished. For small $N$ values the error seems to be dominated by the accuracy of the numerical differentiation algorithm, which exhibits super-algebraic convergence. For larger $N$ values, in turn,  the quadrature  errors become dominant and, as expected, they exhibit $O(N^{-3})$ and $O(N^{-2})$ convergence rates for the evaluation of formulae~\eqref{eq:single-double} (involving the single- and double-layer operator) and~\eqref{eq:adj-hyper}  (involving the adjoint double-layer and hypersingular operators), respectively.  The same example is then performed for a parallelepiped obstacle (featuring sharp edges and corners) in Figure~\ref{convergence_parallelepiped}, where the same convergence orders are observed. We note that the accuracies reported in Figure~\ref{convergence_bean} related to evaluations of the single and double-layer operators are similar to those achieved by the kernel regularization method of Beale et al.~\cite{2016CCoPh..20..733B}. While different in spirit, both methods lead to third-order convergence and require simple implementations. Also, both methods can in principle be pursued to higher orders in 3D, but at the expense of incorporation of higher-order derivatives and more complicated implementations. The main appeal of the kernel regularization method in~\cite{2016CCoPh..20..733B} is that it does not depend on surface parametrization. One possible advantage of the HDI method is the fact that is oblivious of the nature of the kernel singularity (it depends only on the algebraic order of that singularity) and as such is directly applicable to evaluations of all four BIO.
\begin{figure}[h!]\centering
 \subfloat[][Bean]{\includegraphics[scale=0.5]{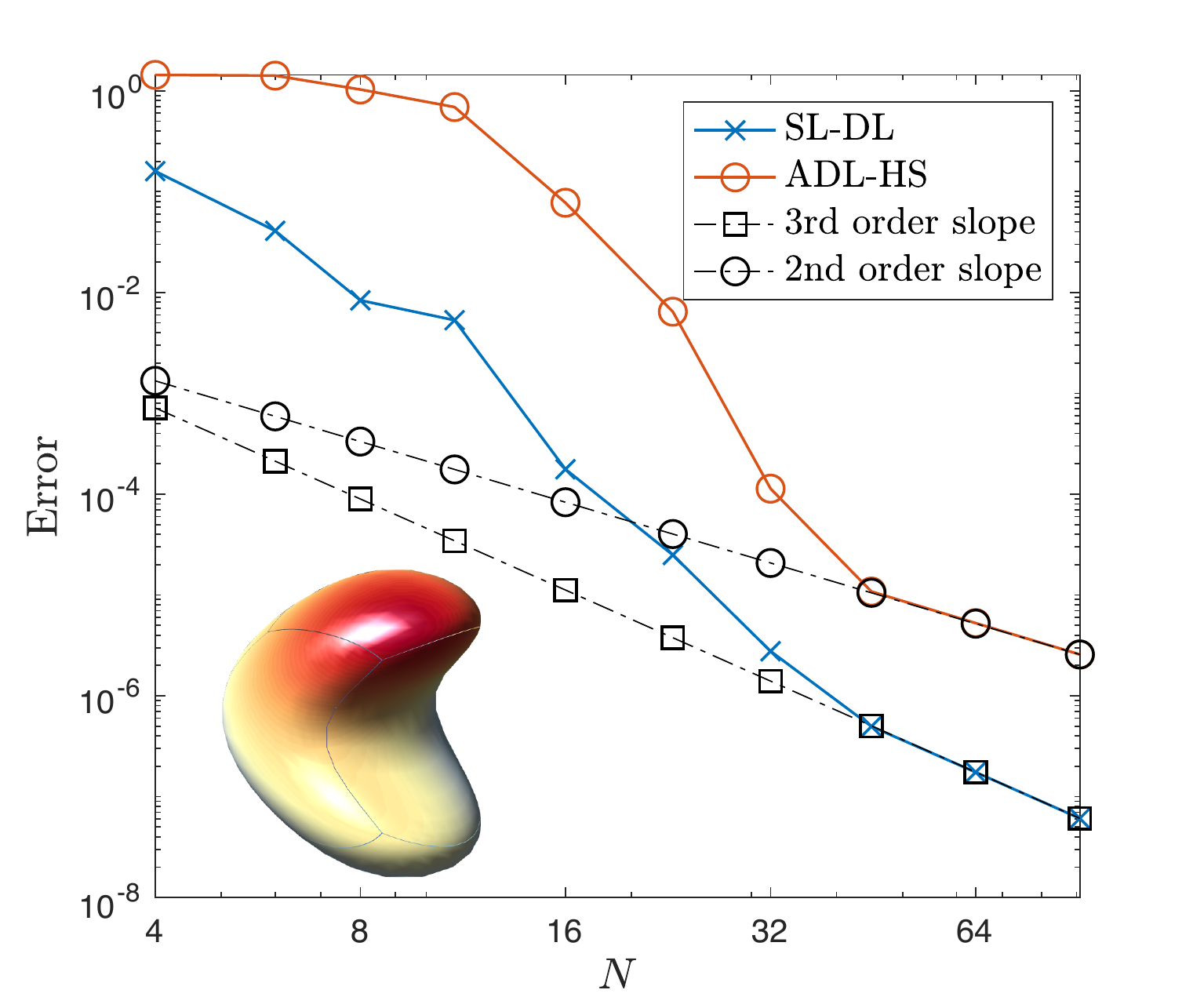}\label{convergence_bean}}\qquad
 \subfloat[][Parallelepiped]{\includegraphics[scale=0.5]{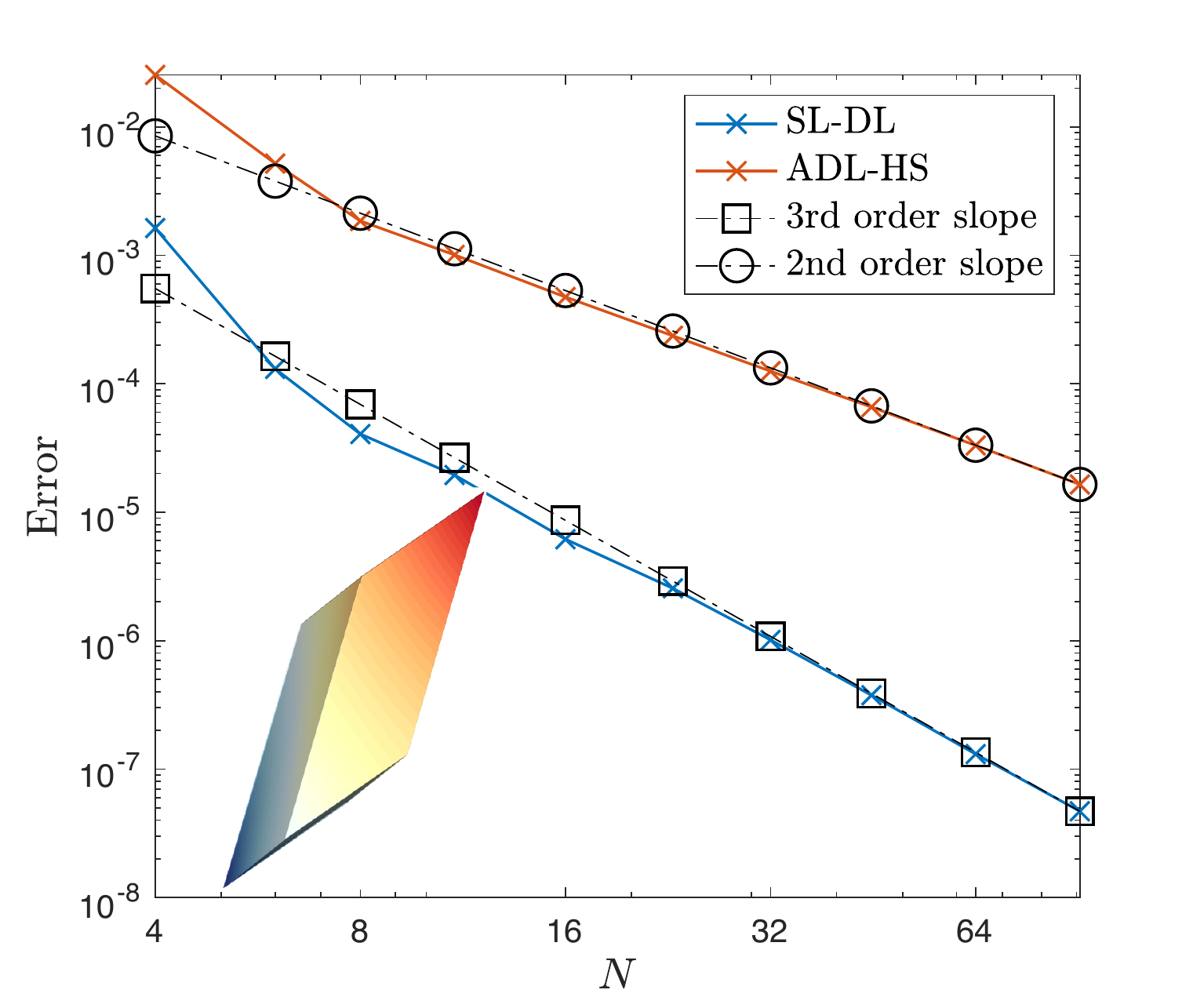}\label{convergence_parallelepiped}}
\caption{Relative errors (in the maximum norm) obtained in the evaluation of the Green's formulae~\eqref{eq:single-double} and~\eqref{eq:adj-hyper} for a bean-shaped (a) and parallelepiped (b) obstacles, using the 3D HDI technique. Here $N$ denotes the number of points per dimension  per patch; thus the total number of points used is $6N^2$ for each obstacle. The harmonic function  $u(\nex) = 1/|\nex-\nex_0|-1/|\nex+\nex_0|$, $\nex_0=(2,2,2)$, is plotted on the surface of each one of  the obstacles.}\label{fig:Green_convergence}
\end{figure}

Finally, in order to demonstrate the accuracy of the proposed technique in the evaluation of nearly singular integral in 3D, we consider a Neumann boundary value problem (BVP) posed in the exterior of two obstacles $\Omega_l$ (cushion) and $\Omega_r$ (sphere) displayed in Figure~\ref{fig:3d_near}(a),  touching at the point $(0,0,0)$. Once again a harmonic function $u(\nex) = 1/|\nex-\nex_l| + 1/|\nex-\nex_r|$ with  $\nex_l\in\Omega_l$  and  $\nex_r\in\Omega_r$ is used to assess the accuracy of the numerical solution. We thus consider the Laplace equation  $\Delta v = 0$ in $\R^3\setminus\{\Omega_l\cup\Omega_r\}$, with the Neumann boundary condition  $\p_n v = \p_n u$ on $\p\Omega_l\cup\p\Omega_r$, and the decay condition $v(\nex)\to 0$ as $|\nex|\to\infty$. Clearly, the exact (and unique) solution of the BVP is~$v=u$ in $\Omega_l\cup\Omega_r$. Using a direct formulation the BVP is posed as the following (uniquely solvable) second-kind integral equation
\begin{equation}
\lf(-\frac{I}{2}+K\rg)v = S[\p_n u]\quad\mbox{on}\quad\p\Omega_l\cup\p\Omega_r,\label{eq:2nd_N3D}
\end{equation}
whose exact solution is $u|_{\p\Omega_l\cup\Omega_r}$. The single- and double-layer operators  in~\eqref{eq:2nd_N3D}, which entail evaluation of both singular and nearly singular integrals, are discretized here utilizing the numerical procedure outlined above in this section. In particular, both surfaces $\p\Omega_l$ and $\p\Omega_r$ are parametrized by means of six non-overlapping patches with $20\times20$ Chebyshev points per patch in the case of the sphere ($\Omega_r$), and $30\times 30$ points in the case of cushion ($\Omega_l$). Nearly singular operators are evaluated using formulae~\eqref{eq:lay_pots}. The resulting discrete linear system is then solved  iteratively using GMRES~\cite{saad1986gmres}, which for this example required 24 iterations to achieve an error tolerance of $10^{-8}$. The integral equation solution achieved a maximum error of $3.77\times10^{-3}$ on the sphere and a maximum error of $1.29\times 10^{-3}$ on the cushion, for the discretization considered. 

The numerical solution of the BVP is then evaluated in both $X\!Z$ and $Y\!Z$ planes (that pass through the touching point). To demonstrate the effectiveness of the proposed HDI technique, the field $v=\mathcal D[v]-\mathcal S[\p_n v]$  is computed with and without taking care of the nearly singular integrands. The logarithm in base ten of the numerical errors in the $X\!Z$-plane (resp. $Y\!Z$-plane) are displayed in Figures~\ref{fig:3d_near}(b) and~\ref{fig:3d_near}(c) (resp. Figures~\ref{fig:3d_near}(d) and~\ref{fig:3d_near}(e)). In particular, the  error obtained at the touching point using the proposed technique is smaller than $6.78\times 10^{-4}$.
  \begin{figure}[h!]
\centering	
\includegraphics[scale=0.5]{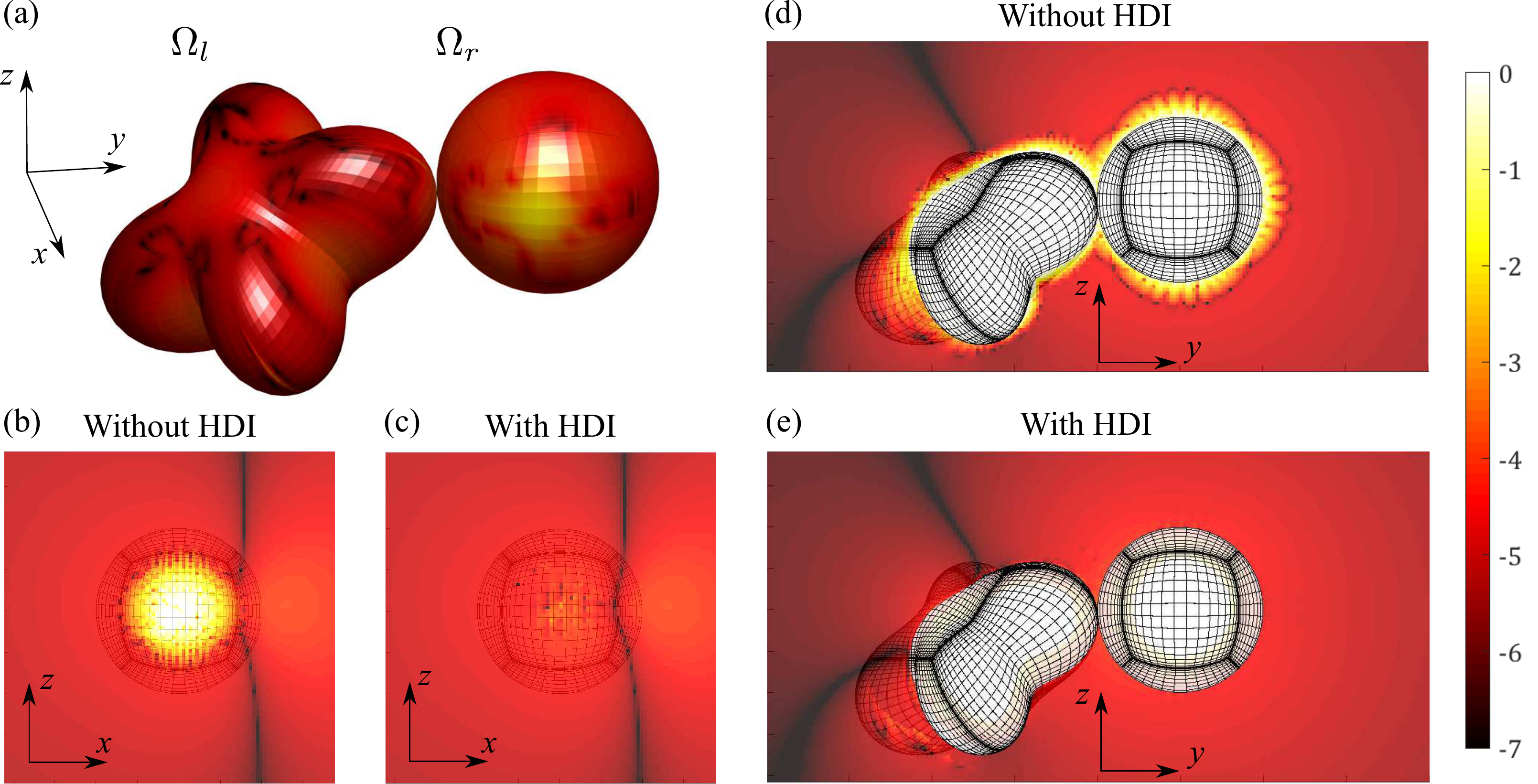}
\caption{Numerical errors in the solution of the Laplace equation in the exterior of two obstacles touching at the point $(0,0,0)\in\p\Omega_l\cap\p\Omega_r$. (a) $\log_{10}$ of the absolute error in the solution of the integral equation~\eqref{eq:2nd_N3D} obtained by means of the proposed  density interpolation technique. (b)-(d): $\log_{10}$ of the absolute errors in approximation of harmonic function $v=\mathcal D[v]-\mathcal S[\p_n v]$ without using any regularization of the nearly singular integrals. (c)-(e): $\log_{10}$ of the absolute errors in approximation of  $v$ using the HDI method. The maximum of the errors displayed in figures (c)~and~(e) is $6.78\times 10^{-4}$.}\label{fig:3d_near}
\end{figure}  

\section{Conclusions} 
We  presented a high-order kernel regularization method based on harmonic density interpolation (HDI) for the
numerical evaluation of integral operators and layer potentials of the
Laplace equation in two and three dimensions. The HDI
method was extended to the numerical evaluation of nearly singular
kernels that arise when considering target points near boundaries. The main advantage of the HDI technique is that it lends itself to straightforward  implementation of second and third order Nystr\"om methods  for evaluations of al four Laplace boundary integral operators as well as nearly singular layer potentials in three dimensions. Possible drawbacks of the HDI method are the need to evaluate simultaneously pairs of boundary integral operators (single-double layer operators and adjoint-double layer-hypersingular operators) as well as its reliance on high-order numerical differentiation that may lead to numerical instabilities for large values of the  interpolation orders. Integration of the HDI methods within a fast solver framework such as FMM will be reported in the near future. Extensions of the kernel regularization method to boundary integral equation approach to Stokes flow and linear elastostatic problems, as well as scattering problems for the Helmholtz, Maxwell and elastodynamics equations, are currently under investigation and will be presented elsewhere. Also, the numerical analysis of HDI methods is subject of ongoing investigation.

\section*{Acknowledgments}
Catalin Turc gratefully acknowledges support from NSF through contract
DMS-1614270. Luiz M. Faria gratefully acknowledges support from NSF through contract CMMI-1727565.

\appendix 
\section{Appendix: Invertibility of the 3D HDI matrix}\label{app:inv_mat}
This appendix is devoted to establishing the invertibility of the matrix $A$ defined~\eqref{eq:mat_entries} for the construction of the harmonic interpolant in 3D. To this end we first note that $A\in\R^{9\times 9}$ in~\eqref{eq:mat_entries} is a block lower triangular matrix of the form
$$
A=\begin{bmatrix}A_{1,1} & O\\ B & A_{2,2}\end{bmatrix},
$$ where $O\in\R^{4\times 5}$ is the zero matrix, $B\in\R^{5\times 4}$, $A_{1,1}\in\R^{4\times 4}$ and $A_{2,2}\in\R^{5\times 5}$. Therefore, it suffices to show that the diagonal blocks $A_{1,1}$ and $A_{2,2}$ are invertible.   

Since the matrix $A$ is  constructed by evaluating both the surface gradient and the Hessian of a certain linear combination of harmonic polynomials at a point $\nex\in\Gamma$, and none of these quantities depend on the surface parametrization, it is expected that---as in 2D---the invertibility of  $A$ is independent of the surface parametrization, provided it is regular, i.e., $\p_1\bnex\wedge \p_2\bnex \neq \bold 0$ at  every point $\nex=\bnex(\bxi)\in\Gamma$. Under this assumption we can then introduce a local re-parametrization around $\nex\in\Gamma$, which up to possibly rotations, takes the form $\tilde\bnex(u,v) = (u,v,f(u,v))$ where  $\tilde\bnex(0,0)=\nex$  with  $f$ being a smooth function. Using this simpler parametrization  and letting  $f_u:=\partial_uf$, $f_v:=\p_v f$ and $g=1+f_u^2+f_v^2$ (the determinant of the metric tensor on $\Gamma$), it is  easy to show that 
\begin{equation}
A_{1,1}=\begin{bmatrix}1&\bold 0\\
0&\p_1{\tilde \bnex}\\
0&\p_2{\tilde \bnex}\\
0&\tilde{\bold n}\end{bmatrix}=\begin{bmatrix}1&0 & 0 & 0\\0 & 1 & 0 & f_u\\0& 0 & 1 & f_v\\ 0 &-f_u/\sqrt{g} & -f_v/\sqrt{g} &1/\sqrt{g}\end{bmatrix}\quad \mbox{at}\quad (u,v)=(0,0),\label{eq:1st_block}\end{equation}
where the unit normal is given by $\tilde{\bold n}=(\p_1\tilde\bnex\wedge\p_2\tilde\bnex)/|\p_1\tilde\bnex\wedge\p_2\tilde\bnex|$. It thus follows from here that $\operatorname{det}(A_{1,1})=\sqrt{g}\neq 0$. In fact, using the identity in the middle of~\eqref{eq:1st_block}, it can be shown that the identity $\operatorname {det}(A_{1,1}) = |\p_1\bnex\wedge\p_2\bnex|$ holds true for any regular parametrization~$\bnex$.

Finally, we turn our attention to the block $A_{2,2}$ which using the parametrization $\tilde \bnex$ takes the form
\[
A_{2,2}=\begin{bmatrix} -f_v/\sqrt{g} &(1-f_u^2)/\sqrt{g}& -f_uf_v/\sqrt{g} & -2f_{u}/\sqrt{g} & -4f_u/\sqrt{g}\\
				-f_u/\sqrt{g} & -f_vf_u/\sqrt{g} & (1-f^2_v)/\sqrt{g} & 2f_{v}/\sqrt{g} &-2f_v/\sqrt{g}\\
				0 & 2f_u & 0 & 2 & 2(1-f_u^2)\\
					1 & f_v & f_u & 0 & -2f_uf_v\\ 
				0 & 0 & 2f_v& -2 &  -2f_v^2
\end{bmatrix}\quad \mbox{at}\quad (u,v)=(0,0).
\]
It turns out that $\operatorname{det}(A_{2,2})=-4g^2\neq 0
$. Therefore, we conclude that  $\operatorname{det}(A)=-4 g^{5/2}$ and thus $A$ is an invertible matrix.

\bibliographystyle{abbrv}
\bibliography{references}
\end{document}